\documentclass [leqno] {amsart}
  
\usepackage {amssymb, amsthm}

\usepackage[all]{xy}
\usepackage {hyperref} 
\usepackage{color}
\usepackage{enumerate}

\theoremstyle{definition}
\newtheorem{defn}{Definition}[section]
\newtheorem{notation}[defn]{Notation}

\theoremstyle{plain}
\newtheorem{introthm}{Theorem}

\newtheorem{prop}[defn]{Proposition}
\newtheorem{lem}[defn]{Lemma}
\newtheorem{cor}[defn]{Corollary}
\newtheorem{thm}[defn]{Theorem}

\theoremstyle{remark}
\newtheorem{rmk}[defn]{Remark}

\newtheorem*{introconvention}{Conventions}
\newtheorem{no}[subsubsection]{}

\newcommand{\eq}[2]{\begin{equation}\label{#1}#2 \end{equation}}
\newcommand{\ml}[2]{\begin{multline}\label{#1}#2 \end{multline}}
\newcommand{\mlnl}[1]{\begin{multline*}#1 \end{multline*}}

\newcommand{\arir}{\ar@{^{(}->}}
\newcommand{\aril}{\ar@{_{(}->}}
\newcommand{\are}{\ar@{>>}}

\newcommand{\xr}[1] {\xrightarrow{#1}}
\newcommand{\xl}[1] {\xleftarrow{#1}}
\newcommand{\lra}{\longrightarrow}
\newcommand{\inj}{\hookrightarrow}
\newcommand{\surj}{\twoheadrightarrow}

\newcommand{\ul}{\underline}
\newcommand{\ol}{\overline}

\newcommand{\codim}{{\rm codim}}

\newcommand{\Pic}{{\rm Pic}}

\newcommand{\Spec}{{\rm Spec \,}}
\newcommand{\Proj}{{\rm Proj\,}}

\newcommand{\Tr}{{\rm Tr}}

\newcommand{\gr}{\mathrm{gr}}
\newcommand{\Nm}{{\rm Nm }}

\newcommand{\id}{{\rm id}}

\newcommand{\CH}{\mathrm{CH}}
\newcommand{\dlog}{\mathrm{\,dlog\,}}
\newcommand{\Nis}{\mathrm{Nis}}
\newcommand{\Zar}{\mathrm{Zar}}
\newcommand{\et}{\text{\'et}}
\newcommand{\Sm}{\mathrm{Sm}}

\newcommand{\Frac}{\mathrm{Frac}}

\newcommand{\Ker}{{\rm Ker}}
\newcommand{\im}{{\rm Im}}

\newcommand{\Coker}{\mathop{\rm Coker}}

\newcommand{\sC}{{\mathcal C}}
\newcommand{\sD}{{\mathcal D}}

\newcommand{\sF}{{\mathcal F}}

\newcommand{\sH}{{\mathcal H}}
\newcommand{\sI}{{\mathcal I}}

\newcommand{\sK}{{\mathcal K}}

\newcommand{\sM}{{\mathcal M}}

\newcommand{\sO}{{\mathcal O}}

\newcommand{\sQ}{{\mathcal Q}}
\newcommand{\sR}{{\mathcal R}}

\newcommand{\sZ}{{\mathcal Z}}
\newcommand{\A}{{\mathbb A}}

\newcommand{\C}{{\mathbb C}}

\newcommand{\F}{{\mathbb F}}

\newcommand{\N}{{\mathbb N}}
\renewcommand{\P}{{\mathbb P}}
\newcommand{\Q}{{\mathbb Q}}

\newcommand{\W}{{\mathbb W}}

\newcommand{\Z}{\mathbb{Z}}

\newcommand{\fV}{\mathfrak{V}}

\newcommand{\fm}{\mathfrak{m}}
\newcommand{\fn}{\mathfrak{n}}

\newcommand{\e}{\epsilon}

\begin{document}
\title{Higher Chow groups with modulus and relative Milnor {$K$}-theory}

\begin{abstract}
Let $X$ be a smooth variety over a field $k$ and $D$ an effective divisor whose support has simple normal crossings.
We construct an explicit cycle map from the Nisnevich motivic complex $\Z(r)_{X|D,\Nis}$ of the pair $(X,D)$
to a shift of the relative Milnor $K$-sheaf $\sK^M_{r,X|D,\Nis}$ of $(X,D)$. We show that
this map induces an isomorphism $H^{i+r}_{\sM,\Nis}(X|D,\Z(r))\cong H^i(X_\Nis, \sK^M_{r, X|D,\Nis})$,
for all $i\ge \dim X$. This generalizes the well-known isomorphism in the case $D=0$. 
We use this to prove a certain Zariski descent property for the motivic cohomology
of the pair $(\A^1_k, (m+1)\{0\})$.
\end{abstract}

\author{Kay R\"ulling and Shuji Saito}

\address{Kay R\"ulling, Freie Universit\"at Berlin, Arnimallee 7, 14195 Berlin, Germany}
\email{kay.ruelling@fu-berlin.de}
\address{Shuji Saito, Interactive Research Center of Science, Graduate School of Science and Engineering, 
               Tokyo Institute of Technology, Ookayama, Meguro, Tokyo 152-8551, Japan}
\email{sshuji@msb.biglobe.ne.jp}

\thanks{The first  author was supported by the ERC Advanced Grant 226257 and the DFG Heisenberg Grant RU 1412/2-1.
            The second author is supported by JSPS Grant-in-aid (B) \#22340003. }

\maketitle

\tableofcontents

\section*{Introduction}
%\addcontentsline{toc}{chapter}{Appendices}

Recently several attempts have been made to introduce a theory of motivic cohomology with modulus.
The first attempt was due to S. Bloch and H. Esnault (\cite{BE}, \cite{BE2}) who introduced additive 
higher Chow groups of a field $k$. It is conceived as motivic cohomology for $k[t]/(t^2)$, or
an additive version of Bloch's higher Chow group for a pair $(\A^1_k,2\cdot \{0\})$ of the affine line $\A^1_k$ over $k$ 
with modulus $2\cdot \{0\}$, where $\{0\}$ denotes the origin of $\A^1_k$ regarded as a divisor.
They showed that the part given by zero-cycles of these groups coincide with the absolute differential forms
of $k$. The first author \cite{Kay} generalized this computation to the case $k[t]/(t^{m+1})$ for arbitrary 
$m\geq 1$ and proved that these groups give a cycle theoretic description of the big de Rham-Witt complex 
of Hesselholt-Madsen \cite{HeMa01} of $k$. 
Park \cite{Pa} extended the definition of Bloch-Esnault to introduce additive higher Chow groups 
${\rm TCH}^r(X,n;m)$ for a $k$-scheme $X$. The groups studied by Bloch-Esnault and R{\"u}lling correspond 
to the case $X=\Spec k$ and $r=n$. 
Motivated by a work \cite{KeS} of Kerz and the second author,
Park's definition is extended in \cite{BS14} to higher Chow groups $\CH^r(X|D,n)$ for a pair $(X,D)$ 
of an equidimensional $k$-scheme $X$ and an effective Cartier divisor $D$ on $X$. 
For $(X,D)=(Y\times \A^1_k, (m+1)\cdot(Y\times \{0\}))$, with $Y$ an equidimensional $k$-scheme and $m\ge 1$,  
we have 
\begin{equation}\label{TCHCH}
\CH^r(X|D,n)={\rm TCH}^r(Y, n+1;m).
\end{equation}
The definition of $\CH^r(X|D,n)$ is given by
\begin{equation}\label{CHZXD}
 \CH^r(X|D,n) = H_n(z^r(X|D,\bullet)),
\end{equation}
where $z^r(X|D,\bullet)$ is the \emph{cycle complex with modulus}, which is a subcomplex of the cubical version of 
Bloch's cycle complex $z^r(X,\bullet)$ consisting of those cycles satisfying a certain modulus condition.
In particular we have a natural map
\[ \CH^r(X|D,n)\to \CH^r(X,n),\]
where $\CH^r(X,n)=H_n(z^r(X,\bullet))$ is Bloch's higher Chow group (see \S\ref{ccmod} 
for other basic properties of $\CH^r(X|D,n)$). 
As in the case of Bloch's cycle complex, $z^r(X|D,\bullet)$ gives rise to a complex $z^r(-|D,\bullet)$ of 
\'etale sheaves on $X$. In case $X$ is smooth over $k$ we define the $r$-th motivic complex of $(X,D)$ to be
the following complex of Zariski sheaves on $X$
\[\Z(r)_{X|D}:=z^r(-|D, 2r-\bullet).\]
We denote by 
\[\Z(r)_{X|D,\Nis}\]
 the corresponding complex on $X_\Nis$.
The motivic cohomology of $(X,D)$ is by definition (see \cite[Def 2.10]{BS14})
\begin{equation}\label{HMXD}
H^i_{\sM}(X|D, \Z(r)):= H^i(X_\Zar, \Z(r)_{X|D}).
\end{equation}
If $D=0$ we get back Bloch's definition of the motivic complex
and motivic cohomology. We simply write $\Z(r)_X$ and $H^i_\sM(X,\Z(r))$ instead of
$\Z(r)_{X|0}$ and $H^i_\sM(X|0,\Z(r))$.
We define the motivic Nisnevich cohomology  of $(X,D)$ to be 
\[H^i_{\sM,\Nis}(X|D,\Z(r)):=H^i(X_\Nis, \Z(r)_{X|D}).\]
\medbreak

An important property of the classical motivic complex is the cycle map to the Milnor $K$-sheaf:
\begin{equation}\label{phiX}
 \phi^r_X : \tau_{\geq r} \Z(r)_X \to \sK^M_{r, X}[-r],
\end{equation}
which is a map in $\sD^b(X_\Zar)$ the derived category of bounded complexes of Zariski sheaves on $X$
(see \ref{MilnorK} for the definition of the Milnor $K$-sheaf $\sK^M_{r, X}$).
By the Gersten resolution for higher Chow groups, one knows that $\phi^r_X$ is actually an isomorphism.
In fact one can realize $\phi^r_X$ as an explicit morphism of  actual complexes from $\tau_{\geq r} \Z(r)_X$
to the Gersten complex of $\sK^M_{r,X}[-r]$.
This construction is well known to the experts but due to the lack of a reference, 
we include its review in \S\ref{reviewcyclemap}. 
The first main result of this paper is a construction of the relative version of \eqref{phiX}:
\begin{equation}\label{phiXD}
 \phi^r_{X|D} : \tau_{\geq r} \Z(r)_{X|D} \to \sK^M_{r, X|D}[-r],
\end{equation}
where $\sK^M_{X|D}$ is the relative Milnor $K$-sheaf for $(X,D)$, which is a subsheaf of $\sK^M_{r, X}$
(see Definition \ref{defn:MilnorModulus}).
Unfortunately we can construct it  as a morphism in $\sD^b(X_\Zar)$ only assuming $D_{\rm red}$ is smooth.
If $D_{\rm red}$ is a simple normal crossing divisor (SNCD) on $X$, 
we can construct a natural map in $\sD^b(X_\Nis)$:
\begin{equation}\label{phiXDNis}
\phi^r_{X|D,\Nis}: \tau_{\ge r}\Z(r)_{X|D,\Nis}\to\sK^M_{r, X|D,\Nis}[-r]
\end{equation}
fitting in the following commutative diagram 
\[\xymatrix{ \tau_{\ge r}\Z(r)_{X|D,\Nis}\ar[r]^-{\phi^r_{X|D,\Nis}}\ar[d] & \sK^M_{r, X|D,\Nis}[-r]\ar[d]\\
                   \tau_{\ge r}\Z(r)_{X,\Nis}\ar[r]^{\simeq}_{\phi^r_{X,\Nis}} & \sK^M_{r, X,\Nis}[-r],   }\]
where $\phi^r_{X,\Nis}$ is the Nisnevich sheafication of $\phi^r_{X}$.
In fact we show that the Cousin complex of $\sK^M_{r, X|D,\Nis}$ is a resolution (see Corollary \ref{cor:CousinComplex-relK-Nis})
and we can realize $\phi^r_{X|D,\Nis}$ as an explicit morphism of complexes from
$\tau_{\ge r}\Z(r)_{X|D,\Nis}$ to  the Cousin complex of $\sK^M_{r, X|D,\Nis}[-r]$.
The inclusion $\sK^M_{r,X|D,\Nis}\inj \sK^M_{r,X,\Nis}$ induces a natural map from the Cousin complex of
$\sK^M_{r,X|D,\Nis}$ to the Gersten complex of $\sK^M_{r,X,\Nis}$ and there is a diagram of morphisms between
actual complexes underlying the above diagram.
We will prove the following.

\begin{introthm}[Theorem \ref{thm:rel-cycle-iso}]\label{thm:rel-cycle-iso.intro} 
Let $X$ be a smooth equidimensional scheme of dimension $d=\dim X$ and $D$ an effective
divisor such that $D_{\rm red}$ is a simple normal crossing divisor.  
Then:
\begin{enumerate}[(i)]
\item $H^i_{\sM}(X|D,\Z(r))=0=H^i_{\sM,\Nis}(X|D,\Z(r))$ for $i>d+r$.
\item The cycle map $\Z(r)_{X|D,\Nis}\to \tau_{\ge r}\Z(r)_{X|D,\Nis}
                           \xr{\phi^r_{X|D,\Nis}} \sK^M_{r,X|D,\Nis}[-r]$
        induces an isomorphism
        \[\phi^{d,r}_{X|D,\Nis}: H^{d+r}_{\sM,\Nis}(X|D,\Z(r))\xr{\simeq} H^d(X_{\Nis}, \sK^M_{r,X|D,\Nis}).\]
      If moreover $D_{\rm red}$ is smooth, then all maps in the following commutative diagram are isomorphisms
     \[\xymatrix{ H^{d+r}_{\sM}(X|D,\Z(r))\ar[r]^{\simeq}
                                                                 \ar[d]_{\phi^{d,r}_{X|D}}^\simeq & 
                       H^{d+r}_{\sM,\Nis}(X|D,\Z(r))\ar[d]^{\phi^{d,r}_{X|D,\Nis}}_\simeq\\
                      H^d(X_\Zar, \sK^M_{r,X|D})\ar[r]^-{\simeq} & 
                                                    H^d(X_\Nis, \sK^M_{r,X|D,\Nis}).
                         }\]
\end{enumerate}
\end{introthm}
\medbreak

As an application of Theorem \ref{thm:rel-cycle-iso.intro}, we will prove the Zariski descent property
for additive higher Chow groups (see Theorem \ref{thm:zariski-descent.intro} below).
From \eqref{CHZXD} and \eqref{HMXD} we have a natural map
\begin{equation}\label{CHHMXD}
\CH^r(X|D,n)\to H^{2r-n}_{\sM}(X|D, \Z(r)).
\end{equation}
In the classical case where $D=0$, \eqref{CHHMXD} is an isomorphism, which is known as 
the Zariski descent property for Bloch's higher Chow groups.
It is a consequence of the Mayer-Vietoris property of Bloch's cycle complex $z^r(X,\bullet)$ which follows from 
the localization theorem for the complex.
In case $D\not=0$,  it is not clear whether $z^r(X|D,\bullet)$ has some reasonable localization property at all.
On the other hand, the higher dimensional class field theory with wild ramification suggests that
the Nisnevich descent property for higher Chow groups with modulus is related to some deep arithmetic questions.
As a consequence of \cite[Theorem III]{KeS}, we have the following result.
\renewcommand{\lim}{\operatornamewithlimits{\varprojlim}}
\newcommand\piab{\pi_1^{\rm ab}}

\begin{introthm}\label{thm:kerz-saito} 
Let $X$ be a smooth projective variety of dimension $d$ over a finite field $k$ and 
$U\subset X$ be the open complement of a SNCD on $X$. 
Write $\CH^d(X|D)=\CH^d(X|D,n)$ for $n=0$. Then the natural map 
\begin{equation}\label{kerzsaito}
\underset{D}{\lim}\; \CH^d(X|D)\to \underset{D}{\lim}\; H^{2d}_{\sM,\Nis}(X|D, \Z(d))
\end{equation}
is an isomorphism, where the limit is taken over all effective divisors $D$ on $X$ supported on $X-U$. 
\end{introthm}

Indeed, by \cite[Theorem 3.3]{BS14} the group $\CH^d(X|D)$ is equal to the Chow group of zero-cycles 
with modulus denoted by $C(X,D)$ in \cite{KeS}, 
and by Theorem \ref{thm:rel-cycle-iso.intro}, the group $H^{2d}_{\sM,\Nis}(X|D, \Z(d))$ is isomorphic to $H^d_{\Nis}(X,\sK^M_{d,X|D,\Nis})$ which is the id\`ele class group used in the class field theory of  Kato-Saito \cite{KS}.
We have a commutative diagram
\[\xymatrix{
\underset{D}{\lim}\;C(X,D) \ar[r] \ar[rd]_{\rho^{KeS}_{U}}& \underset{D}{\lim}\;H^d_{\Nis}(X,\sK^M_{d,X|D,\Nis}) \ar[d]^{\rho^{KaS}_{U}}\\
& \piab(U)\\
}\]
where $\piab(U)$ is the abelian fundamental group of $U$ and $\rho^{KeS}_{U}$ (resp. $\rho^{KaS}_{U}$)
is the reciprocity map from \cite{KeS} (resp. \cite{KS}). The reciprocity maps $\rho^{KeS}_{U}$ and $\rho^{KaS}_{U}$ were
shown to be bijections onto the subgroup of $\piab(U)$ consisting of those elements,
 whose images in $\piab(\Spec k)$ are integral powers of the Frobenius substitution of $k$.
These clearly imply Theorem \ref{thm:kerz-saito}. On the other hand, one can deduce \cite[Theorem III]{KeS} from 
the Kato-Saito class field theory assuming \eqref{kerzsaito} is an isomorphism. 
\medbreak

Using Theorem \ref{thm:rel-cycle-iso.intro} one can find examples, where the map \eqref{CHHMXD} or its Nisnevich version
\begin{equation}\label{CHHMXDNis}
\CH^r(X|D,n)\to H^{2r-n}_{\sM,\Nis}(X|D, \Z(r))
\end{equation}
are not isomorphisms, see Remark \ref{rmk:not-iso}. These examples however  arise from the fact
that, if  $r<\dim X$ and $n=r-\dim X$, the right-hand side of the above map does not necessarily vanish.
At this moment we don't have any definitive idea on what to expect for $n\ge 0$. 
In this paper we only present the following special case. 

\begin{introthm}[Theorem \ref{thm:add-Zar-descent}]\label{thm:zariski-descent.intro} 
Let $k$ be a field of characteristic $\neq 2$. 
The natural maps 
\begin{equation}\label{zariski-descent.intro} 
\CH^r(\A^1_k|(m+1)\{0\}, r-n)\xr{\simeq} H^{r+n}_\sM(\A^1_k|(m+1)\{0\}, \Z(r)),\quad  n \ge 1,
\end{equation}
are isomorphisms.
\end{introthm}

Notice that the right hand side of \eqref{zariski-descent.intro} is isomorphic to the Nisnevich motivic cohomology
$H^{r+n}_{\sM,\Nis}(\A^1_k|(m+1)\{0\}, \Z(r))$ by Theorem \ref{thm:rel-cycle-iso.intro}.
Notice also that the left hand side of \eqref{zariski-descent.intro} is ${\rm TCH}^r(k, r-n+1;m)$ 
(cf. \eqref{TCHCH}) which is clearly zero 
for $n\ge 2$ and the right hand side is zero by Theorem \ref{thm:rel-cycle-iso.intro}.
We prove the isomorphism \eqref{zariski-descent.intro} for $n=1$ by constructing the following commutative diagram 
for all $r,m\ge 1$
\[\xymatrix{ 
 \CH^r(\A^1_k|(m+1)\{0\}, r-1)\ar[d]^\simeq_{\alpha}
                                           \ar[r]^-{\eqref{zariski-descent.intro}} & 
                 H^{r+1}_{\sM}(\A^1_k|(m+1)\{0\},\Z(r))\ar[d]_\simeq^{\beta}\\
              \W_m\Omega^{r-1}_k\ar[r]_-\simeq^-{\gamma} & 
                     U^1K^M_r(k((T)))/U^{m+1}K^M_r(k((T))).
            }\]
where $\alpha$ is up to sign the isomorphism from \cite{Kay}, $\beta$ is an isomorphism deduced 
by using Theorem \ref{thm:rel-cycle-iso.intro}, and $\gamma$ is an isomorphism following from 
a comparison isomorphism between the big de Rham-Witt sheaves and relative Milnor $K$-sheaves established in 
Theorem \ref{thm-DRW-MilnorK}, which is reminiscent of Bloch's original construction of the $p$-typical
de Rham-Witt complex in \cite{BlochDRW}.

The following theorem was suggested by the referee.
\begin{introthm}[Theorem \ref{thm:Anvan}]\label{introthm4}
Let $k$ be a field and $X$ a smooth equidimensional $k$-scheme of dimension $d$, $D$ an effective Cartier divisor on $X$
 such that $D_{\rm red}$ is a simple normal crossing divisor. 
Then we have, for all $n\ge 2$ and all non-negative integers $m_1,\ldots, m_n$, 
                    \[H^{d+r+n}_{\sM,\Nis}(X\times_k \A^n_k| (p^*D+\sum_{i=1}^n m_i \cdot q_i^*\{0\}), \Z(r))=0,\]
where $q_i:X\times_k \A^n_k\to \A^1_k$ denotes the projection to the $i$-th factor of $\A^n_k$ and 
$p:X\times_k \A^n_k\to X$ is the projection.
\end{introthm}
This gives another example, where \eqref{CHHMXDNis} is an isomorphism, since the
vanishing $\CH^r(X\times_k \A^n_k|(p^*D+\sum_{i=1}^n m_i \cdot q_i^*\{0\}), r-(d+n))=0$,  for $n\ge 2$, was 
proven in \cite[Thm 5.11]{KP15}.

\subsection*{Acknowledgements}
The second author wishes to thank heartily Moritz Kerz, Lars Hesselholt and Federico Binda for inspiring discussions. 
He is also very grateful to the department of mathematics of the university of Regensburg
for the financial support via the SFB 1085 ``Higher Invariants'' (Regensburg).
The first author thanks Takao Yamazaki for discussions around Remark \ref{rmk:not-iso}.
The authors thank the referee for suggesting Theorem \ref{introthm4}.

\addtocontents{toc}{\protect\setcounter{tocdepth}{2}}

\begin{introconvention}\label{convention}
A $k$-scheme is a separated  scheme of finite type over a field $k$.
A {\em simple normal crossing divisor} (SNCD) on a smooth $k$-scheme $X$ is by definition a
reduced effective Cartier divisor $E$ on $X$ such that, if $E_1,\ldots, E_n$ are the irreducible
components of $E$, then the intersections $E_{i_1}\cap\cdots\cap E_{i_r}$ are {\em smooth} over $k$
and have codimension $r$ in $X$, for all $r\in [1,n]$ and $(i_1,\ldots, i_r)\in [1,n]^r$.
\end{introconvention}

\section{Cycle complex with modulus}\label{ccmod}\setcounter{subsection}{1}
\numberwithin{equation}{defn}
We recall the definition of Chow groups with modulus from \cite[2.]{BS14}.
In this section $k$ is a field, $X$ an equidimensional $k$-scheme and $D$ an effective Cartier divisor on $X$ 
with complement $U=X\setminus |D|$.

\begin{no}\label{cubes}
Set $\P^1=\Proj k[Y_0,Y_1]$ and let $y=Y_1/Y_0$ be the standard coordinate function on $\P^1$.
We set 
\[\square=\P^1\setminus\{1\},\quad \square ^n=(\P^1\setminus\{1\})^n, \quad n>1.\]
By convention we set $\square^0=\Spec k$.  Let $q_i: (\P^1)^n\to \P^1$ 
be the projection onto the $i$-th factor. We use the coordinate system $(y_1,\ldots, y_n)$ on $(\P^1)^n$ with 
$y_i=y\circ q_i$. Let $F_i^n\subset(\P^1)^n$ be the Cartier divisor defined by $\{y_i=1\}$ and 
put $F_n=\sum_{i=1}^n F^n_i$.

A {\em face} of $\square^n$ is a subscheme $F$ defined by equations of the form 
\[y_{i_1}=\epsilon_1,\ldots, y_{i_r}=\epsilon_r,\quad 
                  r\in [1,n],\, (i_1,\ldots, i_r)\in [1,n]^r, \epsilon_{i_j}\in \{0,\infty\}.\]
We denote by $\imath_F: F\inj \square^n$ the closed immersion. For $\epsilon=0,\infty$  and $i\in [1,n]$, let 
\[\imath^n_{i,\epsilon}: \square^{n-1}\inj \square^n\]
be the inclusion of the face of codimension 1 given by $y_i=\epsilon$.
\end{no}

\begin{defn}\label{defn:modulus-cycle}
 For $r,n\ge 0$ we denote by $C^r(X|D,n)$ the set of all
integral closed subschemes $Z\subset U\times\square^n$ of codimension $r$ which satisfy the following conditions:
\begin{enumerate}
\item $Z$ intersects  $U\times F$ properly for all faces $F\subset \square^n$.
\item The case  $n=0$: The closure of $Z$ in $X$ does not meet $D$.
\item The case $n\ge 1$: Denote by $\ol{Z}\subset X\times(\P^1)^n$ the closure of $Z$ and 
             by $\nu_{\ol{Z}}: \tilde{Z}\to X\times (\P^1)^n $ the composition 
                of the normalization $\tilde{Z}\to \ol{Z}$ 
           followed by the closed immersion $\ol{Z}\inj X\times (\P^1)^n$.
          Then the following inequality between Cartier divisors  holds:
                   \eq{defn:modulus-cycle1}{\nu_{\ol{Z}}^*(D\times(\P^1)^n)\le \nu_{\ol{Z}}^*(X\times F_n).}
\end{enumerate}
An element of $C^r(X|D,n)$ is called an {\em integral relative cycle of codimension $r$ for $(X,D)$}.
\end{defn}

\begin{lem}\label{lem:mod-cond-closed-subschemes}
Let $Z'\subset Z$ be integral closed subschemes in $X\times (\P^1)^n$ intersecting
the Cartier divisors $D\times (\P^1)^n$ and $X\times F_n$ properly. Let
 $\nu_{Z}: \tilde{Z}\to X\times (\P^1)^n$ be the composition of the normalization $\tilde{Z}\to Z$ with the
natural inclusion $Z\inj X\times (\P^1)^n$ and similar with $\nu_{Z'}: Z'\to X\times (\P^1)^n$.
Then the inequality $\nu_Z^*(D\times (\P^1)^n)\le \nu_Z^*(X\times F_n)$ implies 
the corresponding inequality with $\nu_Z$ replaced by $\nu_{Z'}$.
\end{lem}
\begin{proof}
This is essentially \cite[Prop 2.4]{KP}, see \cite[Lem 2.1]{BS14} for the version we need here.
\end{proof}

\begin{no}\label{cubical-modulus-cycles}
{\em Chow groups with modulus.}
Denote by $\ul{z}^r(X|D,n)$ the free abelian group on the set $C^r(X|D,n)$.
By Lemma \ref{lem:mod-cond-closed-subschemes} there is a well-defined pullback map
$(\id_X\times\imath_F)^*: \ul{z}^r(X|D,n)\to \ul{z}^r(X|D,m)$ for any
$m$-dimensional face $\square^m\cong F\subset \square^n$. 
We obtain a cubical object of abelian groups (see e.g. \cite[1.1]{Le09}):
\[\ul{n}\mapsto \ul{z}^r(X|D,n) \quad (\ul{n}=\{0,\infty\}^n, n=0,1,2,3,\ldots).\]
For each $n$ we have the subgroup $\ul{z}^r(X|D,n)_{\rm degn}$ of degenerate cycles,
i.e. those cycles which come from $\ul{z}^r(X|D,n-1)$ via pullback along 
one of the $n$ projections $U\times \square^{n}\to U\times \square^{n-1}$.
We set 
\[z^r(X|D,n):=\frac{\ul{z}^r(X|D,n)}{\ul{z}^r(X|D,n)_{\rm degn}}.\]
The $n$-th boundary operator $\partial: z^r(X|D,n)\to z^r(X|D,n-1)$ is given by 
\[\partial= \sum_{i=1}^n (-1)^i (\partial^\infty_{i}-\partial^0_i),\]
where $\partial^\epsilon_i= (\id_X \times \imath^n_{i,\epsilon})^*: z^r(X|D,n)\to z^r(X|D,n-1)$
is the pullback along the face $\{y_i=\epsilon\}$.
We get a complex $z^r(X|D,\bullet)$, which is the complex associated to the cubical object 
$\ul{n}\mapsto \ul{z}^r(X|D,n)$.
The higher Chow groups of $(X,D)$ are defined to be 
\[\CH^r(X|D,n):= H_n(z^r(X|D,n)), \quad n,r\ge 0,\]
see \cite[Def 2.5]{BS14}.  
If $D=0$ we get back Bloch's classical definition of the cycle complex and 
higher Chow groups. We simply write $z^r(X,\bullet)$ and $\CH^r(X,n)$ instead of 
$z^r(X|0,\bullet)$ and $\CH^r(X|0,n)$, respectively. 
\end{no}

\begin{no}\label{motivic-coh-mod}
{\em Motivic cohomology with modulus}. 
For an \'etale map $V\to X$ we denote by $D_V$ the pullback of $D$ to $V$. 
Then the presheaves 
\[z^r(-|D,n): X_{\et}\ni (V\to X)\mapsto z^r(V|D_V,n)\]
are sheaves for the \'etale topology, {\em a fortiori} for the Zariski and the Nisnevich topology. 
In case $X$ is smooth over $k$ the $r$-th motivic complex of $(X,D)$ is defined to be
the complex of Zariski sheaves on $X$
\[\Z(r)_{X|D}:=z^r(-|D, 2r-\bullet).\]
We denote by 
\[\Z(r)_{X|D,\Nis}\]
 the corresponding complex on $X_\Nis$.
The motivic cohomology $(X,D)$ is by definition
\[H^i_{\sM}(X|D, \Z(r)):= H^i(X_\Zar, \Z(r)_{X|D}),\]
see \cite[Def 2.10]{BS14}. If $D=0$ we get back Bloch's definition of the motivic complex
and motivic cohomology. We simply write $\Z(r)_X$ and $H^i_\sM(X,\Z(r))$ instead of
$\Z(r)_{X|0}$ and $H^i_\sM(X|0,\Z(r))$.
We define the motivic Nisnevich cohomology  of $(X,D)$ to be 
\[H^i_{\sM,\Nis}(X|D,\Z(r)):=H^i(X_\Nis, \Z(r)_{X|D}).\]
\end{no}

\begin{no}\label{Chow-mod-properties}
We give a list of properties and results:
\begin{enumerate}
\item\label{Chow-mod-properties1} 
       The modulus condition \eqref{defn:modulus-cycle1} implies that any $Z\in C^r(X|D,n)$
          is already closed in $X\times \square^n$. Therefore there is a natural map 
               \[\CH^r(X|D, n)\to \CH^r(X,n),\]
              where the right hand side is (the cubical version of) Bloch's higher Chow groups.
\item The above definition generalizes the additive higher Chow groups 
        defined by Bloch-Esnault and Park.
          In the case $(X,D)=(Y\times \A^1_k, (m+1)\cdot(Y\times \{0\}))$, 
       with $Y$ an equidimensional $k$-scheme and $m\ge 1$,  we have 
                    \[\CH^r(X|D,n)={\rm TCH}^r(Y, n+1;m).\]
\item There is a natural isomorphism 
           \[\CH^r(X|D,0)\xr{\simeq} \CH^r(X|D),\]
          where the right hand side is the group of $r$-codimensional cycles on $U$ modulo
           ``rational equivalence with modulus $D$", see \cite[3.]{BS14}.
\item Assume $X$ is normal. Then there is a natural quasi-isomorphism 
           \[\Z_{X|D}(1)\cong \Ker(\sO_X^\times\to\sO_D^\times)[-1],\]
          see \cite[1.5, Thm 4.3]{BS14}.
\item If $X$ is smooth and $D_{\rm red}$ is a simple normal crossing divisor, then there is
          a cycle map $\phi_{X|D}: \Z(r)_{X|D}\to \Omega^{\ge r}_{X/\Z}(\log D)\otimes_{\sO_X} \sO_X(-D)$
          in the derived category $\sD^-(X)$ of bounded above complexes of Zariski sheaves, see \cite[7.3]{BS14}.
          For $k=\C$, there are regulator maps  from the motivic cohomology of $(X,D)$ to a relative version
          of Deligne cohomology, Betti cohomology and a relative Abel-Jacobi map, see \cite[8., 9.]{BS14}.
\end{enumerate}
\end{no}

\section{Relative Milnor {$K$}-sheaves}
In this section $k$ is a field and $X$ a smooth connected $k$-scheme. 
We denote by $X^{(c)}$ the set of codimension $c$ points in $X$ and by $\eta$ the generic point of $X$ .

\subsection{The Gersten complex of Milnor {$K$}-sheaves}
\begin{no}\label{MilnorK}
For $r\in\Z$ we denote by $\sK^M_{r,X}$ the $r$-th Milnor $K$-sheaf on $X$.
By definition it is the Zariski sheaf which on an open $V\subset X$ is given by
\[\sK^M_{r,X}(V)= \Ker(K^M_{r}(k(\eta))\xr{\oplus\partial_x} 
                                     \bigoplus_{x\in X^{(1)}\cap V} K^M_{r-1}(k(x))),\]
where $\partial_x: K^M_{r}(k(\eta))\to K^M_{r-1}(k(x))$ denotes the tame symbol from \cite[\S 4]{BT}.
In particular $\sK^M_{r,X}=0$ for $r< 0$, $\sK^M_{0,X}=\Z$ and $\sK^M_{1,X}=\sO_X^\times$.
There is a canonical resolution $\sK^M_{r,X}\to C^\bullet_{r,X}$ by flasque sheaves called the {\em Gersten resolution}
(see e.g. \cite[Thm 6.1]{Rost})
\eq{Gersten}{0\to \sK^M_{r,X}\to \underbrace{i_{\eta*}K^M_{r}(k(\eta))\to 
     \bigoplus_{x\in X^{(1)}}i_{x*}K^M_{r-1}(k(x))\to 
                 \bigoplus_{x\in X^{(2)}} i_{x*}K^M_{r-2}(k(x))\to \ldots}_{:= C^\bullet_{r,X}},}
where $i_{x}: x\to X$ denotes the inclusion.

By \cite[Prop 10, (8) and Thm 13]{Kerz10}
the stalk of $\sK^M_{r, X}$ at $x\in X$ is the subgroup of 
$K^M_{r}(k(\eta))$ generated by symbols of the form $\{a_1,\ldots, a_r\}$, $a_i\in \sO_{X,x}^\times$, i.e.
\eq{description-MilnorK}{\sK^M_{r,X,x}=\{\sO_{X,x}^\times,\ldots, \sO_{X,x}^\times\}\subset K^M_r(k(\eta)).}
If $k$ is an infinite field, then by \cite[Thm 1.3 and Def 2.1]{Kerz09} 
\eq{MilnorK-relations}{\sK^M_{r,X}= (\sO^\times)^{\otimes_\Z r}/\sR,}
where $\sR\subset (\sO^\times)^{\otimes_\Z r}$ is the subsheaf of abelian groups which 
is generated by local sections of the form
$b_1\otimes\ldots \otimes b_{i-1}\otimes a\otimes (1-a)\otimes b_{i+2}\otimes\ldots \otimes b_r$,
where $a\in \sO_X^\times$ with $1-a\in \sO_X^\times$ and $b_i\in \sO_X^\times$.

In case $X$ is not connected and $X=\sqcup_j X_j$ is its decomposition into connected components 
with corresponding immersions $i_j: X_j\inj X$, we set
\[\sK^M_{r,X}:=\bigoplus_i i_{j*}\sK^M_{r,X_i}.\]
\end{no}

\begin{no}\label{K-Nis}
By \cite[Prop 10, (11)]{Kerz10} we have an isomorphism of Zariski sheaves
\[\sK^M_{r, X}\cong \sH^r(\Z(r)_X).\]
In particular $Y\mapsto H^0(Y,\sK^M_{r, Y})$ defines a homotopy invariant presheaf with transfers
on the category of smooth $k$-schemes in the sense of \cite[3.]{VoDM}. 
Hence  by \cite[Thm 3.1.12]{VoDM} it restricts to a sheaf on the Nisnevich site $X_{\Nis}$,
which we continue to denote by $\sK^M_{r,X}$ or if we want to stress that we are on $X_\Nis$
by $\sK^M_{r,X,\Nis}$.
\end{no}

\subsection{Milnor $K$-sheaf of a complement of an SNCD}
\begin{no}\label{cech-symbol}
Let $Y$ be a scheme and $\sF$ a sheaf of abelian groups on $Y$.
Let $Z\subset Y$  be a closed subscheme and denote by $j: V=Y\setminus Z\inj Y$ the inclusion
of the complement. We denote by $\ul{\Gamma}_Z(\sF)$ the sheaf on $Y$ of sections of $\sF$
with supports in $Z$ and by $\sH^i_Z(\sF)= R^i\ul{\Gamma}_Z(\sF)$ the $i$-th cohomology
sheaf with support in $Z$.  
For a scheme point $y\in Y$, we also define
\[H^i_y(\sF):=\sH^i_{\ol{y}}(\sF)_y=\varinjlim_{y\in U} H^i_{\ol{y}\cap U}(U, \sF),\]
where $\ol{y}$ denotes the closure of $y$ in $Y$ and the limit is over all open neighborhoods of $y$.
We have isomorphisms
\eq{cech-symbol1}{\frac{j_*(\sF_{|V})}{\sF}\cong \sH^1_Z(\sF)\quad \text{and}\quad
                            R^{i-1}j_*(\sF_{|V})\cong \sH^i_Z(\sF),\quad \text{for }i\ge 2.}
Assume that the ideal sheaf $\sI$ of $Z\subset X$ is generated
by a regular sequence of global sections $t_1,\ldots, t_c\in \Gamma(Y,\sO_Y)$. 
Then we can use the Zariski cover $\fV= \{V_i= Y\setminus V(t_i), i=1,\ldots c\}$ of $V$ 
to build the Cech complex $\sC^\bullet(\fV, \sF)$, which is a complex of sheaves on $V$
resolving $\sF_{|V}$. We obtain a natural map $H^i(j_*\sC^\bullet(\fV, \sF))\to R^i j_*\sF_{|V}$.
For an element $a\in \sF(V_1\cap\ldots\cap V_c)$ we denote by 
\eq{cech-symbol2}{\genfrac{[}{]}{0pt}{}{a}{t_1,\ldots, t_c}\in \Gamma(Y,\sH^c_Z(\sF))}
the image of $a$ under the composition
\mlnl{\sF(V_1\cap\ldots\cap V_c)\to \Gamma(Y, H^{c-1}(j_*\sC^\bullet(\fV, \sF)))\\ 
        \to   \Gamma(Y, R^{c-1}j_*(\sF_{|V}))\xr{\eqref{cech-symbol1}} \Gamma(Y, \sH^c_Z(\sF)).}
\end{no}

\begin{lem}\label{lem:MilnorKsuppSNCD}
Let $E\subset X$ be a simple normal crossing divisor and denote by $j: V\inj X$ the inclusion
of the complement. 

Then 
$\sH^i_E(\sK^M_{r,X})=0$ for all $i\neq 1$. 
Furthermore,  for $r\ge 1$ and $x\in E$
\eq{lem:MilnorKsuppSNCD1}{(j_*\sK^M_{r,V})_x= 
        \{(\sO_{X,x}[\tfrac{1}{f}])^\times,\ldots, (\sO_{X,x}[\tfrac{1}{f}])^\times\}\subset K^M_{r}(k(\eta)),}
where $f\in \sO_{X,x}$ is a local equation for $E$.
\end{lem}
\begin{proof}
First of all, notice that for a smooth closed subscheme $Z\subset X$ of pure codimension $c$, we have 
$\ul{\Gamma}_Z(C^\bullet_{r,X})= C^\bullet_{r-c, Z}[-c]$.
Hence $\sH^i_Z(\sK^M_{r,X})=0$ for all $i\neq c$ and $\sH^c_Z(\sK^M_{r,X})=\sK^M_{r-c, Z}$.

Now for the lemma write $E=\cup_{i=1}^n E_i$, where the $E_i$ are the irreducible components of $E$. 
We do induction on $n$. 
If $n=1$, i.e. $E\subset X$ is a smooth integral subscheme of codimension 1, 
the first statement follows directly from the remark above. For the second statement observe
that we obtain the following exact sequence from the long exact localization sequence 
\eq{lem:MilnorKsuppSNCD2}{0\to \sK^M_{r,X}\to j_*\sK^M_{r,V}\xr{\partial_E} i_*\sK^M_{r-1,E}\to 0,}
where $i:E\inj X$ denotes the closed immersion and $\partial_E$ is induced by the symbol 
$\partial_e: K^M_r(k(\eta))\to K^M_{r-1}(k(e))$, with $e\in E$ the generic point.
 Clearly the right hand side of \eqref{lem:MilnorKsuppSNCD1} is contained in the left hand 
side. Therefore it suffices to show
that the left hand side is contained in 
\[\{\sO_{X,x}^\times,\ldots, \sO_{X,x}^\times\}+ \{\sO_{X,x}^\times,\ldots, \sO_{X,x}^\times,f\}
\subset K^M_r(k(\eta)).\]   
 This follows from the short exact sequence above and 
the description \eqref{description-MilnorK}  for $\sK^M_{r, X}$ and $\sK^M_{r-1, E}$.

In general, set $E'= \cup_{i<n}E_i$. Thus $E=E'\cup E_n$ and 
the vanishing assertion follows by induction from the long exact sequence
$\ldots\to \sH^i_{E_n}(\sK^M_{r, X})\to \sH^i_{E}(\sK^M_{r, X})\to 
\sH^i_{E\setminus E_n}(\sK^M_{r, X})\to \ldots$.
Denote by $j_n: V\inj X\setminus E'$ and $i_n: E_n\setminus(E_n\cap E')\inj X\setminus E'$ the inclusions.
The second statement follows by induction from the exact sequence
\[0\to \sK^M_{r, X\setminus E'}\to j_{n*}\sK^M_{r, V}\to i_{n*}\sK^M_{r-1, E_n\setminus(E_n\cap E')}\to 0\]
and a similar argument as in the case $n=1$.
\end{proof}

\begin{cor}\label{cor:MilnorKcompSNCD}
Let $E\subset X$ and $j:V\inj X$ be as in Lemma \ref{lem:MilnorKsuppSNCD}.
\begin{enumerate}
\item\label{jlowerstar} We have $R^i j_*\sK^M_{r,V}=0$, for all $i\ge 1$, and  
\[j_*\sK^M_{r,V}=\{j_*\sO_V^\times,\ldots, j_*\sO_V^\times\}\subset K^M_r(k(\eta)).\]
\item\label{cohsuppnotc} For $T\subset X$ a closed subscheme of pure codimension $c$, we have
\[\sH^i_T(j_*\sK^M_{r,V})=0,\quad \text{for } i< c.\]
\end{enumerate}
\end{cor}
\begin{proof}
By the long exact localization sequence $R^i j_*\sK^M_{r,V}\cong \sH^{i+1}_E\sK^M_{r,X}$ for $i\ge 1$.
Hence \eqref{jlowerstar} follows directly from Lemma \ref{lem:MilnorKsuppSNCD}.
It follows that $j_*C^\bullet_{r, V}$ is a flasque resolution of $j_*\sK^M_{r,V}$, 
which directly implies  \eqref{cohsuppnotc}.
\end{proof}

\begin{cor}\label{cor:MilnorK-coh-in-c}
Let $z\in X$ be a point of codimension $c\ge 1$ and $t_1,\ldots, t_c\in \sO_{X,z}$ a regular system of parameters.
We set $t:=t_1\cdots t_c$ and  $t_{\hat{j}}:=t_1\cdots \widehat{t_j}\cdots t_c$.
(By convention if $c=1$ we set $t_{\hat{1}}:=1$.)
Then with the notation from \ref{cech-symbol}
\[\frac{\{(\sO_{X,z}[\frac{1}{t}])^\times,\ldots, (\sO_{X,z}[\frac{1}{t}])^\times\}}{
  \sum_{j=1}^c 
        \{(\sO_{X,z}[\frac{1}{t_{\hat{j}}}])^\times,\ldots, (\sO_{X,z}[\frac{1}{t_{\hat{j}}}])^\times\}}
         \cong H^c_{z}(\sK^M_{r,X})\cong K^M_{r-c}(k(z)),\]
where on the left the quotient is between two subgroups of $K^M_r(k(\eta))$.
Moreover with the notation from \eqref{cech-symbol2}, the isomorphism 
$K^M_{r-c}(k(z))\xr{\simeq} H^c_{z}(\sK^M_{r,X})$ 
is given by
\eq{cor:MilnorK-coh-in-c1}{\{b_1,\ldots, b_{r-c}\}\mapsto 
          \genfrac{[}{]}{0pt}{}{\{\tilde{b}_1,\ldots, \tilde{b}_{r-c}, t_1,\ldots, t_c\}}{t_1,\ldots, t_c},}
where $\tilde{b}_i\in (\sO_{X,z})^\times$ is any lift of $b_i\in k(z)^\times$.
\end{cor}
\begin{proof}
Since the question is local around $z$ we can assume that the sequence $t_1,\ldots, t_c$ is 
a regular sequence of global sections of $\sO_X$ and that $Z=\ol{\{z\}}$ is globally defined by their vanishing.
We denote by $j: V:= X\setminus Z \inj X$ the open embedding.
For $n\ge 0$ denote by $S^{n}\subset \N^{n+1}$ the set of tuples $(i_0,i_1,\ldots, i_n)$ 
with $1\le i_0< \ldots <i_n\le c$.
For $i\in [1,c]$ set $V_i:= X\setminus V(t_i)$ and for 
$I=(i_0,\ldots, i_n)\in S^n$ set $V_I:= V_{i_0}\cap\ldots\cap V_{i_n}$.
Denote by $j_I: V_I\inj V$ the open embeddings.  
By Corollary \ref{cor:MilnorKcompSNCD}, \eqref{jlowerstar}, we have 
\[Rj_{*} (j_{I*}j_I^{-1}\sK^M_{r,V}) = Rj_{*}Rj_{I*}\sK^M_{r,V_I}=
              j_{*}j_{I*}\sK^M_{r,V_I},\]
where for the first and second equality,  we use that the inclusions $V_I\inj V$ and $V_I\inj X$ 
are complements of an SNCD.
It follows that the Cech complex $\sC^\bullet(\fV, \sK^M_{r,V})$ is acyclic for $j_{*}$,
where $\fV=\{V_1,\ldots, V_c\}$.
Therefore $R j_{*} (\sK^M_{r,V})= j_{*}\sC^\bullet(\fV, \sK^M_{r,V})$.
Now the first isomorphism from the statement of the corollary follows from \eqref{cech-symbol1} and
Corollary \ref{cor:MilnorKcompSNCD}, \eqref{jlowerstar}.
The second isomorphism in the statement is an immediate consequence of the fact that $j_*C^\bullet_{r, V}$ is a 
flasque resolution of $j_*\sK^M_{r,V}$ (see Corollary \ref{cor:MilnorKcompSNCD}, \eqref{jlowerstar}).

It remains to prove the explicit formula \eqref{cor:MilnorK-coh-in-c1}.
We can assume $X=\Spec A$ with  $A:=\sO_{X,z}$. Then for an abelian sheaf $\sF$ on $V$
the stalk of $j_{*}\sC^\bullet(\fV,\sF)$ at $z$ is the following complex of abelian groups (starting in degree 0)
\[G(\sF)^\bullet: \bigoplus_{I\in S^0} \sF(V_I)\xr{\check{\partial}^0} 
            \bigoplus_{I\in S^1}\sF(V_I)\xr{\check{\partial}^1}\ldots
         \xr{\check{\partial}^{c-2}}   \bigoplus_{I\in S^{c-1}}\sF(V_I),\]
with 
\[(\check{\partial}^n(\alpha_I)_{I\in S^n})_J= \sum_{j=0}^{n+1} (-1)^j (\alpha_{J(j)})_{|V_J},\]
where $J(j)$ equals the tuple $J$ with the $j$-th entry omitted. 
Let $C^\bullet_{r,V}$ be the Gersten complex from \eqref{Gersten} and 
set $C^\bullet:=j_*C^\bullet_{r,V}$. Then
the sections of $C^\bullet$ over $V_I$ form the following complex
\mlnl{C^\bullet(V_I):=\Gamma(V_I, j_*C^\bullet_{r,V}):\\
          \bigoplus_{x\in V_I^{(0)}} K^M_r(k(x))\xr{\partial^{K,0}} 
                \bigoplus_{x\in V_I^{(1)}}K^M_{r-1}(k(x))\xr{\partial^{K,1}}\ldots \xr{\partial^{K,c-2}}
              \bigoplus_{x\in V_I^{(c-1)}} K^M_{r-c+1}(k(x)).}
Let $T$ be the total complex associated to the double complex $G^\bullet(C^\bullet)$; its differentials are given by
\[\partial^{T,n}=(\check{\partial}^i + (-1)^i\partial^{K, n-i})_{i\in [0,n]}:  T^n\to T^{n+1}.\]
Then the natural maps $G^\bullet(\sK^M_{r,V})\to G^\bullet(C^0)$
and  $C^\bullet(V)\to G^0(C^\bullet)$ induce quasi-isomorphisms
\[G^\bullet(\sK^M_{r,V})\xr{\simeq} T \xl{\simeq} C^\bullet(V).\]
For $i\in[0, c]$ the vanishing loci $V(t_{c-i+1},\ldots, t_c)\subset X=\Spec A$ 
are integral closed subschemes which are regular; we denote by $z_i$ their unique generic points, i.e.
\[\ol{\{z_i\}}=V(t_{c-i+1},\ldots, t_c).\]
In particular, $z_i\in X^{(i)}$,  $z_c=z$ and $z_0$ is the generic point of $X$. 
Take  $b_1,\ldots, b_{r-c}\in k(z)^\times$ and let $\tilde{b}_1,\ldots, \tilde{b}_{r-c}\in A^\times$ be lifts.
(By abuse of notation we will also write $\tilde{b}_i$ (resp. $t_i$) for the image of $\tilde{b}_i$ (resp. $t_i$)
 under any ring homomorphism $A\to R$.) For $i\in [0, c-1]$ set
\[a_{c-1-i}:=\{\tilde{b}_1,\ldots, \tilde{b}_{r-c}, t_1,\ldots, t_{i+1}\}\in K^M_{r-(c-1-i)}(k(z_{c-1-i})),\quad
              i\in[0, c-1],\]  
and define
\[\alpha_{i}:=\big((\alpha_{i, I,x})_{x\in V_I^{(c-1-i)}}\big)_{I\in S^i} 
            \in G^i(C^{c-1-i})=\bigoplus_{I\in S^i}C^{c-1-i}(V_I),\]
by
\[\alpha_{i,I,x}=\begin{cases} a_{c-1-i}, &\text{if } I=(1,\ldots i+1)\text{ and } x= z_{c-1-i,}\\
                                             0, & \text{else.}\end{cases}\]
By definition 
\[\{\tilde{b}_1,\ldots,\tilde{b}_{r-c}, t_1,\ldots, t_c\}\mapsto\alpha_{c-1} \quad\text{under}\quad 
 G^{c-1}(\sK^M_{r,V})\to T^{c-1}\]
and 
\[\{\tilde{b}_1,\ldots, \tilde{b}_{r-c}, t_1\}\mapsto \alpha_0 \quad\text{under}\quad 
                K^M_{r-c+1}(k(z_{c-1}))\to C^{c-1}(V)\to T^{c-1}.\]
Further, under the composition 
\mlnl{K^M_{r-c+1}(k(z_{c-1}))\to C^{c-1}(V)\surj (R^{c-1}j_*\sK^{M}_{r,V})_z\\ \xr{\partial}
     H^c_{z}(\sK^M_{r,X})= \Gamma_{z}C^c_{r,X}=K^M_{r-c}(k(x))}
the element $\{\tilde{b}_1,\ldots, \tilde{b}_{r-c}, t_1\}$ is sent to  $\{b_1,\ldots, b_{r-c}\}$.
Altogether it remains to show that for $c\ge 2$ we have
\eq{cor:MilnorK-coh-in-c2}{\alpha_0\equiv \alpha_{c-1}\text{ mod }\partial^{T,c-2} T^{c-2}.}
To this end we define for $i\in [0, c-2]$,
\[\beta_i=((\beta_{i,I,x})_{x\in V^{(c-2-i)}})_{I\in S^i}\in G^i(C^{c-2-i})\]
by
\[\beta_{i,I,x}=\begin{cases} a_{c-2-i}, &\text{if } I=(1,\ldots, i+1) \text{ and } x=z_{c-2-i},\\
                                               0, &\text{else}.
         \end{cases}\]
We have 
\[\check{\partial}^i(\beta_i)= (-1)^{i+1}\alpha_{i+1}.\]
One checks this easily  using that for $J\in  S^{i+1}$ and $j\in[1,i+2]$,
we have 
\[J(j)=(1,\ldots, i+1)\text{ and } z_{c-2-i}\in V_J^{c-2-i}\Longleftrightarrow J=(1,\ldots, i+2) \text{ and } j=i+2.\]
On the other hand 
\[\partial^{K, c-2-i}(\beta_i)=\alpha_i.\]
This directly follows from
\[x\in \ol{\{z_{c-2-i}\}}^{(1)}\cap V_{(1,\ldots, i+1)} \text{ and } \partial_x(a_{c-2-i})\neq 0 \Longleftrightarrow
        x=z_{c-1-i}.\]
Thus
\[\partial^{T, c-2}(\beta_i)=\check{\partial}^i(\beta_i)+ (-1)^i \partial^{K,c-2-i}(\beta_i)
              = (-1)^{i+1}(\alpha_{i+1}-\alpha_i).\]
Altogether
\[\alpha_{0}\equiv \alpha_1\equiv \ldots\equiv \alpha_{c-1} \quad \text{mod } \partial^{T,c-2}T^{c-2}.\]
This shows \eqref{cor:MilnorK-coh-in-c2} and hence finishes the proof.
\end{proof}

\subsection{The relative Milnor {$K$}-sheaf}

\begin{defn}\label{defn:MilnorModulus}
Let $D$ be an effective divisor on $X$. Denote by $j: U:=X\setminus D\inj X$ the inclusion of the complement.
\begin{enumerate}
\item We define the Zariski sheaf $\sK^M_{r,X|D}$ for $r\in \Z$ to be the image of the map 
\[\Ker(\sO_X^\times\to \sO_D^\times)\otimes_{\Z} j_*\sK^M_{r-1,U}\to j_*\sK^M_{r,U},\quad 
       a\otimes \{b_1,\ldots, b_{r-1}\}\mapsto \{a,b_1,\ldots, b_{r-1}\}.\]
In particular $\sK^M_{r, X|D}=0$ for $r\le 0$ and $\sK^M_{1, X|D}=\Ker(\sO_X^\times\to \sO_D^\times)$.
\item\label{defn:MilnorModulus(2)}  We have a presheaf on the small Nisnevich site of $X$
   \[X_{\rm Nis}\to (\text{abelian groups}), \quad (v:V\to X)\mapsto H^0(V, \sK^{M}_{r,V| v^*D})=:\sK^M_{r,X|D}(V).\]
  We denote  by $\sK^M_{r, X|D,{\rm Nis}}$ the Nisnevich sheaf on $X_{\rm Nis}$ associated to this functor.
    If $u:X'\to X$ is \'etale and $x'\in X'$ is a point we set 
 \eq{defn:MilnorModulus3}{\sK^{M,h}_{r,X|D,x'}:=\varinjlim_{(v,y)} H^0(V,\sK^M_{r, V|(u\circ v)^*D}),}
      where the limit is over the filtered category  of pairs $(v,y)$, where $v:V\to X'$ is \'etale and $y\in V$ is a point
such that $v$ induces an isomorphism $k(x')\xr{\simeq} k(y)$.  
\end{enumerate}
    \end{defn}

\begin{rmk}\label{rmk:MilnorModulus}
If $v: V\to X$ is an \'etale map that factors through the open immersion $ j:U\inj X$, then  
by \ref{K-Nis}
\[H^0(V, \sK^M_{r, X|D, {\rm Nis}})=\sK^M_{r,U}(V)=H^0(V,\sK^M_{r,V}).\] 
Assume $D_{\rm red}$ is a SNCD. For $x\in D$, set $A:=\sO_{X,x}$   
and denote by $A^h$ its henselization. 
Then $\sK^M_{r, X|D, x}$ (resp. $\sK^{M,h}_{r, X|D, x}$)
is by Lemma \ref{lem:MilnorKsuppSNCD}
the subgroup of $K^M_r(k(\eta))$ (resp. $K^M_r({\rm Frac}(A^h))$)
generated by symbols of the form $\{1+fa, b_1,\ldots, b_{r-1}\}$, 
where $f\in A$ is a local equation for $D$,
 $a\in A$ (resp. $A^h$) and $b_i\in A[\frac{1}{f}]^\times$ 
(resp. $A^h[\frac{1}{f}]^\times$).
\end{rmk}

The stalk of the sheaf $\sK^M_{r,X|D}$  at a generic point of the effective divisor $D$  looks as follows.
\begin{no}{}\label{U-for-DVR}  
       Let $A$ be a discrete valuation ring with its maximal ideal $\fm$ and $K$ the field of fractions.
       We set $U^{(0)}_K=A^\times$ and $U^{(n)}_K=1+\fm^n$, for $m\ge 1$. 
     We denote by $U^{0}K^M_r(K)$ the image of the natural map 
       $(A^\times)^{\otimes r}\to K^M_r(K)$ and by $U^{n}K^M_r(K)$, $n\ge 1$,
 the image of the multiplication map $U^{(n)}_K\otimes_\Z K^M_{r-1}(K)\to K^M_r(K)$.
\end{no}
  
The following two Lemmas are well known.
        
\begin{lem}\label{lem:DVR-cDVR}
Let $(A, K,\fm)$ be as above and denote by $\hat{K}$ the fraction field of the completion of $A$ along $\fm$.
 Then for all $n\ge 1$ the natural map
\[K^M_r(K)/U^{n}K^M_r(K)\to K^M_r(\hat{K})/U^{n}K^M_r(\hat{K})\]
is an isomorphism.
\end{lem}
\begin{proof}
We define an inverse map. Clearly there is a well defined map 
$(\hat{K}^\times)^{\otimes_\Z r}\to K^M_r(K)/U^{n}K^M_{r}(K)$ which sends an element
$a_1\otimes\ldots\otimes a_r$ to the class of $\{b_1,\ldots, b_r\}$, where we take  any $b_i\in K^\times$
with $b_i\equiv a_i$ mod $1+\fm^n$. This map also kills the Steinberg relations.
Indeed if we take $a\in \hat{K}^\times\setminus U^{(1)}_{\hat{K}}$ and $b\in K^\times$ with
$b\equiv a$ mod $U^{(n)}_{\hat{K}}$ then $1-b\equiv 1-a$ mod $U^{(n)}_{\hat{K}}$.
Hence $a\otimes (1-a)$ is sent to the class of $\{b, 1-b\}=0$. If we take $a\in U^{(1)}_{\hat{K}}$ and
$b\in K^\times$ with $b\equiv 1-a$ mod $U^{(n)}_{\hat{K}}$, then $1-b\equiv a$ mod $U^{(n)}_{\hat{K}}$
and $a\otimes (1-a)$ is sent to the class of $\{1-b,b\}=0$. 
It follows that this map factors to give a well-defined map inverse to the natural map from the statement.
\end{proof}

\begin{lem}\label{lem:symbol-formula}
Let $A$ be an integral local ring with its maximal ideal $\fm$ and the fraction field $K={\rm Frac}(A)$.
For elements $a,b,c\in A$ and $s,t\in \fm$, the following equality holds in $K^M_2(K)$:
\begin{enumerate}
\item\label{lem:symbol-formula1} $\{1+as,1+bt\}=-\{1+\frac{ab}{1+as}st,-as(1+bt)\}$.
\item $\{1+\frac{s-1}{1+ct} ct, 1-\frac{1+ct}{1+cst} s\}= \{1+c st, s\}$.
\end{enumerate}
\end{lem}
\begin{proof}
(1) is straightforward and (2) follows from (1) by setting
\[a=-\frac{1+ct}{1+cst}, \quad b=\frac{c(s-1)}{1+ct}.\]
\end{proof}

\begin{prop}\label{prop:rel-K-in-K}
Let $D$ be an effective divisor on $X$ whose support has simple normal crossings.
Let $x\in D$ be a point and $D_1,\ldots, D_n$ all the irreducible components of $D$
passing through $x$. Let  $t_i\in \sO_{X,_x}$ be a local equation for $D_i$ around $x$
and  assume that around $x$ the divisor $D$ is given by the vanishing of $t_1^{m_1}\cdots t_n^{m_n}$, 
with $m_i\ge 1$.
\begin{enumerate}
\item Assume either there exists an $i_0\in [1,n]$ with $m_{i_0}\ge 2$ or $n\ge r$.
        Then $\sK^M_{r,X|D, x}$ is equal to the subgroup of $K^M_r(k(\eta))$ which is generated by
elements of the form
\eq{prop:rel-K-in-K1}{\{1+ a\cdot \prod_{i\in I_s}t^{m_i-1}_i\cdot
      \prod_{i\in [1,n]\setminus I_s}t^{m_i}_i, 
 1+u_1t_{i_1}, \ldots, 1+u_s t_{i_s},  u_{s+1},\ldots, u_{r}\},}
where $s\in [0,{\rm min}(r-1,n)]$, $I_s=\{i_1,\ldots, i_s\}\subset [1,n]$, $a\in \sO_{X,x}$ and 
 $ u_i\in \sO_{X,x}^\times$.
\item Assume $m_1=\ldots= m_n=1$ and $n< r$. 
         Then $\sK^M_{r,X|D, x}$ is equal to the subgroup of $K^M_r(k(\eta))$ which is generated by
elements of the form \eqref{prop:rel-K-in-K1} for $s\le n-1$ together with elements of the form
  \eq{prop:rel-K-in-K2}{\{1+u_1t_1,\ldots, 1+u_n t_n, u_{n+1},\ldots, u_r\}, 
            \quad u_i\in\sO_{X,x}^\times.}   
\end{enumerate}
\end{prop}

\begin{proof}
Set $A=\sO_{X,x}$ and denote by $\fm$ its maximal ideal. 
The statement holds for $r= 1$ by definition. For $r\ge 2$
denote by $L_r$ the subgroup of $K^M_r(k(\eta))$ which in case (1) is generated by the elements
\eqref{prop:rel-K-in-K1}  and in case (2) is generated  
by the elements \eqref{prop:rel-K-in-K1} for $s\le n-1$ and the elements \eqref{prop:rel-K-in-K2}.
In both cases the inclusion $L_r\subset \sK^M_{r,X|D,x}$ follows directly from 
Lemma \ref{lem:symbol-formula}, (1) and Remark \ref{rmk:MilnorModulus}. 
For the other inclusion it suffices to show (in both cases)
\[\{1+at_1^{m_1}\cdots t_n^{m_n}, t_{i_1},\ldots, t_{i_s} \}\in L_{s+1}\]
for $a\in A$, $\{i_1,\ldots, i_s\}\subset[1,n]$.
If one of the $m_i$'s is $\ge 2$ or $n>s$ this follows directly from Lemma \ref{lem:symbol-formula}, (2).
If $m_1=\ldots=m_n=1$ and $s=n$, then we can use Lemma \ref{lem:symbol-formula}, (2)
to reduce to the case $n=1$. 
Setting $t:=t_1$ it therefore remains to show 
\eq{prop:rel-K-in-K3}{\{1+a t, t\}\in L_2,\quad a\in A.}
To this end notice that $1+t A$ is multiplicatively generated by elements in
$1+t A^\times$. Indeed if $b\in\fm$ we can write
 \[1+t b= (1+t\frac{1}{1+t(b-1)}) (1+t(b-1)).\]
Therefore we can assume in \eqref{prop:rel-K-in-K3} that $a\in A^\times$.
Then the statement follows from
$0=\{1+ta, -ta\}= \{1+ta, t\} + \{1+ta, -a\}.$
This finishes the proof.
\end{proof}

\begin{cor}\label{cor:rel-K-in-K}
Let $D_1$ and $D_2$ be effective divisors on $X$ whose supports are simple normal crossing divisors.
Assume $D_1\le D_2$. Then we have the inclusion of sheaves
\[\sK^M_{r, X|D_2}\subset \sK^M_{r, X|D_1}\subset \sK^M_{r,X}\quad \text{on }X_{\Zar}\]
and 
\[\sK^M_{r, X|D_2,\Nis}\subset \sK^M_{r, X|D_1,\Nis}\subset \sK^M_{r,X}\quad \text{on }X_{\Nis}\]
\end{cor}
\begin{proof}
This follows directly from Proposition \ref{prop:rel-K-in-K}.
\end{proof}

\subsection{The structure of relative Milnor {$K$}-sheaves}\label{sec-structure-relK}
In this subsection we assume that $D$ is an effective divisor on $X$ whose support has simple normal crossings.
We denote by $i: D_{\rm red}\inj X$ the corresponding closed immersion, by
 $j: U=X\setminus D\inj X$ the inclusion of the complement and by $\{D_\lambda\}_{\lambda\in \Lambda}$ the irreducible components of $D$. We write $\Omega^q_{X}=\Omega^q_{X/\Z}$ etc.

\begin{no}{}\label{gr-rel-K}
We write $\N=\{0,1,2,\ldots\}$ and endow $\N^\Lambda$ with a semi-order by
\[(m_\lambda)_{\lambda\in \Lambda}\le (m_\lambda)_{\lambda\in \Lambda} \, \Leftrightarrow\,
  m_\lambda\le n_\lambda,\quad \text{for all }\lambda\in \Lambda.\]
For $\fm=(m_{\lambda})_{\lambda\in \Lambda}\in \N^\Lambda$, we set
\[D_\fm:=\sum_{\lambda\in\Lambda}m_{\lambda} D_\lambda.\]
For $\nu \in \Lambda$, set 
\[\delta_\nu=(0,\ldots,\overset{\overset{\nu}{\lor}}{1},\ldots,0)\in \N^\Lambda\] 
and we define the following sheaves for $r\ge 1$
\[\gr^{\fm,\nu}\sK^M_{r,X}:= \sK^M_{r,X|D_\fm}/\sK^M_{r,X|D_{\fm+\delta_\nu}} \quad \text{on } X_{\Zar}\]
and
\[\gr^{\fm,\nu}\sK^M_{r,X, {\rm Nis}}:= 
          \sK^M_{r,X|D_\fm, {\rm Nis}}/\sK^M_{r,X|D_{\fm+\delta_\nu}, {\rm Nis}}\quad \text{on } X_\Nis.\]
Notice that this makes sense by Corollary \ref{cor:rel-K-in-K} and that these sheaves have support in $D_\nu$.
We remark that $\gr^{\fm,\nu}\sK^M_{r,X, {\rm Nis}}$
is also the Nisnevich sheaf associated to the presheaf
\eq{gr-rel-K1}{X_{\rm Nis}\ni (v:V\to X)\mapsto 
                H^0(V_{\rm Zar}, \sK^M_{r, V|v^*D_{\fm}}/ \sK^M_{r, V|v^*D_{\fm+\delta_\nu}})=:
                      \gr^{\fm,\nu}\sK^M_{r,X}(V).}
For an \'etale map $v:V\to X$ we  can write 
$v^*D_\lambda= D_{\lambda, 1}\sqcup \ldots\sqcup D_{\lambda, j_\lambda}$, with $D_{\lambda,i}\subset V$
irreducible smooth divisors. For a subset $S\subset \Lambda$ set
\[v^*S:= \{(\lambda, i)\,|\, \lambda\in S,\, i\in[1,j_\lambda]\}\]
and for $i\in [1, j_\nu]$
\[  \fm_{(\nu,i)}:=(m_{(\lambda,j)})_{(\lambda,j)\in v^*(\Lambda\setminus\{\nu\}) } +m_\nu \delta_{(\nu,i)},\]
  with 
\[ m_{(\lambda,j)}:=m_\lambda\quad \text{and}\quad 
    \delta_{(\nu,i)}=(0,\ldots,\overset{\overset{(\nu,i)}{\lor}}{1},\ldots,0)\in \N^{v^*\Lambda} .\]
Then $\{D_{\lambda'}\}_{\lambda'\in v^*\Lambda}$ are the irreducible components 
of $v^*D_{\rm red}$ and 
\eq{gr-rel-K2}{\frac{\sK^M_{r, V|v^*D_{\fm}}}{\sK^M_{r, V|v^*D_{\fm+\delta_\nu}}}
       =\bigoplus_{i=1}^{j_\nu}
         \frac{\sK^M_{r, V|D_{\fm_{(\nu,i)}}}}{ \sK^M_{r, V|D_{\fm_{(\nu,i)}+\delta_{(\nu,i)}}}}
       \stackrel{\text{by defn}}{=}\bigoplus_{i=1}^{j_\nu} \gr^{\fm_{(\nu,i)}, (\nu,i)}\sK^M_{r, V}.}
\end{no}

\begin{prop}\label{prop:gr0-map}
We keep the notations from above.
Let  $\fm=(m_\lambda)_{\lambda\in \Lambda}$ be an element in $\N^\Lambda$
and take $\nu\in \Lambda$, $r\ge 1$. Denote by $i_\nu: D_\nu\inj X$ the closed immersion.
Assume $m_\nu=0$ and set 
      \[D_{\nu,\fm}:=\sum_{\lambda\in \Lambda\setminus\{\nu\}} m_\lambda (D_\nu\cap D_\lambda).\]
 Then there is a natural surjection
          \eq{prop:gr0-map1}{\gr^{\fm,\nu}\sK^M_{r,X}\surj i_{\nu*}\sK^M_{r, D_\nu|D_{\nu,\fm} }. }
 If the $t_\lambda$'s are local equations for the $D_\lambda$'s around a point $x\in X$, then the composition 
 of this map with $\sK^M_{r, X|D_\fm}\to \gr^{\fm,\nu}\sK^M_{r,X}$ is given by
  \eq{prop:gr0-map1.5}{
           \{1+t^\fm a, b_1,\ldots, b_{r-1}\}\mapsto \{1+t^\fm\bar{a}, \bar{b}_1,\ldots, \bar{b}_{r-1}\},}
           where $a\in \sO_X$, $b_i\in \sO_{X\setminus |D_{\fm}|}^\times$ with
           $\bar{a}\in \sO_{D_\nu}$, $\bar{b}_i\in \sO_{D_\nu\setminus |D_{\nu,\fm}|}^\times$
               as their images and $t^\fm=\prod_{\lambda\in\Lambda} t_\lambda^{m_\lambda}$.
This map induces an isomorphism between sheaves on $X_{\rm Nis}$
 \eq{prop:gr0-map2}{\gr^{\fm,\nu}\sK^M_{r,X,{\rm Nis}}\xr{\simeq} 
                                                      i_{\nu*}\sK^M_{r, D_\nu|D_{\nu,\fm}, {\rm Nis}}. }
Furthermore, if $D_{\rm red}$ is smooth, then \eqref{prop:gr0-map1} is already an isomorphism.
\end{prop}
\begin{proof}
Assume $t_\nu\in \Gamma(X, \sO_X)$ is an equation for $D_\nu$. Then we have the following map at our disposal
\[s_{t_\nu}:\sK^M_{r,X}\to \sK^M_{r, D_\nu}, \quad
      \alpha\mapsto s_{t_\nu}(\alpha):=\partial_{D_\nu}(\alpha\cdot\{t_\nu\}),\]
where $\partial_{D_\nu}: K^M_{r+1}(k(X))\to K^M_{r}(k(D_\nu))$ denotes the tame symbol defined by the 
valuation corresponding to $D_\nu$. One directly checks 
\[s_{t_{\nu}}(\{a_1,\ldots, a_r\})= \{\bar{a}_1,\ldots, \bar{a}_r\}, \]
where $a_i\in \sO_{X}^\times$ and $\bar{a}_i\in\sO_{D_\nu}^\times$ is its image.
This also shows that $s_{t_\nu}$ does not depend on the choice of the equation $t_\nu$.
Therefore we can write $s_{D_\nu}$ instead of $s_{t_\nu}$. In particular, in case $D_\nu$ is not
given by a global equation we can locally  define maps as above and glue them to obtain a morphism
\[s_{D_\nu}:\sK^M_{r,X}\to \sK^M_{r,D_\nu}.\]
Restricting along the open immersion $j:X\setminus|D_\fm|\inj X$ we obtain an induced map
\[\sK^M_{r, X|D_\fm}\inj j_*\sK^M_{r, X\setminus|D_\fm|}\xr{s_{D_\nu}} 
          j_*\sK^M_{D_\nu\setminus|D_{\nu,\fm}|}.\]
It is immediate to check that the image of this map is $\sK^M_{r,D_\nu|D_{\nu,\fm}}$ and that it factors to give the
map \eqref{prop:gr0-map1} from the statement. (Use Proposition \ref{prop:rel-K-in-K} 
to check that $\sK^M_{r,X|D_{\fm+\delta_\nu}}$ is mapped to zero.)

If $D_{\rm red}$ is smooth, then \eqref{prop:gr0-map1} is an isomorphism.
Indeed, it suffices to consider the case in which $D$ is connected. Then
           \eqref{prop:gr0-map1} is a map $\sK^M_{r, X}/\sK^M_{r, X|D}\to \sK^M_{r, D}$
and it is easy to see that the assignment $\{\bar{a}_1,\ldots, \bar{a}_r\}\mapsto \{a_1,\ldots, a_r\}$ 
mod $\sK^M_{r, X|D}$,  in which the $a_i\in \sO_{X}^\times$ are lifts of the $\bar{a}_i\in \sO_{D}^\times$, 
induces a well-defined map $\sK^M_{r, D}\to \sK^M_{r, X}/\sK^M_{r, X|D}$, 
which is inverse to \eqref{prop:gr0-map1}.

Let $v: V\to X$ be \'etale. With the notation from \eqref{gr-rel-K2} we have 
\[\sK^M_{r, v^*D_\nu|v^*D_{\nu,\fm}}=
 \bigoplus_{i=1}^{j_\nu}\sK^M_{r, D_{(\nu,i)}|D_{(\nu,i),\fm_{(\nu,i)}}},\]
here $D_{(\nu,i)}$, $i\in[1,j_\nu]$, are the irreducible components of $|v^*D_\nu|$
and 
\[D_{(\nu,i),\fm_{(\nu,i)}}=
 \sum_{(\lambda,j)\in v^*(\Lambda\setminus\{\nu\})} m_\lambda (D_{(\lambda,j)}\cap D_{(\nu,i)}) .\]
It follows that the map \eqref{prop:gr0-map1} induces a map from the presheaf
\eqref{gr-rel-K1} to the presheaf 
\[X_{\rm Nis}\ni(v:V\to X)\mapsto H^0(V, i_{\nu*}\sK^M_{D_\nu|D_{\nu,\fm}})
                        =\sK^M_{D_\nu|D_{\nu,\fm}}(v^{-1}D_{\nu}),\]
where we use the notation from Definition \ref{defn:MilnorModulus}, \eqref{defn:MilnorModulus(2)}.
We obtain the map \eqref{prop:gr0-map2} by Nisnevich sheafification.
The surjectivity of \eqref{prop:gr0-map2} follows from the surjectivity of \eqref{prop:gr0-map1}.
To prove the injectivity, it suffices to show the following (for all $(X,D)$):
Let $x\in D_\nu$ be a point, let $V\subset X$ be an open neighborhood of $x$ and
$\alpha\in H^0(V, \gr^{\fm ,\nu}\sK^M_{r,X})$ be an element which under \eqref{prop:gr0-map1} 
is mapped to zero in $H^0(V\cap D_\nu, \sK^M_{r, D_\nu|D_{\nu,\fm}})$. Then there exists 
an \'etale morphism $v :V'\to V$ and a point $x'\in V'$ such that $v$ induces an isomorphism $k(x)\xr{\simeq} k(x')$
with the property that $v^*\alpha=0$ in $ \gr^{\fm,\nu}\sK^M_{r, X|D}(V')$.

To this end, we can assume, after shrinking $V$ around $x$, that we have a cartesian diagram
\[\xymatrix{ D_{\nu,V}:=D_\nu\cap V\ar@{^(->}[r]\ar[d] & V\ar[d] \\ \A^{n-1}_k\ar@{^(->}[r] & \A^n_k, }\]
in which the vertical arrows are \'etale, the bottom horizontal arrow is induced by 
$k[t_1,\ldots, t_n]\surj k[t_1,\ldots,t_n]/(t_n)$ and the pullback of the coordinate $t_\lambda$ to $\sO_V$ is a local
equation for $D_\lambda$. We choose a splitting $\A^n_k\to \A^{n-1}_k$ of the bottom map; 
in this way $V$ becomes an $\A^{n-1}_k$-scheme and we set
\[V_1:= V\times_{\A^{n-1}_k} D_{\nu, V}.\]
We have a diagonal embedding $D_{\nu, V}\inj V_1$. The projection $v_1: V_1\to V$ is \'etale and hence
we can write $v_1^*(D_{\nu,V})= D_{\nu, V}\sqcup E$, for some smooth divisor $E\subset V_1$. 
We set $V':=V_1\setminus E$ and denote by $v: V'\to V$ the map induced by $v_1$.
Then $v: V'\to V$ is \'etale, $v$ induces an isomorphism $v^{-1}(D_{\nu,V})\xr{\simeq} D_{\nu, V}$ and 
there is a natural map induced by the projection $\pi:V'\to D_{\nu, V}$ which splits the inclusion
$D_{\nu, V}\inj V'$. We obtain a commutative diagram
\eq{prop:gr0-map2.5}{\xymatrix{ \gr^{\fm,\nu}\sK^M_{r,X}(V')\ar[dr]^{(*)}  &  \\
                   \gr^{\fm,\nu}\sK^M_{r,X}(V)\ar[u]^{v^*}\ar[r] & \sK^M_{r, D_\nu|D_{\fm,\nu}}(D_{\nu,V}) .   
}}
It suffices to show that $(*)$ in \eqref{prop:gr0-map2.5} is injective.
Denote by $D_{\fm,V}$, $D_{\fm+\delta_\nu,V}$ and $D_{\fm,\nu,V}$
 the pullback along the open immersion $V\inj X$ of $D_{\fm}$, $D_{\fm+\delta_\nu,V}$ and $D_{\fm,\nu}$,
 respectively. We consider the composition
\[\sK^M_{r,D_{\nu, V}}\xr{\pi^*} \sK^M_{r, V'}\to \sK^M_{r, V'}/\sK^M_{r, V'|v^*D_{\fm+\delta_\nu,V}}.\]
The restriction of this map to $\sK^M_{r, D_{\nu,V}|D_{\fm,\nu, V}}$ induces a map
\[\sK^M_{r, D_{\nu,V}|D_{\fm,\nu, V}}\to \sK^M_{r, V'|v^*D_{\fm,V}}/\sK^M_{r, V'|v^*D_{\fm+\delta_\nu,V}}.\]
Taking global sections we obtain a map 
\eq{prop:gr0-map3}{\sK^M_{r, D_\nu|D_{\fm,\nu}}(D_{\nu,V})\to  \gr^{\fm,\nu}\sK^M_{r,X}(V').}
Using the explicit description \eqref{prop:gr0-map1.5} of the map $(*)$ in \eqref{prop:gr0-map2.5} it is 
straightforward to check that \eqref{prop:gr0-map3} and $(*)$ are inverse to each other. 
This finishes the proof of the proposition.
 \end{proof}

\begin{rmk}
The proof of the injectivity of \eqref{prop:gr0-map2} is the only place where we need the Nisnevich topology.
\end{rmk}

\begin{no}\label{Nis-to Zar}
We denote by  $X_{\Zar}^{\tilde{}}$ (resp. $X_{\Nis}^{\tilde{}}$) the topos of sheaves of sets
on the site $X_{\Zar}$ (resp. $X_{\Nis}$) and by
$\e=(\e^{-1}, \e_*): X_{\Nis}^{\tilde{}}\to X_{\Zar}^{\tilde{}}$ the natural morphism of topoi.
Then $\epsilon_*$ is left exact when restricted to the category of abelian sheaves and 
right derives to a functor
\[R\e_*: \sD^+(X_{\Nis})\to \sD^+(X_\Zar),\]
between the derived categories of bounded below complex of abelian sheaves on $X_\Nis$ and $X_\Zar$, respectively.
Since the cohomological dimension of $X_\Nis$ is $\le\dim X$ (see e.g. \cite[1.32]{Nis}) this functor restricts 
to a functor between the derived category of complexes with bounded cohomology
\[R\e_* : \sD^b(X_\Nis)\to \sD^b(X_\Nis).\]
\end{no}

\begin{cor}\label{cor:gr0-map}
In the situation of Proposition \eqref{prop:gr0-map} we have a distinguished triangle in $\sD^b(X_{\rm Zar})$
\[R\e_*(\sK^M_{r, X|D_{\fm+\delta_\nu}, {\rm Nis}})\to R\e_*(\sK^M_{r, X|D_\fm, {\rm Nis}})\to
     R(\e\circ i_{\nu})_*(\sK^M_{r, D_\nu|D_{\nu,\fm},\Nis})\xr{[1]}.\]
\end{cor}
\begin{proof}
This follows directly from Proposition \ref{prop:gr0-map} together with the observation
$R\e_* i_{\nu*}=R\e_*Ri_{\nu*}=R(\e\circ i_\nu)_*$.
\end{proof}

\begin{no}\label{restricted-log-differentials}
We keep the notations from \ref{gr-rel-K}.
For $\nu\in \Lambda$ and $q\ge 0$ we define the following sheaf on $X_{\Zar}$ (with support in $D_\nu$) 
\eq{restricted-log-differentials1}{\omega_{X|D,\fm,\nu}:=\omega^q_{\fm, \nu}:= 
                        (\Omega^q_X(\log D)(-D_\fm))_{|D_\nu},}
where we use the short hand notation 
\[\Omega^q_X(\log D)(-D_\fm):=\sO_X(-D_\fm) \otimes_{\sO_{X}}\Omega^q_X(\log D).\]
It is immediate to check that the differential $d^q:\Omega^q_U\to \Omega^{q+1}_U$
restricts to a differential $d^q:\Omega^q_{X}(\log D)(-D_\fm)\to \Omega^{q+1}_{X}(\log D)(-D_\fm)$, which induces a differential
\[d^q: \omega^q_{\fm,\nu}\to \omega^{q+1}_{\fm,\nu}.\]
If $t_\lambda\in \sO_{X}$ are local parameters of the $D_\lambda$, 
then  this differential is explicitly given by
\eq{restricted-log-differentials2}{d^q(t^\fm\otimes \omega)= 
t^\fm\otimes \bigg(d\omega+\sum_{\lambda\in \Lambda} m_\lambda \cdot \dlog(t_\lambda)\wedge\omega\bigg),}
where we write $t^\fm:=\prod_{\lambda\in \Lambda} t_\lambda^{m_\lambda}$.
We set
\eq{restricted-log-differentials3}{
 Z^q_{\fm,\nu}:=\Ker(\omega^q_{\fm,\nu}\xr{d^q}\omega^{q+1}_{\fm,\nu}), \quad
    B^q_{\fm,\nu}:=\im(\omega^{q-1}_{\fm,\nu}\xr{d^{q-1}}\omega^q_{\fm,\nu}). }
\end{no}

\begin{prop}\label{prop-rest-log-diff-gen-rel}
We keep the notation from above.
Set $\sM:=j_*(\sO_U^\times)\cap\sO_X$ and denote by $\sM^{\rm gp}$ the sheaf of groups on $X$ 
associated to the monoid $\sM$. Then 
there is a surjective morphism of $\sO_{D_\nu}$-modules
\[\sO_X(-D_\fm)_{|D_\nu}\otimes_{\Z}\bigwedge^q\sM^{\rm gp}\surj \omega^q_{\fm,\nu},\quad
     a\otimes x_1\wedge\ldots\wedge x_q\mapsto  a\otimes \dlog(x_1)\wedge \ldots\wedge \dlog(x_q).\]
With the notation from \eqref{restricted-log-differentials2} the kernel is the $\sO_{D_\nu}$-module 
which is locally generated by elements of the form
\[t^\fm\bar{a}\otimes a\wedge x_2\wedge\ldots\wedge x_{q}- 
                   \sum_i t^\fm\bar{u}_i\otimes u_i\wedge x_2\wedge\ldots\wedge x_{q},\]
for all $a, x_i\in \sM$ and $u_i\in \sO^\times_X$ satisfying $a=\sum_i u_i$ in $\sO_X$
and where $\bar{a},\bar{u}_i$ denote the images in $\sO_{D_\nu}$.
\end{prop}
\begin{proof}
This follows directly from the definition of $\omega^q_{\fm,\nu}$ and the description
of logarithmic differentials given in \cite[(1.7), p. 196]{Kato-log}.
\end{proof}

\begin{prop}\label{prop:higher-gr-map}
We keep the notations from above. 
Let  $\fm=(m_\lambda)_{\lambda\in \Lambda}$ be an element in $\N^\Lambda$
and take $\nu\in \Lambda$, $r\ge 1$.
Assume $m_\nu\ge 1$. 
Then there is a natural surjection
            \eq{prop:higher-gr-map0}{\omega^{r-1}_{\fm, \nu}/ B^{r-1}_{\fm,\nu}\surj \gr^{\fm,\nu}\sK^M_{r,X}}
        given by 
         \[{\text{class of }}( \bar{a}\otimes \dlog x_1\wedge\ldots\wedge \dlog x_{r-1})\mapsto 
                       {\text{class of }}\{1+a, x_1,\ldots, x_{r-1}\},\]
         where $x_i\in\sM^{\rm gp}$, $\bar{a}\in \sO_X(-D_\fm)_{|D_\nu}$  and $a\in\sO_X(-D_{\fm})$ 
          is a lift of $\bar{a}$.
\end{prop}
\begin{proof}
For $\bar{a}\in \sO_X(-D_\fm)_{|D_\nu}$ and $\ul{x}=(x_1,\ldots, x_{r-1})\in \prod_{i=1}^{r-1}\sM^{\rm gp}$
 define
\[\varphi(\bar{a},\ul{x}):= \text{ class of } \{1+a,\ul{x}\} \quad\text{in }\gr^{\fm,\nu}\sK^M_{r,X},\]
where $a\in \sO_{X}(-D_\fm)$ is some lift of $\bar{a}$.
Since $(1+a)(1+b)\equiv (1+a+b)$ mod  $1+\sO_{X}(-D_{\fm+\delta_\nu})$, for all  $a,b\in \sO_{X}(-D_{\fm})$,
this element is well-defined and induces a multilinear map 
$\varphi:\sO_{X}(-D_{\fm})_{|D_\nu}\oplus \bigoplus_{i=1}^{r-1}\sM^{\rm gp}\to \gr^{\fm,\nu}\sK^M_{r,X}$.
This also implies that if one of the $x_i$'s equals $-1$, then $\varphi(\bar{a},\ul{x})=0$.
Since $\{x,x\}=\{x,-1\}$ in $\sK^M_{2,X}$, the map $\varphi$ induces a surjective homomorphism
\[\sO_X(-D_\fm)_{|D_\nu}\otimes_{\Z}\bigwedge^{r-1}\sM^{\rm gp}\surj \gr^{\fm,\nu}\sK^M_{r,X}.\]
For $a\in\sM$ and $\ul{y}=(y_2,\ldots, y_{r})\in \prod_{i=1}^{r-2}\sM$ we have  
\begin{align}\label{prop:higher-gr-map1}
\varphi(t^\fm\bar{a},a,\ul{y}) &=-\varphi(t^\fm\bar{a}, -t^\fm,\ul{y})
        =-\varphi(t^\fm\bar{a}, t^\fm,\ul{y})- \varphi(t^\fm\bar{a}, -1,\ul{y})\\
       &=-\varphi(t^\fm\bar{a}, t^\fm,\ul{y}).\notag
\end{align}
For $a=\sum_i u_i$, with $u_i\in \sO_X^\times$, we get 
\[
\varphi(t^\fm\bar{a},a,\ul{y}) = -\varphi(t^\fm\bar{a}, t^\fm,\ul{y}) =
                                               \sum_i -\varphi(t^\fm\bar{u}_i, t^\fm,\ul{y})\notag\\
                                          = \sum_i\varphi(t^\fm \bar{u}_i, u_i,\ul{y}).\]
Hence by Proposition \ref{prop-rest-log-diff-gen-rel}, $\varphi$ factors through $\omega^{r-1}_{\fm,\nu}$.
It remains to show that $\varphi$ vanishes on $B^{r-1}_{\fm ,\nu}$.
It suffices to check this locally. Therefore it suffices to show that the boundary of
a form (with the obvious abuse of notation) $t^\fm a \dlog \ul{y}$, with either $a\in \sO^\times_X$ or 
$1+a\in \sO^\times_X$, is mapped to zero under $\varphi$.
Using the formula \eqref{restricted-log-differentials2} for the differential, we see that it suffices to show 
\[0=\begin{cases} \varphi(t^\fm \bar{a}, a,\ul{y})+
             \sum_{\lambda}\varphi(m_\lambda t^\fm \bar{a}, t_\lambda,\ul{y}),& \text{if }a\in \sO^\times_X,\\
             \varphi(t^\fm (\ol{1+a}), (1+a),\ul{y})+
            \sum_{\lambda}  \varphi(m_\lambda t^\fm \bar{a}, t_\lambda,\ul{y}),& \text{if }1+a\in \sO^\times_X.
                  \end{cases}\] 
In case $a\in \sO^\times_X$, this vanishing follows directly from  \eqref{prop:higher-gr-map1}.
In case $1+a\in\sO^\times_X$, we observe that $\varphi(t^\fm, t^\fm,\ul{y})=0$ and hence
\begin{align*}
\varphi(t^\fm (\ol{1+a}), (1+a),\ul{y})& =-\varphi(t^\fm (\ol{1+a}), t^\fm,\ul{y})
= -\varphi(t^\fm, t^\fm,\ul{y})-\varphi(t^\fm \bar{a}, t^\fm,\ul{y})\\
& =-\varphi(t^\fm \bar{a}, t^\fm,\ul{y}),
\end{align*}
which yields the promised vanishing in this case.
\end{proof}

\begin{prop}\label{prop:higher-gr-map-char0}
Assume $m_\nu\ge 1$ and that $k$ has either characteristic $0$ or prime to $m_\nu$.
Then  the map \eqref{prop:higher-gr-map0}
is an isomorphism.
\end{prop}
\begin{proof}
For $r= 1$ the statement holds by definition. For $r\ge 2$ we have 
\eq{prop:higher-gr-map-char01}{B^{r-1}_{\fm,\nu}=Z^{r-1}_{\fm,\nu}}
by \cite[Lem 6.2]{BS14} (here we use that either ${\rm char}(k)=0$ or $({\rm char}(k), m_\nu)=1$).
We have a well-defined map
\[K^M_r(k(\eta))\to \Omega^r_{k(\eta)},\quad 
              \{a_1,\ldots, a_r\}\mapsto \dlog(a_1)\wedge\ldots\wedge \dlog(a_r).\]
This clearly induces a map $\sK^M_{r,X|D_{\fm}}\to \Omega^r_{X}(\log D_{\rm red})(-D_\fm)$.
We obtain a well-defined map
\eq{prop:higher-gr-map-char02}{\gr^{\fm,\nu}\sK^M_{r,X}\lra 
\frac{\Omega^r_{X}(\log D_{\rm red})(-D_\fm)}{\Omega^r_{X}(\log D_{\rm red})(-D_{\fm+\delta_\nu})}
  =\omega^r_{\fm,\nu}.}
The composition 
\eq{prop:higher-gr-map-char03}{\omega^{r-1}_{\fm,\nu}\xr{\eqref{prop:higher-gr-map0}} 
     \gr^{\fm,\nu}\sK^M_{r,X}  \xr{\eqref{prop:higher-gr-map-char02}} \omega^{r}_{\fm,\nu}}
is equal to the differential \eqref{restricted-log-differentials2}.
Indeed, under this composition a local section $t^\fm \bar{a}\dlog\ul{x}\in\omega^{r-1}_{\fm,\nu}$ is sent to 
\begin{align*}
\dlog(1+t^\fm \bar{a})\dlog \ul{x} &= 
       \frac{t^\fm}{1+t^\fm \bar{a}} (da+\sum_\lambda m_\lambda a \dlog t_\lambda)\wedge\dlog(\ul{x})\\
 & =t^\fm (1-t^\fm \bar{a}) (da+\sum_\lambda m_\lambda a \dlog t_\lambda)\wedge\dlog(\ul{x})\\
& =t^\fm  (da+\sum_\lambda m_\lambda a \dlog t_\lambda)\wedge\dlog(\ul{x}).
\end{align*}
Hence the statement follows from \eqref{prop:higher-gr-map-char01}.
\end{proof}

\begin{thm}\label{thm:Cartier-mod}
Assume $k$ has characteristic $p>0$. Let the notation be as above and let $\fm'\in \N^\Lambda$ be the
smallest tuple with $p\cdot\fm'\ge \fm$. 
Assume $p|m_\nu$. Then the inverse Cartier  operator induces an isomorphism 
\[C^{-1}_{\fm,\nu}:\omega^q_{\fm',\nu}\xr{\simeq} \sH^q(\omega^\bullet_{\fm,\nu}),\] 
\[ a\otimes \dlog x_1\wedge \ldots\wedge \dlog x_q\mapsto a^p\otimes \dlog x_1\wedge \ldots\wedge \dlog x_q,\]
where $a\in \sO_X(-D_{\fm'})_{|D_\nu}$ and $x_i\in \sM$.
\end{thm}
\begin{proof}
This is proven in \cite[Thm 3.2]{Kato-Saito-Sato} (in a slightly different situation).
For the reader's convenience we give the proof.
By \cite[Lem 6.2]{BS14} (which is \cite[Lem 3.4]{Kato-Saito-Sato})
$\omega^\bullet_{\fn,\mu}$ is acyclic if $(n_\mu,p)=1$,
for all $\fn=(n_\lambda)\in\N^\Lambda$ and $\mu\in \Lambda$.
By the special choice of $\fm'$ we see that the natural inclusion
\eq{thm:Cartier-mod1}{\Omega^\bullet_X(\log D)(-D_{p\fm'})\inj \Omega^\bullet_X(\log D)(-D_{\fm})}
is a quasi-isomorphism. Indeed we can refine this inclusion to a filtration whose
graded pieces are of the form $\omega^\bullet_{\fn,\mu}$ as above.
We have $p\cdot(\fm'+\delta_\nu)\ge \fm+\delta_\nu$ and since $p|m_\nu$ the tuple $\fm'+\delta_\nu$ is 
minimal with this property. Thus if  we replace in \eqref{thm:Cartier-mod1} $\fm'$ by $\fm'+\delta_\nu$ and 
$\fm$ by $\fm+\delta_\nu$, we again get a quasi-isomorphism.
This yields a distinguished triangle in $D^b(X_\Zar)$
\[\Omega^\bullet_X(\log D)(-D_{p(\fm'+\delta_\nu)})\to \Omega^\bullet_X(\log D)(-D_{p\fm'})\to
\omega^\bullet_{\fm,\nu}\xr{[1]}.\]
Let $F:X\to X$ be the absolute Frobenius. The classical Cartier isomorphism
(see e.g. \cite[Thm (7.2)]{KatzNil}) gives an isomorphism of $\sO_X$-modules
\[C^{-1}:\Omega^q_X(\log D_{\rm red})\xr{\simeq} \sH^q(F_*\Omega^\bullet_X(\log D)).\]
Twisting this with $\sO_X(-D_{\fm'})$ yields an isomorphism
\[\Omega^q_X(\log D)(-D_{\fm'})\xr{\simeq} 
\sH^q(F_*(\Omega^\bullet_X(\log D)(-D_{p\fm'}))).\]
Using the triangle from above we get the following commutative diagram of abelian sheaves for all $q\ge 0$
\[\xymatrix{0\ar[r] & \Omega^q_{X|D_{\fm'+\delta_\nu}}\ar[r]\ar[d]_{\simeq}^{ C^{-1}} &
                              \Omega^q_{X|D_{\fm'}}\ar[r]\ar[d]_{\simeq}^{ C^{-1}} &
                             \omega^q_{\fm',\nu}\ar[d]\ar[r] & 0 &\\
              \cdots\ar[r]  &  \sH^q(\Omega^\bullet_{X|D_{p(\fm'+\delta_\nu)}})\ar[r]
               & \sH^q(\Omega^\bullet_{X|D_{p\fm'}})\ar[r]
            &\sH^q(\omega^\bullet_{\fm,\nu})\ar[r] &\cdots,
}\]
where we use the short hand notation $\Omega^\bullet_{X|D_\fm}=\Omega^\bullet_X(\log(D))(-D_\fm)$.
The statement follows.
\end{proof}

\begin{no}\label{higher-Bs-Zs}
With the notations above write $m_\nu=p^s\cdot m_\nu'$ with $s\ge 0$ and $(p,m_\nu')=1$. 
We inductively define sheaves of subgroups on $D_\nu$
\[B^q_{r,\fm,\nu},\, Z^q_{r,\fm,\nu}\subset \omega^q_{\fm,\nu},\quad \text{for }r\in[1, s+1], q\ge 0,\]
by the formulas
\[B^q_{1,\fm,\nu}:= B^q_{\fm,\nu},\quad Z^q_{1,\fm,\nu}:=Z^q_{\fm,\nu}\]
and 
\[B^q_{r,\fm',\nu}\xrightarrow[\simeq]{C_{\fm,\nu}^{-1}} B^q_{r+1,\fm,\nu}/B^q_{\fm,\nu},\quad
  Z^q_{r,\fm',\nu}\xrightarrow[\simeq]{C_{\fm,\nu}^{-1}} Z^q_{r+1,\fm,\nu}/B^q_{\fm,\nu}, 
 \quad r\in [1,s].\]
We obtain a chain of inclusions
\[B^q_{\fm,\nu}=B^q_{1,\fm,\nu}\subset\ldots\subset B^q_{s+1,\fm,\nu}\subset Z^q_{s+1,\fm,\nu}\subset
     \ldots \subset Z^q_{1,\fm,\nu}=Z^q_{\fm, \nu}\subset \omega^q_{\fm,\nu}.\]
\end{no}

\begin{prop}\label{prop:Bs-ZsCM}
For $\fm\in \N^\Lambda$, $\nu\in \Lambda$, with $m_\nu\ge 1$,
$q\ge 0$ and  $T\subset D_\nu$ a closed subset of codimension $c$ we have
\[\sH^i_T(\omega^q_{\fm,\nu}/B^q_{\fm,\nu})=0=\sH^i_T(\omega^q_{\fm,\nu}/Z^q_{\fm,\nu}),
\quad \text{for all }i<c.\]
Furthermore if $k$ has characteristic $p>0$ and $m_\nu= p^s m_\nu'$, with $s\ge 0$ and $(m_\nu',p)=1$,
then also
\eq{prop:Bs-ZsCM0}
{\sH^i_T(\omega^q_{\fm,\nu}/B^q_{r,\fm,\nu})=0=\sH^i_T(\omega^q_{\fm,\nu}/Z^q_{r,\fm,\nu}),
\quad \text{for all }i<c,\text{ and } r\in [1,s+1].}
\end{prop}
\begin{proof}
First we observe that $\omega^q_{\fm,\nu}$ is a locally free sheaf on the regular scheme $D_\nu$.
Hence $\sH^i_T(\omega^q_{\fm,\nu})=0$, for all $i<c$ and $q\ge 0$.
Set $p:={\rm char}(k)$.
Now assume either $p=0$ or  $p>0$ and $(m_\nu,p)=1$.
By \cite[Lem 6.2]{BS14} we have $B^q_{\fm ,\nu}=Z^q_{\fm,\nu}$, for all $q\ge 0$.
Therefore the exact sequence 
\eq{prop:Bs-ZsCM1}{0\to \omega^q_{\fm,\nu}/Z^q_{\fm,\nu}\to \omega^{q+1}_{\fm,\nu}
               \to \omega^{q+1}_{\fm,\nu}/B^{q+1}_{\fm,\nu}\to 0}
yields,  for all $i<c$,
\[\sH^{i-1}_T(\omega^{q+1}_{\fm,\nu}/Z^{q+1}_{\fm,\nu})= 
   \sH^{i-1}_T(\omega^{q+1}_{\fm,\nu}/B^{q+1}_{\fm,\nu})=
    \sH^i_{T}(\omega^q_{\fm,\nu}/Z^q_{\fm,\nu})= 
 \sH^i_T(\omega^q_{\fm,\nu}/B^q_{\fm,\nu}). \]
Since $\omega^q_{\fm,\nu}=0$ for $q>>0$ we get by descending induction over $q$ that 
\[\sH^i_{T}(\omega^q_{\fm,\nu}/Z^q_{\fm,\nu})= \sH^i_T(\omega^q_{\fm,\nu}/B^q_{\fm,\nu})=0,\quad
 \text{for all }q\ge 0,\, i<c.\]
In particular the statement is proven if $p=0$. Furthermore, if $p>0$, then \eqref{prop:Bs-ZsCM0} is proven in the case $s=0$.
To finish the proof we assume $p>0$ and $s\ge 1$. 
Let $\fm'\in \N^\Lambda$ be the smallest tuple such that $p\cdot\fm'\ge \fm$. By induction on $s$ we have 
$\sH^i_T(\omega^q_{\fm',\nu}/B^q_{r,\fm',\nu})=\sH^i_T(\omega^q_{\fm',\nu}/Z^q_{r,\fm',\nu})=0$, 
for all $r\in[1,s]$, $q\ge 0$ and $i<c$.
An application of the Cartier operator yields 
\eq{prop:Bs-ZsCM2}{\sH^i_T(Z^q_{\fm,\nu}/B^q_{r+1,\fm,\nu})=
          \sH^i_T(Z^q_{\fm,\nu}/Z^q_{r+1,\fm,\nu})=0,\quad \text{for all }r\in[1,s], \,q\ge 0,\, i<c.}
Now assume that the vanishing \eqref{prop:Bs-ZsCM0} holds for $q+1$; we want to show that it also holds for $q$.
The exact sequence \eqref{prop:Bs-ZsCM1} gives the vanishing $\sH^i(\omega^q_{\fm,\nu}/Z^q_{\fm,\nu})$
for all $i<c$. Therefore the exact sequence
\[0\to Z^q_{\fm,\nu}/ B^q_{r+1,\fm,\nu}\to \omega^q_{\fm,\nu}/ B^q_{r+1,\fm,\nu}\to 
      \omega^q_{\fm,\nu}/Z^q_{\fm,\nu}\to 0\]
together with \eqref{prop:Bs-ZsCM2} yields 
\[\sH^i_T(\omega^q_{\fm,\nu}/ B^q_{r+1,\fm,\nu})=0,\quad \text{for all } i<c, \,r\in [1,s].\]
Similarly we also get
\[\sH^i_{T}(\omega^q_{\fm,\nu}/ Z^q_{r+1,\fm,\nu})=0,\quad \text{for all } i<c,\, r\in [1,s].\]
Finally the exact sequence
\[0\to \omega^q_{\fm',\nu}\xr{C^{-1}_{\fm,\nu}} \omega^q_{\fm,\nu}/B^q_{\fm,\nu}\to 
\omega^q_{\fm,\nu}/Z^q_{\fm,\nu}\to 0\]
yields $\sH^i_T(\omega^q_{\fm,\nu}/B^q_{\fm,\nu})=0$, for all $i< c$.
This finishes the proof.
\end{proof}

\begin{rmk}\label{rmk:Bs-Zs-free}
One can show that $\omega^q_{\fm,\nu}/B^q_{r,\fm,\nu}$ and $\omega^q_{\fm,\nu}/Z^q_{r, \fm,\nu}$
are locally free $\sO_{D_\nu}$-modules, where the $\sO_{D_\nu}$-module structure is induced by
the one from $F_{X*}^r \omega^q_{\fm,\nu}$, where $F_X: X\to X$ is the absolute Frobenius
(cf. \cite[0, Prop 2.2.8]{IlDRW}). This immediately implies \eqref{prop:Bs-ZsCM0}.
\end{rmk}

\begin{thm}\label{thm:higher-gr-map-poschar}
Assume $k$ has characteristic $p>0$ and $m_\nu=p^s m_{\nu}'$, with $s\ge 0$ and
$(m_\nu',p)=1$. Then the map \eqref{prop:higher-gr-map0} factors to give an 
isomorphism
\eq{thm:higher-gr-map-poschar1}{\omega^{r-1}_{\fm,\nu}/B^{r-1}_{s+1,\fm,\nu}\xr{\simeq}
            \gr^{\fm,\nu}\sK^M_{r,X}.}
\end{thm}
\begin{proof}
For $s=0$ this is Proposition \ref{prop:higher-gr-map-char0}.
Now assume $s\ge 1$ and take $\fm\in\N^\Lambda$ minimal with $p\cdot\fm'\ge \fm$.
Clearly the multiplication with $p$ on $\sK^M_{r,X}$ induces maps
\[\sK^M_{r, X|D_{\fm'}}\xr{\cdot p} \sK^M_{r, X|D_\fm}, \quad 
                                 \sK^M_{r, X|D_{\fm'+\delta_\nu}}\xr{\cdot p} \sK^M_{r, X|D_{\fm+\delta_\nu}}.\]
It is direct to check that we obtain a commutative diagram
\[\xymatrix{  
0\ar[r] &B^{r-1}_{s,\fm',\nu}\ar[r]\ar[d]^{C^{-1}_{\fm,\nu}}_{\simeq} & 
 \omega^{r-1}_{\fm',\nu}\ar[d]^{C^{-1}_{\fm,\nu}}\ar[r]^{\eqref{prop:higher-gr-map0}} &
   \gr^{\fm',\nu}\sK^M_{r,X}\ar[d]^{\cdot p}\ar[r] & 0\\
      0\ar[r] &B^{r-1}_{s+1,\fm,\nu}/B^{r-1}_{\fm,\nu}\ar[r] & 
     \omega^{r-1}_{\fm,\nu}/B^{r-1}_{\fm,\nu}\ar[r]_{\eqref{prop:higher-gr-map0}}&
   \gr^{\fm,\nu}\sK^M_{r,X}\ar[r] & 0. }\]
By induction on $s$, the upper horizontal sequence is exact and we have to show that so is the lower one.
Clearly, the lower sequence is exact on the left and on the right.  The exactness of the upper sequence implies
that the lower is a complex.  It remains to show that the induced map \eqref{thm:higher-gr-map-poschar1}
is injective. By Proposition \ref{prop:Bs-ZsCM} it suffices to check this at the generic point $\eta_\nu$ of $D_\nu$.
Since the composition \eqref{prop:higher-gr-map-char03} is equal to the differential,
the kernel of $\omega^{r-1}_{\fm,\nu}/B^{r-1}_{\fm,\nu}\to \gr^{\fm,\nu}\sK^M_{r,X}$
is contained in $Z^{r-1}_{\fm,\nu}/B^{r-1}_{\fm,\nu}$ and therefore lies in the image of $C^{-1}_{\fm,\nu}$.
Thus it remains to show that the map
$(\gr^{\fm',\nu}\sK^M_{r,X})_{\eta_\nu}\xr{\cdot p} (\gr^{\fm,\nu}\sK^M_{r,X})_{\eta_\nu}$ is injective.
Set $A:=\sO_{X,\eta_\nu}$, $K:={\rm Frac}(A)$ and write $m_\nu=p m$.
Then $A$ is a DVR which is essentially smooth over $k$ and 
we have to show that the map
\eq{thm:higher-gr-map-poschar2}{
    U^m K^M_r(K)/U^{m+1} K^M_r(K)\xr{\cdot p} U^{pm} K^M_r(K)/U^{pm+1} K^M_r(K)   }
induced by multiplication with $p$ is injective (here we use the notations of \ref{U-for-DVR}).
To this end we may replace $A$ and $K$ by their completions, where now $A$ is formally smooth over $k$,
see Lemma \ref{lem:DVR-cDVR}. Denote by $K_0$ the residue field of $A$; it is separable over $k$ 
(since $D_\nu$ is smooth over $k$). By \cite[Thm (19.6.4)]{EGAIV1} there is an isomorphism of $k$-algebras
$A\cong K_0[[t]]$; hence $K\cong K_0((t))$. Therefore the injectivity of \eqref{thm:higher-gr-map-poschar2}
follows from Corollary \ref{cor-mult-p-gradedK}, proven later independently. This finishes the proof.
\end{proof}

\begin{cor}\label{cor:higher-gr-map}
Let $k$ be a field of characteristic $p\ge 0$ and assume  $m_\nu\ge 1$. 
Set
\[s:=\begin{cases} 0, &\text{if } p=0\\
                            v_p(m_\nu), & \text{if }p>0,\end{cases}\]
where $v_p:\Q\to \Z$ is the $p$-adic valuation.
Then there is a distinguished triangle in $\sD^b(X_\Zar)$ (with the notation from \ref{Nis-to Zar})
\[R\e_*(\sK^M_{r, X|D_{\fm+\delta_\nu}, {\rm Nis}})\to R\e_*(\sK^M_{r, X|D_\fm, {\rm Nis}})\to
     \omega^{r-1}_{\fm,\nu}/B^{r-1}_{s+1,\fm,\nu}\xr{[1]}.\]
Furthermore, the canonical map
\[\gr^{\fm,\nu}\sK^M_{r,X}\xr{\simeq} R\e_*\gr^{\fm,\nu}\sK^M_{r,X,\Nis}\]
is an isomorphism.
\end{cor}

\begin{proof}
The assignment 
\[X_\Nis\ni(v:V\to X)\mapsto H^0(V, v^*i_{\nu*}\omega^q_{\fm,\nu})\]
defines a sheaf on $X_\Nis$ which we denote by $\omega^q_{\fm,\nu,\Nis}$.
We define sheaves on $X_\Nis$ by
\[Z^q_{\fm,\nu,\Nis}:=\Ker(\omega^q_{\fm,\nu,\Nis}\xr{d^q}\omega^{q+1}_{\fm,\nu,\Nis}), \quad
    B^q_{\fm,\nu,\Nis}:=\im(\omega^{q-1}_{\fm,\nu,\Nis}\xr{d^{q-1}}\omega^q_{\fm,\nu,\Nis}).\]
Furthermore, if $p>0$, then the Cartier isomorphism from Theorem \ref{thm:Cartier-mod}
induces an isomorphism 
$C^{-1}_{\fm,\nu,\Nis}: \omega^q_{\fm,\nu,\Nis}\xr{\simeq} \sH^q(\omega^\bullet_{\fm,\nu,\Nis})$
and we can define the sheaves $Z^q_{r,\fm ,\nu,\Nis}$ and $B^q_{r,\fm,\nu,\Nis}$ as in 
\ref{higher-Bs-Zs}.  Proposition \ref{prop:higher-gr-map-char0} and Theorem \ref{thm:higher-gr-map-poschar}
yield an isomorphism between sheaves on $X_\Nis$
\[\omega^{r-1}_{\fm,\nu,\Nis}/B^{r-1}_{s+1,\fm,\nu,\Nis}\xr{\simeq}
            \gr^{\fm,\nu}\sK^M_{r,X,\Nis}.\]
Therefore it suffices to show that the natural map
\eq{cor:higher-gr-map1}{\omega^{r-1}_{\fm,\nu}/B^{r-1}_{s+1,\fm,\nu}\to 
             R\e_*(\omega^{r-1}_{\fm,\nu,\Nis}/B^{r-1}_{s+1,\fm,\nu,\Nis})}
is an isomorphism. To this end we note that for a quasi-coherent sheaf $E$ on $X$
we have $R\e_*E_{\Nis}=E$, where $E_\Nis$ is the Nisnevich sheaf 
$X_\Nis\ni (v:V\to X)\mapsto H^0(V, v^*E)$ (cf. \cite[III, Prop 3.7]{MilneEt}).
If $p>0$ and $F_X$ denotes the absolute Frobenius on $X$, 
then $\omega^{r-1}_{\fm,\nu}/B^{r-1}_{s+1,\fm,\nu}$ is a  quotient of the quasi-coherent $\sO_X$-module
$F^{s+1}_{X*}\omega^{r-1}_{\fm,\nu}$ and hence is quasi-coherent.
We get \eqref{cor:higher-gr-map1} in this case. 
If $p=0$ we have the natural isomorphism $\omega^{q}_{\fm,\nu}\cong R\e_*\omega^{q}_{\fm,\nu,\Nis}$,
for all $q\ge 0$. Furthermore, $Z^q_{\fm,\nu,\Nis}\cong B^q_{\fm,\nu,\Nis}$, 
see \eqref{prop:higher-gr-map-char01}. Hence descending induction on $q$ and the exact sequence on $X_\Nis$
\[0\to Z^q_{\fm,\nu\Nis}\to\omega^q_{\fm,\nu\Nis}\to B^{q+1}_{\fm,\nu,\Nis}\to 0\] 
give
\[R\e_*Z^q_{\fm,\nu,\Nis}\cong Z^q_{\fm,\nu}\cong B^q_{\fm,\nu}\cong 
R\e_*B^q_{\fm,\nu,\Nis}.\]
The isomorphism \eqref{cor:higher-gr-map1} follows.
\end{proof}

\begin{cor}\label{cor:relK-smDiv}
Assume that $D_{\rm red}$ is smooth ($D=0$ is allowed). Then the natural map
\eq{cor:relK-smDiv1}{\sK^M_{r, X|D}\xr{\simeq} R\e_*\sK^M_{r, X|D,\Nis}}
is an isomorphism.
\end{cor}
\begin{proof}
By \ref{K-Nis} and \cite[Thm 5.1, 2.]{VoDM} the natural map 
 $\sK^M_{r,D_{\rm red}}\to R\e_*\sK^M_{r,D_{\rm red},\Nis}$ is an isomorphism.
Thus by Proposition \ref{prop:gr0-map} and Corollary \ref{cor:higher-gr-map}, the natural maps
$\gr^{\fm,\nu}\sK^M_{r,X}\to R\e_*\gr^{\fm,\nu}\sK^M_{r,X,\Nis}$ are isomorphism for all
$\fm$, $\nu$, $r$. Hence the statement.
\end{proof}

\subsection{The Cousin resolution of relative Milnor {$K$}-sheaves}
\begin{thm}\label{thm:relK-is-CM}
Let $D$ be an effective divisor on $X$ and assume that $D_{\rm red}$ is a simple normal crossing divisor.
Then for all closed subschemes $T\subset X$ of codimension $c$, and for all $i<c$, we have 
\[\sH^i_T(R\e_*\sK^M_{r, X|D,\Nis})=0.\]
\end{thm}
\begin{proof}
Corollary \ref{cor:rel-K-in-K} and Corollary \ref{cor:relK-smDiv} (for $D=0$) give a distinguished triangle
in $\sD^b(X_{\Zar})$ 
\[R\e_*\sK^M_{r,X|D,\Nis}\to \sK^M_{r,X}\to R\e_*(\sK^M_{r,X}/\sK^M_{r, X|D,\Nis})\xr{[1]}.\]
By the exactness of the Gersten resolution \ref{Gersten} we have 
$\sH^i_T(\sK^M_{r,X})=0$, for all $i<c$. Hence it suffices to show that 
    $\sH^{i-1}_T(R\e_*(\sK^M_{r,X}/\sK^M_{r, X|D,\Nis}))=0$, for all $i<c$.
With the notation from \ref{gr-rel-K} we have
\[\sH^{i-1}_T(R\e_*\gr^{\fm,\nu}\sK^M_{r,X,\Nis})=
          \sH^{i-1}_{T\cap D_\nu}(R\e_*\gr^{\fm,\nu}\sK^M_{r,X,\Nis}).\]
Since $c-1\le\codim(T\cap D_\nu, D_\nu)$ it follows from Corollary \ref{cor:gr0-map} together 
with induction on the dimension of $X$, Corollary \ref{cor:higher-gr-map} 
and Proposition \ref{prop:Bs-ZsCM} that these groups vanish for $i<c$.
Now the theorem follows since $\sK^M_{r,X}/\sK^M_{r, X|D,\Nis}$ is a successive extension of the sheaves
$\gr^{\fm,\nu}\sK^M_{r,X,\Nis}$.
\end{proof}

\begin{no}\label{CousinComplex-relK}
{\em The Cousin complex.}
Let $D$ be an effective divisor on $X$. We denote by $C^\bullet_{r, X|D}$ the Cousin complex of 
$\sK^{M}_{r, X|D}$ (see \cite[IV, 2.]{Ha}). It has the following shape (with the notation from \ref{cech-symbol})
\mlnl{C^\bullet_{r, X|D}: i_{\eta*}H^0_\eta(\sK^M_{r, X|D})\to
                              \bigoplus_{x\in X^{(1)}} i_{x*}H^1_x(\sK^M_{r, X|D})\to
          \ldots \\ \to\bigoplus_{x\in X^{(i)}} i_{x*}H^i_x(\sK^M_{r, X|D})\to\ldots .}
Here  $i_x: x\inj X$ denotes the immersion.  Similarly,
we denote by $C^{h,\bullet}_{r,X|D}$ the Cousin complex of $R\e_*\sK^M_{r, X|D,\Nis}$
\mlnl{C^{h,\bullet}_{r, X|D}: i_{\eta*}H^0_\eta(R\e_*\sK^M_{r, X|D,\Nis})\to
                              \bigoplus_{x\in X^{(1)}} i_{x*}H^1_x(R\e_*\sK^M_{r, X|D,\Nis})\to
          \ldots \\  \to\bigoplus_{x\in X^{(i)}} i_{x*}H^i_x(R\e_*\sK^M_{r, X|D,\Nis})\to\ldots .}
In particular these are complexes of flasque sheaves. The restriction of $C^\bullet_{r,X|D}$ to 
$U=X\setminus D$ equals  the Gersten resolution of $\sK^M_{r,U}$ by Corollary \ref{cor:MilnorK-coh-in-c}:
\eq{CousinComplex-relK0}{(C^\bullet_{r, X|D})_{|U}= C^\bullet_{r, U}.}
If furthermore $D_{\rm red}$ has  simple normal crossings, then by Corollary \ref{cor:relK-smDiv}
(for $(X,D)=(U,0)$) we also have
\[(C^{h,\bullet}_{r, X|D})_{|U}= C^\bullet_{r, U}.\]
The natural map $\sK^M_{r, X|D}\to R\e_*\sK^M_{r,X|D,\Nis}$ induces a natural map of complexes on $X_\Zar$
\eq{CousinComplex-relK1}{C^\bullet_{r,X|D}\to C^{h,\bullet}_{r, X|D}.}

Finally we give an alternative description of the terms appearing in $C^{h,\bullet}_{r,X|D}$:
If $Z\subset X$ is closed we have $R\e_*R\ul{\Gamma}_Z=R\ul{\Gamma}_ZR\e_*$, 
by \cite[V, Prop 4.9, Prop 4.11]{SGA4II}. Hence for $x\in X^{(c)}$ we have
\[H^c_x(R\e_*\sK^M_{r, X|D,\Nis})=\varinjlim_{x\in V} H^c_{\ol{x}\cap V}(V_{\Nis}, \sK^M_{r,X|D,\Nis}),\]
where the limit ranges over all Zariski open neighborhoods $V\subset X$ of $x$. 
Let $X_{(x)}^h=\Spec \sO_{X,x}^h$ be the henselization of $X$ at $x$ and denote by $i^h_x: X^h_{(x)}\to X$
the canonical map. Then the above together with \cite[1.27 and 1.29.3]{Nis} yield
\[H^c_x(R\e_*\sK^M_{r, X|D,\Nis})= H^c_x(X^h_{(x),\Nis}, (i_x^h)^{-1}\sK^M_{r, X|D,\Nis}).\]
\end{no}

\begin{cor}\label{cor:CousinComplex-relK}
Assume that $D_{\rm red}$ has simple normal crossings. Then there is an isomorphism
\[R\e_*\sK^M_{r, X|D,\Nis}\xr{\simeq} C^{h,\bullet}_{r, X|D}\quad \text{in }\sD^b(X_\Zar).\]
Furthermore if $D_{\rm red}$ is smooth the natural morphisms
\[\sK^M_{r,X|D}\to C^\bullet_{r,X|D}\xr{\eqref{CousinComplex-relK1}} C^{h,\bullet}_{r,X|D}\]
are quasi-isomorphism of complexes. 
\end{cor}
\begin{proof}
The first part follows from Theorem \ref{thm:relK-is-CM} and \cite[IV, Prop 3.1]{Ha}, the second part from the first 
and Corollary \ref{cor:relK-smDiv}.
\end{proof}

\begin{no}\label{CousinComplex-relK-Nis}
{\em The Cousin Complex in the Nisnevich topology.}
Consider the presheaf of complexes
\[X_\Nis\ni (v: V\to X)\mapsto \Gamma(V, C^h_{r, V|v^*D}).\]
The explicit description of $C^{\bullet,h}_{r, X|D}$ in \ref{CousinComplex-relK} above and 
excision for local Nisnevich cohomology (see \cite[1.27 Thm]{Nis}) implies that this presheaf 
is a sheaf of complexes on $X_\Nis$, which we denote by 
\[C^\bullet_{r, X|D,\Nis}.\]
By construction there are natural maps of complexes $\sK^M_{r, X|D}(V)\to C^\bullet_{r, X|D,\Nis}(V)$
where we use the notation from Definition \ref{defn:MilnorModulus},\eqref{defn:MilnorModulus(2)}.
This yields a morphism 
\eq{CousinComplex-relK-Nis1}{\sK^M_{r, X|D,\Nis}\to C^\bullet_{r, X|D,\Nis}\quad \text{on } X_{\Nis}.}
\end{no}

\begin{cor}\label{cor:CousinComplex-relK-Nis}
Assume $D_{\rm red}$ has simple normal crossings. Then 
\eqref{CousinComplex-relK-Nis1} is a quasi-isomorphism.
\end{cor}
\begin{proof}
It suffices to show that for all  \'etale maps $v: V\to X$ and all points $y\in V$ the Nisnevich stalk
$\sH^i(C^\bullet_{r,X|D,\Nis})_y^{h}$ (defined as in \eqref{defn:MilnorModulus3}) 
vanishes for $i\ge 1$ and is isomorphic to $\sK^{M,h}_{r, X|D,y}$ for $i=0$.
This follows directly from Corollary \ref{cor:CousinComplex-relK}.
\end{proof}

\subsection{Pushforward for projections from projective space}
\begin{no}\label{pf}Let $f:Y\to Z$ be a proper morphism between equidimensional finite type $k$-schemes.
              Set $e=\dim Y-\dim Z$. Then there is a morphism of complexes
\[f_*: f_*C^\bullet_{r+e,Y}[e]\to C^\bullet_{r,Z}.\]
See e.g. \cite[Prop 4.6, (1)]{Rost}. (Also notice that the complexes $C^\bullet_{r,Y}$ are  defined 
if $Y$ is not smooth, see e.g. \cite[5]{Rost}.)
If $Y$ and $Z$ are smooth, then this map induces a morphism in the derived category
\[f_*: Rf_*\sK^M_{r+e,Y}[e]\to \sK^M_{r,Z}.\]
\end{no}

\begin{no}\label{chern}
Let $Y$ be a smooth scheme and denote by $\pi: \P^N_Y\to Y$ the projection.
Denote by 
\[c_1(O(1))\in R^1\pi_*\sO_{\P^N_Y}^\times\]
the first Chern class of $\sO_{\P^N_Y}(1)$ and  by
\[c_1(O(1))^i\in R^i\pi_*\sK^M_{i,\P^N_Y}, \quad i\in [0,N],\] 
its $i$-fold cup-product (by convention $c_1(O(1))^0=1\in \Z$).
Finally, 
\[\dlog (c_1(O(1)))^i\in R^i\pi_*\Omega^i_{\P^N_Y/Y}, \quad i\in [0,N],\] 
denotes the image of $c_1(O(1))^i$ under the map 
$\dlog:R^i\pi_*\sK^M_{i,\P^N_Y}\to R^i\pi_*\Omega^i_{\P^N_Y/Y}$. 
\end{no}

\begin{lem}\label{lem-vanishing-hdi-proj-space}
Let $D$ be an effective Cartier divisor on $X$ and assume that $D_{\rm red}$ is a simple normal crossing
divisor. Let $\{ D_\lambda\}_{\lambda\in \Lambda}$ be the union of the irreducible components of $D$.
For a scheme $Y$ set $P_Y:=\P^N_Y$ and denote by $\pi_Y: P_Y\to Y$ the projection. 
For $\fm\in \N^\Lambda$ and $\nu\in\Lambda$
and with the notation from \eqref{restricted-log-differentials1}
we have the sheaves $\omega^q_{X|D,\fm,\nu}$ on $D_\nu$ at our disposal
together with the subsheaves   $B^q_{X|D,r,\fm,\nu}$, for $r\ge 1$, as defined in 
\ref{higher-Bs-Zs}. (In characteristic 0, we set  $B^q_{X|D,r,\fm,\nu}:=B^q_{X|D,\fm,\nu}$ for all $r\ge 1$.)
Then for all $q\ge 0$ and $r\ge 1$ we have on $X_\Zar$
\[   R^i\pi_{D_\nu*}(\omega^q_{P_X|P_D,\fm,\nu}/B^q_{P_X|P_D,r,\fm,\nu})=
     0=R^i\pi_{X*}\sK^M_{q,P_X}, \quad 
\text{for all }i> N,\]
and for $i\in [0,N]$ there are natural isomorphisms
\eq{lem-vanishing-hdi-proj-space1}{
         -\cup c_1(O(1))^i:\sK^M_{q-i,X}\xr{\simeq}R^i\pi_{X*}\sK^M_{q, P_X}}
and
\[-\cup \dlog (c_1(O(1)))^i:\omega^{q-i}_{X|D,\fm,\nu}/B^{q-i}_{X|D,r,\fm,\nu}\xr{\simeq}
R^i\pi_{D_\nu*}(\omega^{q}_{P_X|P_D,\fm,\nu}/B^{q}_{P_X|P_D,r,\fm,\nu}),\]
induced by the  cup product with $(c_1(O(1)))^i$ and $\dlog (c_1(O(1)))^i$, respectively.
Furthermore the corresponding statement on $X_\Nis$ equally holds. 
\end{lem}
\begin{proof}
We have the exact sequence (see \eqref{lem:MilnorKsuppSNCD2})
\[0\to \sK^M_{q, P_X}\to j_*\sK^M_{q, A_X}\to \sK^M_{q-1, H_X} \to 0,\]
where $H_X\subset P_X$ is a hyperplane with complement $j: A_X\inj P_X$. 
Therefore the statement for $\sK^M_q$ follows by induction from the isomorphism 
\[\sK^M_{q,X}\xr{\simeq} R(\pi_X\circ j)_*\sK^M_{q,A_X}\cong R\pi_{X*}j_*\sK^M_{q, A_X}\]
where the first isomorphism is homotopy invariance (see \cite[Thm 3.1.12]{VoDM} together with \ref{K-Nis}) 
and the second comes from Corollary \ref{cor:MilnorKcompSNCD}.

Now we prove the statement for $\omega^q_{\fm,\nu}$.
Let $F\subset k$ be the prime subfield.  We have 
\[ \omega^{q}_{P_X|P_D,\fm,\nu} =
     \bigoplus_{j=0}^N \pi_{D_\nu}^{-1}(\omega^{q-j}_{X|D,\fm,\nu})
                                  \otimes_{F} \rho^{-1}\Omega^j_{P_F/F},\]
where $\rho: P_{D_\nu}=D_\nu\times_F P_F\to P_F$ is the projection.
This decomposition is compatible with the differential and the Cartier operator in the obvious sense.
We get
\[ \omega^{q}_{P_X|P_D,\fm,\nu}/B^{q}_{P_X,P_D,r,\fm,\nu} =
     \bigoplus_{j=0}^N \pi_{D_\nu}^{-1}(\omega^{q-j}_{X|D,\fm,\nu}/B^{q-j}_{X|D,r,\fm,\nu})
                                  \otimes_{F} \rho^{-1}(\Omega^j_{P_F/F}/B^j_{P_F, r}),\]
where $B^j_{P_F,r}$ is defined as in \ref{BK-graded-group} below.
In the following we write
$P:=P_F$ and $\Omega^q:=\Omega^q_{P_F/F}$ and $B^q_r:= B^q_{r,P_F}$, etc.
By the K\"unneth formula (see \cite[Thm (6.7.8)]{EGAIII2}) it suffices to show that
$H^i(P, \Omega^j/B^j_{r})=0$, for  $i\neq j$,  and 
that the cup product with $\dlog (c_1(\sO(1)))^j$ induces an isomorphism 
$F\xr{\simeq}H^i(P,\Omega^i/B^i_r)$, for $i\in [0,N]$.
This statement holds in the case $r=0$, where we set $B^j_0:=0$ (see e.g. \cite[Exp XI]{SGA7II}).
Hence it suffices to show
\[H^i(P, B^j_r)=0 \quad \text{for all }i, j,r.\]

If ${\rm char}(k)>0$, the vanishing for $r=1$ holds by \cite[Prop 1.4]{IlOrd}. 
For $r\ge 2$ the vanishing follows by induction from the  isomorphism $B^N_{r}\cong B^N_{r+1}/B^N_{1}$
which is induced by the inverse Cartier operator. In characteristic zero the statement follows from 
Lemma \ref{lem-DRP} below. 

Finally the Nisnevich case. In view of the definition of the corresponding Nisnevich sheaves,
see \ref{K-Nis} and the proof of Corollary \ref{cor:higher-gr-map} the statement for the Nisnevich sheaves
follows from the two facts which hold for any smooth $k$-scheme:
\begin{enumerate}
\item $H^i(X_\Nis,\sK^M_{r, X})=H^i(X_\Zar, \sK^M_{r, X})$ (see \ref{K-Nis} and \cite[Thm 3.1.12]{VoDM}).
\item $H^i(X_\Nis, \sF_\Nis)=H^i(X_\Zar, \sF)$, where $\sF$ is any quasi-coherent sheaf and $\sF_\Nis$
        its associated Nisnevich sheaf (cf. \cite[III, Prop 3.7]{MilneEt}).
\end{enumerate}
This finishes the proof of the lemma.
\end{proof}

\begin{lem}\label{lem-DRP}
Let $k$ be a field of characteristic zero. Set $P:=\P^N_k$.
Then 
\[H^i(P, \sH^j(\Omega^\bullet_{P/k}))=0,\quad  i\neq j,\]
and the cup-product with $\dlog(c_1(\sO(1)))^i$ induces an isomorphism 
\[k\xr{\simeq} H^i(P,\sH^i(\Omega^\bullet_{P/k})), \quad i\in [0,N].\]
Furthermore for $B^j:=\im (d:\Omega^{j-1}\to \Omega^j)$
and $Z^i:=\Ker(d:\Omega^j\to \Omega^{j+1})$ we have 
\[H^i(P,B^j)=0,\text{ all } i,j,\quad
H^i(P, Z^j)=0,\text{ all }i\neq j,\quad H^i(P, Z^i)\cong k, \text{ all } i\in [0,N].\]
\end{lem}

\begin{proof}
By \cite[(4.2)Thm and (2.2)]{BO} the Cousin complex of $\sH^j(\Omega^\bullet_{P/k})$ is a resolution.
Since de Rham cohomology in characteristic zero satisfies purity we get that for 
$H\subset P$ a hyperplane the complex $\ul{\Gamma}_H(\text{Cousin}(\sH^j(\Omega^\bullet_{P/k})))$
is isomorphic to the complex  $\text{Cousin}(\sH^{j-1}(\Omega^\bullet_{H/k}))$ shifted by $-1$, i.e. 
we have an isomorphism 
\[R\ul{\Gamma}_{H}\sH^j(\Omega^\bullet_{P/k})\cong \sH^{j-1}(\Omega^\bullet_{H/k})[-1]
\quad \text{in }\sD^b(P_\Zar).\]
Hence the long exact localization sequence looks like this
\eq{lem-DRP1}{\ldots\to H^{i-1}(H,\sH^{j-1}(\Omega^\bullet_{H/k}))\to 
       H^i(P,\sH^{j}(\Omega^\bullet_{P/k}))\to
     H^i(A, \sH^j(\Omega^\bullet_{A/k}))\to \ldots,}
where $A=P\setminus H$.  Furthermore the presheaf $X\mapsto H^j(X,\Omega^\bullet_{X/k})$ on $\Sm_{k}$
is a homotopy invariant pretheory (see \cite[3.4]{VoPST}) and hence so is its Zariski sheafification
$X\mapsto \Gamma(X,\sH^j(\Omega^\bullet_{X/k}))$ (see \cite[Prop 4.26]{VoPST}). Hence
\cite[Thm 4.27]{VoPST} implies
\[H^i(A, \sH^j(\Omega^\bullet_{A/k}))=0, \quad \text{for all } (i,j)\neq (0,0),\quad \text{and}\quad 
   H^0(A, \sH^0(\Omega^\bullet_{A/k}))=k.\]
The first two statements of the lemma are direct consequences of this, the exact sequence \eqref{lem-DRP1}
and induction.

We prove the last statement.  Observe that
the natural maps $H^i(P, Z^j)\to H^i(P, \Omega^j)$ and $H^i(P, \sH^j(\Omega^\bullet))$
are surjective for all $i,j$. (Clearly for $i\neq j$ and for $i=j$ it follows from the fact that
the isomorphism $k\cong H^i(P,\Omega^i)$ and $k\cong H^i(P, \sH^i(\Omega^\bullet))$
both given by the cup-product with $\dlog(c_1(\sO(1)))^i$ factor over $H^i(P, Z^i)$.)
We obtain short exact sequences for all $i,j$
\[0\to H^i(P, B^j)\to H^i(P, Z^j)\to H^i(X, \sH^j(\Omega^\bullet))\to 0\]
and 
\[0\to H^i(P, B^{j+1})\to H^{i+1}(P, Z^j )\to H^{i+1}(P,\Omega^j)\to 0.\]
The last statement of the lemma follows directly from this via descending induction
over $i$.
\end{proof}

\begin{lem}\label{lem:pf-MK} 
We keep the notations from above and set $\pi:=\pi_X$. Then the pushforward 
$\pi_*:R\pi_*\sK^M_{r+N,P_X}[N]\to \sK^M_{r,X}$ from \ref{pf} is equal to the
composition of the canonical map $R\pi_*\sK^M_{r+N,P_X}[N]\to R^N\pi_*\sK^M_{r+N, P_X}$
with the inverse of the isomorphism \eqref{lem-vanishing-hdi-proj-space1} (for $(i,q)=(N, r+N)$).
\end{lem}
\begin{proof}
Notice that there is a canonical map $R\pi_*\sK^M_{r+N,P_X}[N]\to R^N\pi_*\sK^M_{r+N, P_X}$ by the
vanishing statement of Lemma \ref{lem-vanishing-hdi-proj-space}. We have to show that
the pushforward $\pi_* : R^N\pi_{*}\sK^M_{r+N, P_X}\to \sK^M_{r, X}$ is the inverse of the isomorphism
$\eqref{lem-vanishing-hdi-proj-space1}$. Let $i: X\inj P_X$ be a section of $\pi$ and consider
the pushforward $i_*: i_*\sK^M_{r,X}[-N]\to \sK^M_{r+N,P_X}$.
The composition
\[\sK^M_{r,X}\xr{R^N\pi_*(i_*)} R^N\pi_*\sK^M_{r,P_X}\xr{\pi_*} \sK^M_{r,X}\]
is the identity. Hence it suffices
to show  that  $R^N\pi_*(i_*)$ is the equal to
\eqref{lem-vanishing-hdi-proj-space1}. Further it suffices to check this in the generic point $\eta\in X$.
The statement now follows directly from the explicit description of the isomorphism 
$K^M_{r}(k(\eta))\cong H^N_\eta(\sK^M_{r+N,P_X})$ given in \eqref{cor:MilnorK-coh-in-c1}.
\end{proof}

\begin{thm}\label{prop:proj-pf-relK}
Let $D$ be an effective Cartier divisor on $X$ and assume that $D_{\rm red}$ is a simple normal crossing
divisor. For a scheme $Y$ set $P_Y:=\P^N_Y$. Denote by $\pi: P_X\to X$ the projection. Then 
for $r\ge 0$ we have on $X_\Nis$
\eq{prop:proj-pf-relK1}{R^i\pi_*\sK^M_{r, P_X|P_D,\Nis }=0,\quad \text{for all }  i>N,}
and for $i\in [0,N]$ the cup product with $c_1(\sO(1))^i\in R^i\pi_*\sK^M_{i,P_X,\Nis}$
induces an isomorphism
\eq{prop:proj-pf-relK2}{-\cup c_1(\sO(1))^i: \sK^M_{r-i,X|D,\Nis}\xr{\simeq} 
           R^i\pi_*\sK^M_{r, P_X|P_D,\Nis}.}
If $D_{\rm red}$ is smooth the same is true on $X_\Zar$ with $\sK^M_{r,X|D,\Nis}$ replaced by $\sK^M_{r, X|D}$.
\end{thm}
\begin{proof}
This follows immediately by induction on the dimension of $X$, Proposition \ref{prop:gr0-map},
Proposition \ref{prop:higher-gr-map-char0}, Theorem \ref{thm:higher-gr-map-poschar}
and Lemma \ref{lem-vanishing-hdi-proj-space}.
\end{proof}

\begin{defn}\label{defn:proj-pf-relK}
In the situation of Theorem \ref{prop:proj-pf-relK} above we define the pushforward
\[\pi_*: R\pi_*\sK^M_{r+N, P_X|P_D,\Nis}[N]\to\sK^M_{r,X|D,\Nis}\]
to be the composition 
\[R\pi_*\sK^M_{r+N, P_X|P_D,\Nis}[N]\xr{\text{can.} \,\eqref{prop:proj-pf-relK1}} 
  R^N\pi_*\sK^M_{r+N, P_X|P_D,\Nis}\xr{\simeq \,\eqref{prop:proj-pf-relK2}} \sK^M_{r, X|D,\Nis}.\]
Notice that by Lemma \ref{lem:pf-MK} this definition of the pushforward is compatible (in the obvious sense)
with the pushforward $\pi_*: R\pi_*\sK^M_{r+N, P_X}[N]\to\sK^M_{r,X}$ from \ref{pf}.
\end{defn}

\section{Cycle Map to cohomology of relative Milnor {$K$}-sheaves}
Let $k$ be a field and $X$ an equidimensional scheme of finite type over $k$. 
\subsection{The classical cycle map}\label{reviewcyclemap}
Everything in this subsection is well known to the experts. We give the proofs for lack of references.
\begin{no}{}\label{CycleMap-MilnorK}
Recall the notations from section \ref{ccmod}. 
In particular for $n\ge 1$ we have 
$\square^n\subset (\P^1)^n\supset (\P^1\setminus\{\infty\})^n=\Spec k[y_1,\ldots, y_n]$. 
By convention $\square^0=\Spec k$. Denote by $\pi_n : X\times \square^n  \to X$ the projection.
Recall that for $r\ge 0$, $n\in [0,r]$ and $Z\subset X\times \square^n$ an integral closed subscheme of
codimension $r$, the dimension formula (see e.g. \cite[Prop (5.6.5)]{EGAIV2}) yields
\[\codim(\ol{\pi_n(Z)}, X)\ge r-n,\]
where $\ol{\pi_n(Z)}$ denotes the closure of $\pi_n(Z)$ in $X$, and
equality holds if and only if $Z$ is generically finite over $\ol{\pi_n(Z)}$.
We can therefore define the a group homomorphism
\[\ul{\varphi}^{r,n}_X: \ul{z}^r(X, n)\to \bigoplus_{x\in X^{(r-n)}} K^M_{n}(k(x))\]
by
\begin{align}
\ul{\varphi}^{r,n}_X(Z) & =
\begin{cases} 
(-1)^{rn}\cdot \Nm_{k(z)/k(\pi_n(z))}\{y_n(z),\ldots, y_1(z)\}, & \text{if } k(z)/k(\pi_n(z)) \text{ is finite},\\
  0, &\text{else}
   \end{cases}\notag\\ 
 &\in K^M_n(k(\pi_n(z))), \notag
\end{align}
where $Z\subset X\times\square^n$ is an integral closed subscheme of codimension $r$ 
which meets all the faces properly 
and has generic point $z\in Z$,   $y_i(z)$ denotes the residue class  of $y_i\in\sO_{X\times\square^n, z}$ and 
$\Nm_{z/\pi_n(z)}: K^M_n(k(z))\to K^M_n(k(\pi_n(z)))$ denotes the norm map on Milnor $K$-theory.
(By convention it equals multiplication with the degree $[k(z):k(\pi_n(z))]$ if $n=0$.)
Clearly $\ul{\varphi}^{r,n}_X$ sends degenerate cycles to 0 and hence it induces a map
\[\varphi^{r,n}_X: z^r(X, n)\to \bigoplus_{x\in X^{(r-n)}} K^M_n(k(x)).\]
For $n\not\in[0,r]$ we define $\varphi^{r,n}_X$ to be the zero map.
\end{no}

\begin{lem}\label{lem:phi-complex-map}
For $r\ge 0$ the collection of maps $(\varphi^{r, 2r-i}_X)_{i\in\Z}$ induces a morphism of complexes
\[\varphi_X^r: z^r(X, 2r-\bullet)\to C^\bullet_{r,X}(X)[-r],\]
where $C^\bullet_{r,X}$ is the Gersten complex, see \ref{MilnorK}.
(It is defined for general $X$, see e.g. \cite[5.]{Rost}, but if $X$ is not smooth it does not need to be a resolution.)
Furthermore this map is compatible with restrictions to open subsets of $X$ in the obvious sense.
\end{lem}
\begin{proof}
The second assertion is clear. For the first assertion
we have to show that for $n\in [1,r+1]$, $Z\subset X\times\square^n$ an integral closed subscheme
of codimension $r$ with generic point $z\in Z$ intersecting all the faces properly  
and $x\in \ol{\{\pi_n(z)\}}\cap X^{(r-n+1)}$ we have the following equality in $K^M_{n-1}(k(x))$
\eq{lem:phi-complex-map1}{(-1)^r\partial^M_x(\varphi^{r,n}_{X,\pi_n(z)}(Z))= 
        \varphi_{X,x}^{r,n-1}(\partial^{\rm cyc}(Z)), }
where we denote by $\varphi^{r,n}_{X,x}$ the composition of $\varphi^{r,n}_X$ with the projection to the 
$x$-summand and $\partial^M_x: K^M_n(k(\pi_n(z)))\to K^M_{n-1}(k(x))$ and 
$\partial^{\rm cyc}: z^r(X,n)\to z^r(X,n-1)$ denote the boundary maps in $C^\bullet_{r,X}$ and $z^r(X,2r-\bullet)$,
respectively. Notice that the factor $(-1)^r$ appears on the left hand side in the equation 
\eqref{lem:phi-complex-map1} since by convention the shifting operation $[-r]$ on complexes multiplies
the boundary maps by this factor. We consider two cases.

{\em 1st case: $k(z)/k(\pi_n(z))$ is finite.} Set $Z_0=\ol{\pi_n(Z)}$. 
We have $x\in Z_0^{(1)}$.  Denote by 
$\ol{Z}\subset X\times (\P^1)^n$  the closures of $Z$. We have a commutative diagram
\[\xymatrix{ \tilde{Z}\ar[d]^{\tilde{j}}\ar[r]^\nu\ar@/_1.5pc/[dd]_{\tilde{\pi}} & Z\ar[d]_j\ar@/^1.5pc/[dd]^{\pi_n}\\
                \tilde{\ol{Z}}\ar[r]^{\ol{\nu}}\ar[d]^{\tilde{\ol{\pi}}} & \ol{Z}\ar[d]_{\ol{\pi}_n}\\
                  \tilde{Z}_0\ar[r]_-{\nu_0} & Z_0,
}\]
in which the horizontal maps are the normalizations, $j$ and $\tilde{j}$ are open immersions and 
the other vertical maps are induced by the projection $X\times (\P^1)^n\to X$. 
 Notice that $\ol{Z}$, $\tilde{\ol{Z}}$ and $\tilde{Z}_0$ are finite over a neighborhood of any point of $Z_0^{(1)}$.
We compute:
\begin{align}
\partial_x^M(\varphi^{r,n}_{X, \pi_n(z)}(Z)) & =(-1)^{nr} \sum_{\tilde{x}_0\in \nu_0^{-1}(x)} 
                        \Nm_{\tilde{x}_0/x}(\partial_{\tilde{x}_0} \Nm_{z/\pi_n(z)}\{y_n(z),\ldots, y_1(z)\})\notag\\
                  &=(-1)^{nr} \sum_{\tilde{x}_0\in \nu_0^{-1}(x)} \,
                            \sum_{\tilde{x}\in\tilde{\ol{\pi}}^{-1}(\tilde{x}_0)} 
                   \Nm_{\tilde{x}/x}(\partial_{\tilde{x}}\{y_n(z),\ldots, y_1(z)\})\notag\\
           &  = (-1)^{nr} \sum_{\tilde{x}_0\in \nu_0^{-1}(x)} \,
                            \sum_{\tilde{x}\in\tilde{\pi}^{-1}(\tilde{x}_0)} 
                   \Nm_{\tilde{x}/x}(\partial_{\tilde{x}}\{y_n(z),\ldots, y_1(z)\}).\notag
\end{align}
Here the first equality holds by definition of $\partial^M_x$,
for the second see e.g. \cite[(1.1), R3b and Thm (1.4)]{Rost},
the third equality holds since a point $\tilde{x}\in \tilde{\ol{Z}}\setminus \tilde{Z}$ has one of the
$y_i$ coordinates equal to 1 and therefore $\partial_{\tilde{x}}\{y_n(z),\ldots, y_1(z)\}=0$ in this case.
In particular we can assume $x\in Z_0^{(1)}\cap \pi_n(Z)$. 
Since $Z$ intersects all faces properly only the two following cases can occur:
\begin{enumerate}
\item\label{lem:phi-complex-map2} 
              $x$ is not contained in any of the subsets $\pi_n(\partial^\epsilon_i(Z))$,
              $i=1,\ldots, n$, $\epsilon= 0,\infty$.
\item\label{lem:phi-complex-map3} There exists exactly one $i_0\in \{1,\ldots, n\}$ and 
                one $\epsilon_0\in \{0,\infty\}$ such that $x\in \pi_n(\partial^{\epsilon_0}_{i_0}(Z))$. 
\end{enumerate}
In case \eqref{lem:phi-complex-map2} we get 
   \[\partial_x^M(\varphi^{r,n}_{X, \pi_n(z)}(Z))=0=\varphi^{r,n-1}_{X, x}(\partial^{\rm cyc}(Z)).\]
In case \eqref{lem:phi-complex-map3} we set $\epsilon_0':=1$ if $\epsilon_0=0$ and $\epsilon_0':=-1$
if $\epsilon_0=\infty$ and get 
\[\partial^M_x(\varphi^{r,n}_{X,\pi_n(z)}(Z)) \]
\[ = (-1)^{nr+i_0-1}\sum_{x'\in\pi_n^{-1}(x)} \sum_{\tilde{x}\in \nu^{-1}(x')}
          v_{\tilde{x}}(y_{i_0}(z)^{\epsilon_0'})\cdot
       \Nm_{\tilde{x}/x}(\nu^*\{y_n(x'),\ldots,\widehat{y_{i_0}(x')},\ldots, y_1(x')\}) \]
\[= (-1)^{nr} \sum_{x'\in\pi_n^{-1}(x)} (-1)^{{i_0}-1}\cdot\epsilon_0'\cdot{\rm ord}_{x'}(y_{i_0}(z))\cdot
          \Nm_{x'/x}\{y_1(x'),\ldots,\widehat{y_{i_0}(x')},\ldots, y_n(x')\}\]
\[= (-1)^r\varphi^{r,n-1}_{X, x}(\partial^{\rm cyc}(Z)).\]
Here the first equality holds by definition of the tame symbol, the second by the projection formula for the norm map
and \cite[Ex 1.2.3]{F} and the third by the definition of the maps involved.
This proves \eqref{lem:phi-complex-map1} in this case.

{\em 2nd case: $k(z)/k(\pi_n(z))$ has positive transcendence degree.}
In this case we have to show
\eq{lem:phi-complex-map4}{\varphi^{r,n-1}_{X,x}(\partial^{\rm cyc}(Z))=0.}
This is clearly the case if there is no point in $Z^{(1)}$ which is finite over $x$.
Otherwise if such a point exists and we denote by $W\subset X\times\square^n$ its closure, 
then the dimension formula yields
\[r+1-n=\codim(\ol{\pi_n(W)}, X)=\codim(\ol{\pi_n(W)},\ol{\pi_n(Z)})+\codim(\ol{\pi_n(Z)}, X).\]
By assumption $\codim(\ol{\pi_n(Z)}, X)>r-n$. Hence $\ol{\pi_n(W)}=\ol{\pi_n(Z)}$ and since $x$ is the 
generic point of $\ol{\pi_n(W)}$ we obtain:
\begin{enumerate}
\item The base change $Z_x=Z\times_{X\times \square^n} (x\times \square^n)$ is an affine integral 1-dimensional 
          scheme of finite type over $x$.
\item The natural map $z\to Z$ factors uniquely through the projection $Z_x\to Z$.
 \end{enumerate}
Thus $Z_x\subset x\times\square^n$ is an integral closed subscheme of dimension 1
which intersects all the faces properly and we have
\[\varphi^{r,n-1}_{X,x}(\partial^{\rm cyc}(Z))=\varphi^{n-1,n-1}_{x,x}(\partial^{\rm cyc}(Z_x)),\]
where the maps on the right are
$\partial^{\rm cyc}: z^{n-1}(x,n)\to z^{n-1}(x,n-1)$ and 
$\varphi^{n-1,n-1}_{x,x}: z^{n-1}(x, n-1)\to K^M_{n-1}(k(x))$.
Denote by $\ol{Z}_x$ the closure of $Z_x$ in $(\P^1_x)^n$ and by $\nu: C\to \ol{Z}_x$ the normalization.
Then by definition
\mlnl{\partial^{\rm cyc}(Z_{x})= 
 \sum_{i=1}^n (-1)^i \bigg(
\Big(\sum_{x'\in Z_x\cap (y_i=\infty)}
\big(\sum_{\tilde{x}\in \nu^{-1}(x')}v_{\tilde{x}}(y_i^{-1})[\tilde{x}:x']\big)\cdot x'\Big)
\\ -\Big(\sum_{x'\in Z_x\cap (y_i=0)}
\big(\sum_{\tilde{x}\in \nu^{-1}(x')}v_{\tilde{x}}(y_i)[\tilde{x}:x']\big)\cdot x'\Big)\bigg). }
Applying $\varphi^{n-1,n-1}_{x,x}$  and using that $Z_x$ intersects all faces properly we obtain by a similar
calculation as in the {\em 1st case}
\[\varphi^{n-1,n-1}_{x,x}(\partial^{\rm cyc}(Z_x)) =(-1)^{n-1}
\sum_{\tilde{x}\in C} \Nm_{\tilde{x}/x}(\partial_{\tilde{x}}(\{y_n(z),\ldots, y_1(z)\}).
\]
This is zero by the reciprocity law for the tame symbol (see e.g. \cite[(2.4)]{Rost}).
Hence the vanishing \eqref{lem:phi-complex-map4}.
\end{proof}

\begin{cor}\label{cor:MilnorChow}
Let $X$ be a smooth equidimensional $k$-scheme and $r\ge 0$.  
Then the maps $\{\varphi^r_{U}\}_{U\subset X}$, where $U$ ranges over all open subsets of $X$,
induces a quasi-isomorphism of complexes of Zariski sheaves on $X_\Zar$
\[\phi^r_X: \tau_{\ge r}\Z(r)_X\xr{\rm qis} C^{\bullet}_{r,X}[-r].\]
Here $\Z(r)_X$ is the complex of Zariski sheaves $U\mapsto z^r(U,2r-\bullet)$.
In particular we have an isomorphism in $\sD^b(X_\Zar)$ (also denoted by $\phi^r_X$)
\[\phi^r_X:\tau_{\ge r}\Z(r)_X\xr{\simeq} \sK^M_X[-r].\]
\end{cor}
\begin{proof}
By the Gersten resolution for higher Chow groups (see \cite[Thm (10.1)]{Bloch-HigherChow})
we have $\sH^i(\Z(r)_X)=0$ for all $i> r$.
Thus it suffices to show that $\phi^r_X$ induces an isomorphism $\sH^r(\Z(r)_X)\cong \sH^0(C^\bullet_{r,X})$.
This follows directly from the definition of $\varphi^{r,r}_X$, \cite[Thm (10.1)]{Bloch-HigherChow} and
the construction of the isomorphism $CH^r(k(X), r)\xr{\simeq}K^M_r(k(X))$ in \cite[3.]{Totaro}.
\end{proof}

\subsection{The relative cycle map}

\begin{no}{}\label{CycleMap-relMilnorK}
Let $D$ be an effective Cartier divisor on $X$ and denote by $j: U:= X\setminus D\inj X$ the inclusion of the 
complement. For $r\ge 0$ let  $C^\bullet_{r,X|D}$ be the Cousin complex of $\sK^M_{r,X|D}$ and 
$C^{h,\bullet}_{r,X|D}$ the Cousin complex of $R\e_*\sK^M_{r, X|D,\Nis}$, see \ref{CousinComplex-relK}.
For $n\in [0,r]$ we define a morphism
\[\varphi^{r,n}_{X|D}: z^r(X|D,n)\to C^{h, r-n}_{r, X|D}(X)\]
as the precomposition of the natural map 
$C^{r-n}_{r,X|D}(X)    \xr{\eqref{CousinComplex-relK1}} C^{h,r-n}_{r, X|D}(X)$ with
\eq{CycleMap-relMilnorK1}{z^r(X|D,n)\inj z^r(X,n)_U\xr{\varphi^{r,n}_X} \bigoplus_{x\in U^{(r-n)}} K^M_n(k(x))
                     \xr{\eqref{CousinComplex-relK0}} C^{r-n}_{r,X|D}(X) ,}
where $z^r(X,n)_U\subset z^r(X,n)$ is the subgroup of cycles on $X\times \square^n$ supported in
$U\times \square^n$ (i.e. cycles on $X\times \square^n$ whose support is contained in $U\times\square^n$, 
cf. \ref{Chow-mod-properties}, \eqref{Chow-mod-properties1}) and the first map is the natural inclusion from
\ref{Chow-mod-properties}, \eqref{Chow-mod-properties1}. 
For $n\not\in [0,r]$ we define $\varphi^{r,n}_{X|D}$ to be the zero map.
\end{no}

\begin{prop}\label{prop:rel-phi-complex-map}
Let $X$ be a smooth equidimensional scheme and $D$ an effective
divisor such that $D_{\rm red}$ is a simple normal crossing divisor.
For $r\ge 0$ the collection of maps $(\varphi^{r, 2r-i}_{X|D})_{i\in \Z}$ induces a morphism of complexes
\[\varphi^r_{X|D}: z^r(X|D, 2r-\bullet)\to C^{h,\bullet}_{r,X|D}(X)[-r].\]
Furthermore, this map is compatible with restriction to open subsets of $X$ in the obvious sense and hence induces
a morphism between complexes of sheaves on $X_\Zar$
\[\phi_{X|D}^r: \tau_{\ge r}\Z(r)_{X|D}\to C^{h,\bullet}_{r,X|D}[-r].\]
If $D_{\rm red}$ is a smooth divisor  $\phi_{X|D}^r$ factors as a morphism of complexes
\[\tau_{\ge r}\Z(r)_{X|D}\to C^{\bullet}_{r,X|D}[-r]\xr{\eqref{CousinComplex-relK1}} C^{h,\bullet}_{r,X|D}[-r],\]
where the first map is induced by \eqref{CycleMap-relMilnorK1}.
\end{prop}
\begin{proof}
Once we know that  $\varphi^r_{X|D}$ is a map of complexes it is clear that
it induces a map between complexes of sheaves $\phi^r_{X|D}$.
For the first statement 
we have to show the following:
For $n\in [1,r+1]$,   $Z\in C^r(X|D,n)$ (see Definition \ref{defn:modulus-cycle}) with generic point $z\in Z$ 
and for all points $x\in \ol{\{\pi_n(z)\}}\cap X^{(r-n+1)}$ 
the following equality holds in $H^{r-n+1}_x(R\e_*\sK^M_{r,X|D,\Nis})$
\eq{prop:rel-phi-complex-map1}{(-1)^r\partial^C_x(\varphi^{r,n}_{X|D,\pi_n(z)}(Z))= 
        \varphi_{X|D,x}^{r,n-1}(\partial^{\rm cyc}(Z)), }
where we denote by $\varphi^{r,n}_{X|D,x}$ the composition of $\varphi^{r,n}_{X|D}$ with the projection to the 
$x$-summand and by 
$\partial^C_x: H^{r-n}_{\pi_n(z)}(R\e_*\sK^M_{r,X|D})\to H^{r-n+1}_x(R\e_*\sK^M_{r,X|D})$ and
$\partial^{\rm cyc}: z^r(X|D,n)\to z^r(X|D,n-1)$ the boundary maps in 
$C^{h,\bullet}_{r,X|D}$ and $z^r(X|D,2r-\bullet)$, respectively.  

Notice that the restriction of $\varphi^{r,n}_{X|D}$ to $U$ equals the map $\varphi^{r,n}_U$ from 
\ref{CycleMap-MilnorK}.
In particular, for $x\in U$ the equality \eqref{prop:rel-phi-complex-map1} follows from 
Lemma \ref{lem:phi-complex-map}. Thus we can assume $x\in D$.
Therefore we have to show the vanishing of the left hand side in \eqref{prop:rel-phi-complex-map1}.
By definition of $\varphi^{r,n}_{X|D}$ we can further assume that $k(z)/k(\pi_n(z))$ is finite.
Taking the definition of the boundary maps in the Cousin complex $C^{h,\bullet}_{r,X|D}$ into account 
we see that it remains to show the following:

Denote by $\ol{Z}_0\subset X$ the closure of $\pi_n(Z)$ and by $z_0=\pi_n(z)\in \ol{Z}_0$ its generic point.
Assume $k(z)/k(z_0)$  is finite and $x\in D\cap \ol{Z}_0\cap X^{(r-n+1)}$. Then we have to show 
\eq{prop:rel-phi-complex-map2}{  \varphi^{r,n}_{X|D,z_0}(Z)\in 
      \text{Im}(\sH^{r-n}_{\ol{Z}_0}(R\e_*\sK^M_{r,X|D})_x \to H^{r-n}_{z_0}(R\e_*\sK^M_{r, X|D}) ).}

Observe that under the above assumptions we have $x\in\ol{Z}_0^{(1)}$.
Denote the composition of the map \eqref{CycleMap-relMilnorK1} with the projection to the $z_0$-summand
by
\[\psi_{z_0}: z^r(X|D,n)\to H^{r-n}_{z_0}(\sK^M_{r,X|D}).\]
Notice that \eqref{prop:rel-phi-complex-map2} holds if
\eq{prop:rel-phi-complex-map2.5}{\psi_{z_0}(Z)\in 
      \text{Im}(\sH^{r-n}_{\ol{Z}_0}(\sK^M_{r,X|D})_x \to H^{r-n}_{z_0}(\sK^M_{r, X|D}) ).}
Also, in case \eqref{prop:rel-phi-complex-map2.5} holds for all $Z$, we actually get that \eqref{CycleMap-relMilnorK1}
induces a morphism of complexes.

In the following we will show that \eqref{prop:rel-phi-complex-map2.5} is satisfied if $\ol{Z}_0$ is normal
or if $D_{\rm red}$ is smooth, and that \eqref{prop:rel-phi-complex-map2} holds in general. This will
prove the proposition.

{\em 1st case: $\bar{Z}_0$ is normal.} In this case $\bar{Z}_0$ is regular at $x$. Hence we find
a regular sequence $t_1,\ldots, t_{r-n}\in \sO_{X, x}$ 
with $\sO_{X,x}/(t_1,\ldots, t_{r-n})\cong \sO_{\ol{Z}_0,x}$. 
Let $f\in \sO_{X,x}$ be a local equation for $D$ and 
denote by $D_0= D_{|\ol{Z}_0}$ the pullback of $D$ to $\ol{Z}_0$. 
The image of $f$ in $\sO_{\ol{Z}_0,x}$ is still denoted by $f$.
We claim that in order to prove \eqref{prop:rel-phi-complex-map2.5} it suffices to show 
\eq{prop:rel-phi-complex-map3}{\Nm_{k(z)/k(z_0)}\{y_n(z),\ldots, y_1(z)\}\in \sK^M_{n, \ol{Z}_0|D_0,x}.}
Indeed set $\nu:=\Nm_{k(z)/k(z_0)}\{y_n(z),\ldots, y_1(z)\}$. If the claim
\eqref{prop:rel-phi-complex-map3} holds we can lift $\nu$ to an element 
$\tilde{\nu}\in \sK^M_{n, X|D,x}$ (using the explicit description from Remark \ref{rmk:MilnorModulus}).
We obtain an element (see \eqref{cech-symbol2})
\[\genfrac{[}{]}{0pt}{}{\tilde{\nu}\cdot\{t_1,\ldots, t_{r-n}\}}{t_1,\ldots,t_{r-n}}\in 
        (\sH_{\ol{Z}_0}^{r-n}(\sK^M_{r, X|D}))_x,\]
which by Corollary \ref{cor:MilnorK-coh-in-c} 
maps under restriction to the generic point of $\ol{Z}_0$ to the element
$\psi_{z_0}(Z)\in H^{r-n}_{z_0}(\sK^M_{r,X|D})\cong K^M_{n}(k(z_0))$.

We have $\sO_{\ol{Z}_0, x}[\frac{1}{f}]=k(z_0)$. Therefore $(\sK^M_{n, \ol{Z}_0|D_0})_x$
is generated by symbols of the form $\{1+fa, \kappa_1,\ldots, \kappa_n\}$, where $a\in \sO_{\ol{Z}_0,x}$
and $\kappa_i\in k(z_0)^\times$ (see Remark \ref{rmk:MilnorModulus}).
Denote by $A$ the completion of $\sO_{\ol{Z}_0,x}$ along the maximal ideal and by $K$ its fraction field;
it is a complete discrete valuation field with $A$ as its ring of integers.
Let $m$ be the valuation of $f\in A$.
Then by Lemma \ref{lem:DVR-cDVR}
the natural map $K^M_n(k(z_0))\to K^M_n(K)$ induces an isomorphism 
\[K^M_n(k(z_0))/(\sK^M_{n, \ol{Z}_0|D})_x\to  K^M_n(K)/U^m K^M_n(K).\]
Therefore it suffices to show that the pullback of $\Nm_{k(z)/k(z_0)}\{y_n(z),\ldots, y_1(z)\}$
to $K^M_n(K)$ lies in $U^m K^M_n(K)$. We have  
$k(z)\otimes_{k(z_0)} K=\prod_i L_i$, where each $L_i$ equals the completion of $k(z)$
along a point in the normalization of the closure of $Z$ in $X\times (\P^1)^n$, which lies above $x$.
Now we fix $i$ and set $L:=L_i$. 
Denote by $B$ the normalization of $A$ in $L$, by $\fm$ its maximal ideal and by $\iota:k(z)\inj L$ the natural inclusion.
We set
\[a_j:=\begin{cases} \iota(y_j(z))-1, & \text{if }\iota(y_j(z))\in B\\
                               \iota(y_j(z)^{-1})-1,  & 
                                                       \text{if }\iota(y_j(z)^{-1})\in B. \end{cases}\]
By the compatibility of the 
norm map with pullback we are reduced to show
\eq{prop:rel-phi-complex-map4}{\Nm_{L/K}\{1+a_1,\ldots, 1+a_n\}\in U^m K^M_n(K).}
The modulus condition \eqref{defn:modulus-cycle1} which the integral cycle $Z$ satisfies translates into
\[a_1\cdots a_n/f \in B.\]
Up to permuting the factors $a_j$ (and therefore changing the element in \eqref{prop:rel-phi-complex-map4}
by a sign) we can assume that there is an integer $\mu\in [1,n]$ such that $a_1,\ldots, a_\mu \in \fm$
and $a_{\mu+1},\ldots, a_{n}\in B^\times$. A fortiori we have 
\[a_1\cdots a_\mu/f\in B.\]
Then we can apply Lemma \ref{lem:symbol-formula}, \eqref{lem:symbol-formula1}, repeatedly
(starting from $s=a_1$, $t=a_2$, $a=b=1$) to obtain
\[\{1+a_1,\ldots, 1+a_n\}=\{1+u a_1\cdots a_\mu, \lambda_2,\ldots, \lambda_n\},\quad 
               u\in B^\times, \lambda_j\in L^\times.\]
In particular, $\{1+a_1,\ldots, 1+a_n\}\in U^{m\cdot e}K^M_n(L)$, where $e$ denotes the ramification index
of $L/K$. Therefore \eqref{prop:rel-phi-complex-map4} follows from  
$\Nm_{L/K}(U^{m\cdot e}K^M_n(L))\subset U^mK^M_n(K)$, see \cite[Prop. 2]{Kato1983}
(also \cite[Thm 1.1]{Morrow12}).

{\em 2nd case: $\ol{Z}_0$ is arbitrary.}
      We denote by $\nu_0:\tilde{Z}_0\to \ol{Z}_0$ the normalization. It is a finite map and hence factors
as  a closed immersion $\tilde{Z}_0\inj \P^N_X:=P_X$ followed by the projection $P_X\to X$.
There is a generic point of $Z\times_{\ol{Z}_0}\tilde{Z}_0$ which maps to the generic point of $Z$
and we denote by $Z'\subset Z\times_{\ol{Z}_0} \tilde{Z}_0$ its closure. 
We can view $Z'$ as a closed subscheme of $P_X\times\square^n$.
By construction  the projection $P_X\times\square^n\to X\times\square^n$ induces a finite
and surjective morphism $Z'\to Z$. It follows that $Z'$ has codimension $N+r$ in $P_X\times \square^n$, 
intersects all faces properly and  satisfies the modulus condition \eqref{defn:modulus-cycle1} 
with respect to the effective divisor $P_D\subset P_X$. Furthermore
the closure of the image of $Z'$ in $P_X$ equals $\tilde{Z}_0$, which has generic point $z_0$.
Thus we can apply the 1st case to obtain
\eq{prop:rel-phi-complex-map4.5}{ \psi_{z_0}(Z')\in \\
      \text{Im}(\sH^{r+N-n}_{\tilde{Z}_0}(\sK^M_{r+N,P_X|P_D})_{x'} 
           \to H^{r+N-n}_{z_0}(\sK^M_{r+N, P_X|P_D}) ),}
where $x'$ is any point in $\tilde{Z}_0^{(1)}\cap P_D$.
A fortiori
\ml{prop:rel-phi-complex-map5}{ \varphi^{r+N,n}_{P_X|P_D,z_0}(Z')\in \\
      \text{Im}(\sH^{r+N-n}_{\tilde{Z}_0}(R\e_*\sK^M_{r+N,P_X|P_D})_{x'} 
           \to H^{r+N-n}_{z_0}(R\e_*\sK^M_{r+N, P_X|P_D}) ).}
Let $x\in D\cap \ol{Z}_0^{(1)}$ be as in \eqref{prop:rel-phi-complex-map2}.
Then there exists an open neighborhood $\tilde{V}$ of $x\times P$ in $P_X$
such that $\tilde{V}\cap \tilde{Z}_0$ contains all 1-codimensional  points of $\tilde{Z}_0$
lying over $x$ and such that $\varphi^{r+N,n}_{P_X|P_D,z_0}(Z')$
comes from an element in
\[H^{r+N-n}_{\tilde{Z}_0\cap \tilde{V}}(\tilde{V}, R\e_*\sK^M_{r+N, P_X|P_D})=
 H^{r+N-n}_{\tilde{Z}_0\cap \tilde{V}}(\tilde{V}_\Nis, \sK^M_{r+N, P_X|P_D,\Nis}).\]
(This follows from \eqref{prop:rel-phi-complex-map5} and the fact that the Cousin complex is a resolution, 
see Corollary \ref{cor:CousinComplex-relK}.)
After possibly shrinking $\tilde{V}$ we find an open neighborhood $V\subset X$ of $x$ such that
$\tilde{V}\subset P_V$ and the complement of 
$\tilde{Z}_0\cap \tilde{V}\subset \tilde{Z}_0\cap P_V$ has codimension 2 in $\tilde{Z}_0$.
It follows from Theorem \ref{thm:relK-is-CM} that $\varphi^{r+N,n}_{P_X|P_D,z_0}(Z')$
spreads out to an element of
\[H^{r+N-n}_{\tilde{Z}_0\cap P_V}(P_{V,\Nis}, \sK^M_{r+N, P_X|P_D,\Nis}).\]
Now the pushforward from Definition \ref{defn:proj-pf-relK} induces a commutative diagram
\[\xymatrix{H^{r+N-n}_{\tilde{Z}_0\cap P_V}(P_{V,\Nis}, \sK^M_{r+N, P_X|P_D,\Nis})\ar[d] &\\
     H^{r+N-n}_{(\ol{Z}_0\cap V)\times_V P_V}(P_{V,\Nis}, \sK^M_{r+N, P_X|P_D,\Nis}) \ar[d]_{\pi_*}\ar[r] &
                H^{r+N-n}_{z_0}(\sK^M_{r+N, P_X|P_D,\Nis})\ar@{=}[d] \\
     H^{r-n}_{\ol{Z}_0\cap V}(V_{\Nis}, \sK^M_{r, X|D,\Nis})\ar[r] &     H^{r-n}_{z_0}(\sK^M_{r, X|D,\Nis}).
                     }\]
For the equality on the right notice that both groups are equal to $K^M_n(k(z_0))$.
By definition of $\varphi^{*,*}_{*|*}$ it is also immediate that
$\varphi^{r+N,n}_{P_X|P_D,z_0}(Z')=\varphi^{r,n}_{X|D,z_0}(Z)$.
Hence \eqref{prop:rel-phi-complex-map2} also holds in the 2nd case. 
If $D_{\rm red}$ is smooth the above proof goes through if we drop the `$\Nis$', hence 
\eqref{prop:rel-phi-complex-map2.5} also holds in the second case.
This finishes the proof of the proposition.
\end{proof}

\begin{cor}\label{cor:cycle-maps}
Let $X$ and $D$ be as in Proposition \ref{prop:rel-phi-complex-map} and 
denote by $j:U:=X\setminus D\inj X$ the inclusion of the complement of $D$. Then for all $r\ge 1$
we have the following commutative diagram in $\sD^b(X_{\rm Zar})$
\[\xymatrix{ \tau_{\ge r}\Z(r)_{X|D}\ar[r]^-{\phi^r_{X|D}}\ar[d] & R\e_*\sK^M_{r, X|D,\Nis}[-r]\ar[d]\\
                   \tau_{\ge r}\Z(r)_X\ar[r]^{\simeq}_{\phi^r_X} & \sK^M_{r,X}[-r],   }\]
in which the lower horizontal map is an isomorphism.
Here the horizontal maps are induced by the maps $\phi^r_X$ and $\phi^r_{X|D}$ from 
Corollary \ref{cor:MilnorChow} and Proposition \ref{prop:rel-phi-complex-map}, respectively,
and the left vertical map is induced by the natural inclusion and the 
right vertical map is induced  by the natural inclusion $\sK^M_{r,X|D,\Nis}\inj \sK^M_{r,X,\Nis}$
and the isomorphism $R\e_*\sK^M_{r, X,\Nis}\cong \sK^M_{r,X}$ from Corollary \ref{cor:relK-smDiv}.
Furthermore if $D_{\rm red}$ is smooth we can replace $R\e_*\sK^M_{r, X|D,\Nis}[-r]$
by $\sK^M_{r, X|D}[-r]$.
\end{cor}
\begin{proof}
This follows directly from
Corollary \ref{cor:MilnorChow}, \ref{MilnorK}, Proposition \ref{prop:rel-phi-complex-map} 
and Corollary \ref{cor:CousinComplex-relK}.
\end{proof}

\begin{rmk}
\begin{enumerate}
\item
The $\dlog$ map induces a natural map 
\[\dlog:\sK^M_{r,X|D}\to \Omega^r_{X}(\log D_{\rm red})(-D),\]
see the proof of Proposition \ref{prop:higher-gr-map-char0}. Clearly it also
induces a map of complexes $\sK^M_{r,X|D}\to \Omega^{\ge r}_{X}(\log D_{\rm red})(-D)$.
The composition in $\sD^b(X_\Zar)$
\mlnl{\Z(r)_{X|D}\to \tau_{\ge r}\Z(r)_{X|D}\xr{\phi^r_{X|D}}R\e_*\sK^M_{r,X|D,\Nis}\\
\xr{\dlog} R\e_*(\Omega^{\ge r}_{X,\Nis}(\log D_{\rm red})(-D))= 
      \Omega^{\ge r}_{X}(\log D_{\rm red})(-D)}
is the regulator map defined in \cite[(7.10)]{BS14}, at least up to sign.
\item Assume $k$ is a perfect field of positive characteristic. 
Denote by $W_n\Omega^\bullet_X(\log D)$ the logarithmic
de Rham-Witt complex for the log scheme $(X, j_*\sO_U^\times\cap \sO_X)$, see \cite[4.]{HK}.
It defines a differential graded algebra  and we denote by 
\[W_n\Omega^\bullet_X(\log D)(-D)\subset W_n\Omega^\bullet_X(\log D)\]
the differential graded ideal generated by $W_n(\sO_X(-D))=\Ker (W\sO_X\to W\sO_D)$.
Then it is not hard to see that there is a natural map
\[\dlog :\sK^M_{r,X|D}\to W_n\Omega^r_X(\log D)(-D), \quad 
     \{a_1,\ldots, a_r\}\mapsto \dlog[a_1]\cdots\dlog[a_r],\]
where $[-]: \sO_X\to W_n\sO_X$ denotes the Teichm\"uller lift.
Since the sheaf $W_n\Omega^r_X(\log D)(-D)$ can be viewed as a coherent sheaf on $W_n X=\Spec W\sO_X$, its
Zariski and its Nisnevich cohomology coincide and as in (1) we obtain a cycle map
\[\Z(r)_{X|D}\to W_n\Omega^{\ge r}_X(\log D)(-D).\]
\end{enumerate}
\end{rmk}

\begin{cor}\label{cor:cycle-maps-Nis}
Assume $D_{\rm red}$ has simple normal crossings. Then  the family 
$\{\varphi^r_{V|v^*D}\}_v$, where $v$ runs through all \'etale maps $v:V\to X$,
induces a morphism of complexes of Nisnevich sheaves
\[\phi^r_{X|D,\Nis}: \tau_{\ge r}\Z(r)_{X|D,\Nis}\to C^\bullet_{r, X|D,\Nis}[-r],\]
see \ref{motivic-coh-mod} and \ref{CousinComplex-relK-Nis} for the notations.
By Corollary \ref{cor:CousinComplex-relK-Nis} we get an induced map (still denoted by the same symbol)
$\phi^r_{X|D,\Nis}: \tau_{\ge r}\Z(r)_{X|D,\Nis}\to\sK^M_{r, X|D,\Nis}[-r]$ in $\sD^b(X_\Nis)$ fitting
in the following commutative diagram 
\[\xymatrix{ \tau_{\ge r}\Z(r)_{X|D,\Nis}\ar[r]^-{\phi^r_{X|D,\Nis}}\ar[d] & \sK^M_{r, X|D,\Nis}[-r]\ar[d]\\
                   \tau_{\ge r}\Z(r)_{X,\Nis}\ar[r]^{\simeq}_{\phi^r_{X,\Nis}} & \sK^M_{r,X,\Nis}[-r],   }\]
in which the lower horizontal map is an isomorphism.
\end{cor}
\begin{proof}
It suffices to prove the existence of $\phi^r_{X|D,\Nis}$. That is,
we have to see that for a map $V'\to V$ between \'etale $X$-schemes the following diagram commutes:
\[\xymatrix{ z^r(V'|D_{V'}, 2r-\bullet)\ar[r]^{\varphi^r_{V'|D_{V'}}} & C^{h,\bullet}_{r,V'|D_{V'}}(V')[-r]\\
              z^r(V|D_{V}, 2r-\bullet)\ar[r]^{\varphi^r_{V|D_{V}}}\ar[u] & C^{h,\bullet}_{r,V|D_{V}}(V)[-r],\ar[u]
}\]
where the vertical arrows are the restriction maps. By definition of $\varphi^r_{*|*}$ in \ref{CycleMap-relMilnorK}
it suffices to check this over $U=X\setminus D$. 
Hence we can assume $D=0$. In this case the statement follows
from the definition of $\varphi^{r,n}$ in \ref{CycleMap-MilnorK} and the compatibility of 
the norm on Milnor $K$-theory with pullback, see e.g. \cite[R1c and (1.4)Thm]{Rost}. 
\end{proof}

\begin{prop}\label{prop:rel-cycle-map-surj-Nis}
Let $X$ be a smooth equidimensional scheme and $D$ an effective
divisor such that $D_{\rm red}$ is a simple normal crossing divisor. 
Then the map on $X_\Nis$
\[\sH^r(\Z(r)_{X|D,\Nis})\surj \sK^M_{r,X|D,\Nis}\]
induced by the cycle map $\phi^r_{X|D,\Nis}$ is surjective for all $r\ge 1$.
Furthermore, if $D_{\rm red}$ is smooth, with $\Nis$ replaced by $\Zar$, the same statement holds.
\end{prop}
\begin{proof}
It suffices to show that for $V\to X$ \'etale and elements
$a\in \sK^M_{1,X|D}(V)$ and $b_i\in \sO_V(U_V)^\times$ (where $U_V:=V\times_X U$)
there exists a cycle $\alpha\in z^r(V|D_V, r)$ with $\partial(\alpha)=0$ in $z^r(V|D_V, r-1)$ that satisfies
\eq{prop:rel-cycle-map-surj-Nis1}{\varphi^{r,n}_{V|D_V}(\alpha)= \{a, b_1,\ldots b_{r-1}\}\in \sK^M_{r, X|D}(V)
      \subset K^M_{r}(k(V)).}
We can assume that none of the elements $a$, $b_i$ is equal to 1.
Denote by $\Gamma_{a,b_1\ldots, b_{r-1}}$
the graph of the map $U_V\to (\P^1)^r$ defined by $a, b_1,\ldots, b_{r-1}$ and set
\[Z:=\Gamma_{a,b_1\ldots, b_{r-1}}\cap (U_V\times\square^r).\]
Notice that $Z$ is  isomorphic to $U_V$, it has empty intersection with all faces and its closure  
$\bar{Z}\subset V\times (\P^1)^r$ is smooth and satisfies (with the notation from \ref{cubes})
\[(D\times (\P^1)^r)\cdot \bar{Z} \le  F^r_1\cdot \bar{Z};\]
thus in particular it satisfies the modulus condition \eqref{defn:modulus-cycle1}. 
It is immediate to check that $\alpha:=(-1)^{\frac{r(r+1)}{2}}\cdot [Z]$ satisfies 
\eqref{prop:rel-cycle-map-surj-Nis1}. This proves the proposition.
\end{proof}

\begin{thm}\label{thm:rel-cycle-iso}
Let $X$ be a smooth equidimensional scheme of dimension $d=\dim X$ and $D$ an effective
divisor such that $D_{\rm red}$ is a simple normal crossing divisor.  
Then:
\begin{enumerate}
\item $H^i_{\sM}(X|D,\Z(r))=0=H^i_{\sM,\Nis}(X|D,\Z(r))$ for $i>d+r$.
\item\label{thm:rel-cycle-iso2} The cycle map $\Z(r)_{X|D,\Nis}\to \tau_{\ge r}\Z(r)_{X|D,\Nis}
                           \xr{\phi^r_{X|D,\Nis}} \sK^M_{r,X|D,\Nis}[-r]$
        induces an isomorphism
        \[\phi^{d,r}_{X|D,\Nis}: H^{d+r}_{\sM,\Nis}(X|D,\Z(r))\xr{\simeq} H^d(X_{\Nis}, \sK^M_{r,X|D,\Nis}).\]
      If moreover $D_{\rm red}$ is smooth, then all maps in the following commutative diagram are isomorphisms
     \eq{thm:rel-cycle-iso0}{\xymatrix{ H^{d+r}_{\sM}(X|D,\Z(r))\ar[r]^{\simeq}
                                                                 \ar[d]_{\phi^{d,r}_{X|D}}^\simeq & 
                       H^{d+r}_{\sM,\Nis}(X|D,\Z(r))\ar[d]^{\phi^{d,r}_{X|D,\Nis}}_\simeq\\
                      H^d(X_\Zar, \sK^M_{r,X|D})\ar[r]^-{\simeq}_{\eqref{cor:relK-smDiv1}} & 
                                                    H^d(X_\Nis, \sK^M_{r,X|D,\Nis}).
                         }}
\end{enumerate}
\end{thm}
We need the following two lemmas in the proof of the theorem. In the following we will
freely use basic properties of local cohomology of Nisnevich sheaves, for details see \cite{Nis}.

\begin{lem}\label{lem:rel-cycle-local-support}
Let $k$, $(X,D)$ be as in Theorem \ref{thm:rel-cycle-iso} above.
Then on $X_\Zar$
\[\sH^n_D(\tau_{\ge r}\Z(r)_{X|D})=\begin{cases} 0, & \text{if }n<r, \\
                   \sH^0_D(\sH^n(\Z(r)_{X|D})), & \text{if } n\ge r \text{ and }n\neq r+1,\end{cases}\]
and for $n=r+1$ there is a natural exact sequence
\[0\to \sH^1_D(\sH^r(\Z(r)_{X|D}))\to \sH^{r+1}_D(\tau_{\ge r}\Z(r)_{X|D})\to 
                \sH^0_D(\sH^{r+1}(\Z(r)_{X|D}))\to 0.\]
Furthermore, the same statements hold when we replace $X_\Zar$ and $\Z(r)_{X|D}$ by 
$X_\Nis$ and $\Z(r)_{X|D,\Nis}$, respectively.
\end{lem}
\begin{proof}We do the proof for $\Z(r)_{X|D}$; it works the same way for $\Z(r)_{X|D,\Nis}$.
Considering the spectral sequence
\[E_2^{a,b}=\sH^a_D(\sH^b(\tau_{\ge r} \Z(r)_{X|D}))\Rightarrow \sH^*_D(\tau_{\ge r}\Z(r)_{X|D})\]
we see that it suffices to prove the following claim:
\eq{prop:rel-cycle-map-support1}{\sH^a_D(\sH^b(\tau_{\ge r} \Z(r)_{X|D}))=0,\quad
    \text{ for all }(a,b)\not\in \{(1,r)\}\cup(\{0\}\times [r, 2r]).}
Clearly we have the vanishing for all 
$(a,b)\in (\Z\times (-\infty, r-1])\cup (\Z\times [2r+1,\infty])\cup ((-\infty,-1]\times \Z)$. 
For $a\ge 1$ and $b\ge r$ we have surjections (which are isomorphisms for $a\ge 2$)
\[R^{a-1}j_*\sH^b(\Z(r)_U)\surj \sH^a_D(\sH^b(\Z(r)_{X|D})),\]
where $j: U=X\setminus D\inj X$ is the inclusion of the complement.
Hence the claim \eqref{prop:rel-cycle-map-support1} follows from
$\sH^b(\Z(r)_U)=0$ for $b>r$,  $\sH^r(\Z(r)_U)=\sK^M_{r,U}$ (see Corollary \ref{cor:MilnorChow})
and  $R^{a-1}j_*\sK^M_{r,U}=0$, for $a\ge 2$ (by Corollary \ref{cor:MilnorKcompSNCD}). 
(For last vanishing in the Nisnevich case use that $R^{a-1}j_*\sK^M_{r, U,\Nis}$ is the sheaf associated to
$V\mapsto H^{b-1}((V\times_X U)_\Nis,\sK^M_{r,U})= H^{b-1}((V\times_X U)_\Zar,\sK^M_{r,U})$.)
\end{proof}

\begin{lem}\label{lem:rel-cycle-global-support}
Let $k$, $(X,D)$ be as in Theorem \ref{thm:rel-cycle-iso} above.
Then 
\[H^i_D(X_\Zar, \tau_{\ge r}\Z(r)_{X|D})= 0, \quad \text{if } i>d+r,\]
and the natural map 
\[H^{d-1}(X_\Zar, \sH^1_D(\sH^r(\Z(r)_{X|D}))\surj H^{d+r}_D(X_\Zar, \tau_{\ge r}\Z(r)_{X|D})\]
is surjective. Furthermore the same statements hold when we replace $X_\Zar$ and $\Z(r)_{X|D}$
by $X_\Nis$ and $\Z(r)_{X|D,\Nis}$, respectively.
\end{lem}
\begin{proof}
We do the proof for $\Z(r)_{X|D}$; it works the same way for $\Z(r)_{X|D,\Nis}$.
Considering the spectral sequence
\[E_2^{a,b}= H^a(X, \sH^b_D(\tau_{\ge r}\Z(r)_{X|D}))
         \Rightarrow H^*_D(X,\tau_{\ge r}\Z(r)_{X|D})\]
we see that by Lemma \ref{lem:rel-cycle-local-support} it suffices to show
\eq{lem:rel-cycle-global-support1}{H^a(X, \sH^0_D(\sH^b(\Z(r)_{X|D})))= 0,\quad 
                                                                     \text{for }b\ge r \text{ and }a+b\ge r+d}
and
\eq{lem:rel-cycle-global-support2}{H^a(X,\sH^1_D(\sH^r(\Z(r)_{X|D})))=0, \quad \text{for }a\ge d.}
For a closed immersion $i_A:A\inj X$ denote by 
$i_A^!: (\text{abelian sheaves on } X)\to (\text{abelian sheaves on } A)$ the unique functor 
which satisfies $i_{A*}i_A^!=\ul{\Gamma}_A=\sH^0_A$, see \cite[Exp. I, 1.]{SGA2}. 
(For the Nisnevich case, see \cite[1.23]{Nis}.) We obtain 
\[H^a(X,\sH^1_D(\sH^r(\Z(r)_{X|D})))= H^a(D, i^!_D\sH^1_D(\sH^r(\Z(r)_{X|D})).\]
Hence the vanishing \eqref{lem:rel-cycle-global-support2} follows directly 
from Grothendieck's general vanishing theorem \cite[Thm 3.6.5]{Tohoku} by which
the cohomological dimension of a noetherian scheme is less or equal to its Krull dimension.
(For the Nisnevich case, see \cite[1.32]{Nis}.)

Next we prove \eqref{lem:rel-cycle-global-support1}.
Denote by $\Phi_{X|D}^n$ the family of supports on $X$ consisting of all closed
subschemes $A\subset X$ of dimension $\dim(A)\le n$ which intersect $D$ properly.
Denote by $\Phi_{X|D}^n\cap D$ the smallest family of supports which contains all
closed subsets of the form $A\cap D$, with $A\in \Phi^n_{X|D}$.
Notice that  
\eq{lem:rel-cycle-global-support3}{\dim(A)\le n-1,\quad  \text{for } A\in \Phi^n_{X|D}\cap D.}
If $Z\subset U\times \square^{2r-b}$ is an integral closed subscheme of codimension $r$,
then the closure $Z_0\subset X$ of its image under the projection to $U$ lies in
$\Phi^{d+r-b}_{X|D}$. Since $\sH^b(\Z(r)_{X|D})$ is the sheaf on $X_\Zar$ associated to
$V\mapsto CH^r(V|D_{V}, 2r-b)$ we obtain
\[\sH^b(\Z(r)_{X|D})=\sH^0_{\Phi^{d+r-b}_{X|D}}(\sH^b(\Z(r)_{X|D}))\] 
and
\[\sH^0_D(\sH^b(\Z(r)_{X|D}))=\sH^0_{\Phi^{d+r-b}_{X|D}\cap D}(\sH^b(\Z(r)_{X|D}))
                 = \varinjlim_{A\in \Phi^{d+r-b}_{X|D}\cap D }i_{A*} i_A^!(\sH^b(\Z(r)_{X|D})).\]
Thus 
\[H^a(X, \sH^0_D(\sH^b(\Z(r)_{X|D})))= 
 \varinjlim_{A\in \Phi^{d+r-b}_{X|D}\cap D} H^a(A, i_A^!(\sH^b(\Z(r)_{X|D}))),\]
which is zero for $a+b\ge d+r$ since
 the cohomological dimension of $A$ is $\le d+r-b-1$ by \eqref{lem:rel-cycle-global-support3}
and  \cite[Thm 3.6.5]{Tohoku}. (For the Nisnevich case, use \cite[1.24]{Nis} to get the equality above
and then apply \cite[1.32]{Nis} to obtain the vanishing.)
\end{proof}

\begin{proof}[Proof of Theorem \ref{thm:rel-cycle-iso}.]
In the following, the subscript $\sigma\in\{\Zar,\Nis\}$ indicates in which topology we are.
Denote by $j: U=X\setminus D\inj X$ the inclusion of the complement of $D$.
First notice that each abelian sheaf $\sF$ on $X_\sigma$ has a $\Gamma(X_\sigma,-)$-acyclic resolution 
of length $d$.
(Take $\tau_{\le d}$ of the Godement resolution of $\sF$; it is $\Gamma(X_\sigma,-)$-acyclic by 
\cite[Thm 3.6.5]{Tohoku} (resp. \cite[1.32]{Nis}); see \cite[2.18]{Nis} for the Godement resolution in case
$\sigma=\Nis$.)
This shows that $H^i(X_\sigma,\tau_{<r}\Z(r)_{X|D})=0$, for all $i\ge d+r$, and hence
\[H^i_{\sM,\sigma}(X|D,\Z(r))= H^i(X_\sigma,\tau_{\ge r}\Z(r)_{X|D,\sigma}),\quad \text{for }i\ge d+r.\]
We have an exact sequence 
\[H^{i+r}_D(X_\sigma, \tau_{\ge r}\Z(r)_{X|D,\sigma})\to H^{i+r}(X_\sigma, \tau_{\ge r}\Z(r)_{X|D,\sigma})\to 
                 H^{i+r}(U_\sigma, \tau_{\ge r}\Z(r)_{U,\sigma}).\]
For $i>d$ the left hand side vanishes by Lemma \ref{lem:rel-cycle-global-support} and 
the right hand side is isomorphic to $H^i(U_\sigma,\sK^M_{r,U,\sigma})$ and hence also vanishes.
This yields the first part of the theorem. 

It remains to prove that $\phi^{d,r}_{X|D,\sigma}$ is an isomorphism for $\sigma=\Nis$
and if $D_{\rm red}$ is smooth also for $\sigma=\Zar$. {\em In the following $\sigma=\Zar$
is only allowed in the case where $D_{\rm red}$ is smooth.}
By Corollary \ref{cor:cycle-maps}  and Corollary \ref{cor:cycle-maps-Nis} we have a commutative diagram 
\[\xymatrix{ H^{d+r-1}(U_\sigma, \sZ_U)\ar[r]\ar[d]^{\phi^{d-1,r}_{U,\sigma}}_\simeq & 
                 H^{d+r}_D(X_\sigma, \sZ_{X|D})\ar[r]\ar[d]^{\phi^{d,r}_{D\subset X,\sigma}} &
                H^{d+r}(X_\sigma, \sZ_{X|D})\ar[r]\ar[d]^{\phi^{d,r}_{X|D,\sigma}}& 
                 H^{d+r}(U_\sigma, \sZ_{U})\ar[d]^{\phi^{d,r}_{U,\sigma}}_\simeq\\
                H^{d-1}(U_\sigma,\sK_{U})\ar[r] & H_D^d(X_\sigma,\sK_{X|D})\ar[r] & 
                            H^d(X_\sigma,\sK_{X|D})\ar[r] &
                  H^d(U_\sigma,\sK_{U}),}\]
where we use the short hand notations $\sZ_{U}=\tau_{\ge r}\Z(r)_{U,\sigma}$,
$\sZ_{X|D}=\tau_{\ge r}\Z(r)_{X|D,\sigma}$, $\sK_U=\sK^M_{r,U,\sigma}$ and 
$\sK_{X|D}=\sK^M_{r,X|D,\sigma}$,
 the rows are the localization exact sequences and the  
maps $\phi^{i,r}_{U,\sigma}$ are isomorphisms (see Corollary \ref{cor:MilnorChow}).
Since $H^{d+r+1}_D(X_\sigma,\sZ_{X|D})=0$ by Lemma \ref{lem:rel-cycle-global-support}
the map $H^{d+r}(X_\sigma, \sZ_{X|D})\to  H^{d+r}(U_\sigma, \sZ_{U})$ is surjective.
Hence it suffices to show that $\phi^{d,r}_{D\subset X,\sigma}$ is an isomorphism.
For $b\ge 2$ we have by Corollary \ref{cor:MilnorKcompSNCD} 
\[\sH^b_D(\sK^M_{r,X|D,\sigma})\cong R^{b-1}j_*\sK^M_{r, U,\sigma}=0.\]
Also $\sH^0_D(\sK^M_{r,X|D,\sigma})=0$, since $\sK^M_{r,X|D,\sigma}$ is by definition a subsheaf
of $j_*\sK^M_{r,U,\sigma}$. Thus 
$H^d_D(X_\sigma,\sK^M_{r,X|D,\sigma})=H^{d-1}(X_\sigma,\sH^1_D(\sK^M_{r, X|D,\sigma}))$.
We have a commutative diagram
\[\xymatrix{H^{d-1}(X_\sigma, \sH^1_D(\sH^r(\Z(r)_{X|D,\sigma})))\ar@{->>}[r]\ar[d] &
                H^{d+r}_D(X_\sigma,\tau_{\ge r}\Z(r)_{X|D,\sigma})\ar[d]^{\phi^{r,d}_{D\subset X,\sigma}}\\
              H^{d-1}(X_\sigma,\sH^1_D(\sK^M_{r, X|D,\sigma}))\ar[r]^{\simeq} &
                     H^d_D(X_\sigma,\sK^M_{r,X|D,\sigma}),  }\]
in which the top horizontal map is surjective by Lemma \ref{lem:rel-cycle-global-support}.
Thus it suffices to show that the map
\eq{thm:rel-cycle-iso1}{\sH^1_D(\sH^r(\Z(r)_{X|D,\sigma}))\to \sH^1_D(\sK^M_{r,X|D,\sigma})}
induced by $\phi^r_{X|D,\sigma}$ is an isomorphism.
To this end, consider the following commutative diagram
\[\xymatrix{ & \sH^r(\Z(r)_{X|D,\sigma})\ar[r]\ar@{->>}[d]^{\phi^{r}_{X|D,\sigma}} 
                                    & j_*\sH^r(\Z(r)_{U,\sigma})\ar[r]\ar[d]_{\simeq}^{\phi^{r}_{U,\sigma}}&
                   \sH^1_D(\sH^r(\Z(r)_{X|D,\sigma}))\ar[r]\ar[d]^{\eqref{thm:rel-cycle-iso1}} & 0\\
               0\ar[r] & \sK^M_{r,X|D,\sigma}\ar[r] & j_*\sK^M_{r,U,\sigma}\ar[r] & 
               \sH^1_D(\sK^M_{r,X|D,\sigma})\ar[r] &0.
 }\]
Here the rows are the localization exact sequences, the map $\phi^{r}_{U,\sigma}$ is an isomorphism
(see Corollary \ref{cor:MilnorChow}) and $\phi^{r}_{X|D,\sigma}$ is surjective 
by Proposition \ref{prop:rel-cycle-map-surj-Nis}. It follows that $\eqref{thm:rel-cycle-iso1}$
is an isomorphism. This finishes the proof of the theorem.
\end{proof}

\begin{rmk}
\begin{enumerate}
\item It follows from Corollary \ref{cor:cycle-maps-Nis} and Lemma \ref{lem:rel-cycle-local-support}
that the obstruction for the cycle  map $\phi^r_{X|D,\Nis}: \tau_{\ge r}\Z(r)_{X|D,\Nis}\to \sK^M_{r,X|D,\Nis}[-r]$
to be an isomorphism is the non-vanishing of 
\[\sH^0_D(\sH^n(\Z(r)_{X|D,\Nis}))=\Ker(\sH^n(\Z(r)_{X|D,\Nis})\to j_*\sH^n(\Z(r)_{U,\Nis})),
      \quad \text{for } n\ge r.\]
Indeed, if this vanishing holds, then $\sH^n_D(\Z(r)_{X|D,\Nis})=0$, for all $n\ge r$ and $n\neq r+1$, and
$\sH^{r+1}_D(\Z(r)_{X|D,\Nis})=\sH^1_D(\sH^r(\Z(r)_{X|D,\Nis}))$ by Lemma \ref{lem:rel-cycle-local-support}. 
Hence $\sH^n(\Z(r)_{X|D,\Nis})= \sH^n(Rj_*\Z(r)_{U,\Nis})=0$, for all $n\ge r+2$. For $n=r$ we obtain a commutative diagram 
\[\xymatrix{0\ar[r]&  \sH^r(\Z(r)_{X|D,\Nis})\ar[r]\ar@{->>}[d] & 
                         j_*\sH^r(\Z(r)_{U,\Nis})\ar[r]\ar[d]^\simeq & \sH^1_D(\sH^r(\Z(r)_{X|D,\Nis}))\ar[r] & 0\\
                  0\ar[r]&    \sK^M_{r,X|D,\Nis}\ar[r] & j_*\sK^M_{r,U,\Nis}. & &
}\] 
Here the rows are exact, the left vertical map is surjective by Proposition \ref{prop:rel-cycle-map-surj-Nis} and
the right vertical map is bijective by the Nisnevich version of 
Corollary \ref{cor:MilnorChow}. It follows that the left vertical map is an 
isomorphism.  Furthermore the right exactness of the top row yields that the natural map
$\sH^{r+1}_D(\Z(r)_{X|D,\Nis})\to \sH^{r+1}(\Z(r)_{X|D,\Nis})$ is the zero map and hence
$\sH^{r+1}(\Z(r)_{X|D,\Nis})\subset \sH^{r+1}(Rj_*\Z(r)_{U,\Nis})=0$.

If we assume that $D_{\rm red}$ is smooth, then a similar remark applies for the corresponding  Zariski statement.

\item Going back through the proofs, one easily checks that the commutative diagram \eqref{thm:rel-cycle-iso0} 
            of isomorphisms exists for all divisors $D$ 
      whose support has simple normal crossings and which satisfies that
      \eqref{prop:gr0-map1} is an isomorphism (for all $(\fm,\nu)$).
\end{enumerate}
\end{rmk}

\begin{cor}\label{cor:modDiv}
Let $X$ be a smooth curve over $k$ and $D$ an effective divisor on $X$. Then we have isomorphisms
\[H^{2}_{\sM,\Nis}(X|D,\Z(1))\cong H^1(X_\Zar, \sO_{X|D}^\times)\cong \CH^1(X|D, 0).\]
\end{cor}
\begin{proof}
The first isomorphism follows from Theorem \ref{thm:rel-cycle-iso} and Hilbert 90.
The second isomorphism is classical and follows from the fact that the two term complex
\[ K^\times\to \bigoplus_{x\in (X\setminus|D|)^{(1)}}i_{x*} (K^\times/\sO^\times_{X,x} )
       \oplus\bigoplus_{x\in |D|^{(0)}}    i_{x*}( K^\times/(1+\fm_x^{n_x}))\]
is the Cousin resolution of $\sO_{X|D}^\times$.
\end{proof}

\begin{rmk}\label{rmk:not-iso}
Let $(X,D)$ be as in Theorem  \ref{thm:rel-cycle-iso} with $d=\dim X$.
We have a natural map
\eq{rmk:not-iso1}{\CH^1(X|D,1-d)\to H^{d+1}_{\sM,\Nis}(X|D,\Z(1)).}
If $d=1$, this is an isomorphism  by Corollary \ref{cor:modDiv}.
But Theorem \ref{thm:rel-cycle-iso} implies that it is in general not an isomorphism
for $d\ge 2$. Indeed, assume $d= 2$, then the left-hand side vanishes, whereas the right-hand side is equal
to $H^2(X_\Nis, \sO_{X|D}^\times)$. The short exact sequence of Nisnevich sheaves 
$0\to \sO_{X|D}^\times\to \sO_X^\times\to i_*\sO_D^\times\to 0$ induces an isomorphism of
$H^2(X_{\Nis},\sO_{X|D}^\times)$ with the cokernel of $\Pic(X)\to \Pic(D)$.
But in general, this cokernel will not be zero, since not every line bundle on $D$ lifts to a line bundle  on $X$.
Note that this non-vanishing already occurs for reduced and irreducible $D$.
In particular, also the Zariski version of \eqref{rmk:not-iso1} is not an isomorphism in general.
For further counter examples in this spirit see Theorem \ref{thm:Anvan}, \eqref{thm:Anvan1}.
\end{rmk}

\section{Motivic cohomology of {$(\A^1,(m+1)\cdot\{0\})$}}

\subsection{Big de Rham-Witt Complex}
A truncation set $S$ is a subset of the positive integers with the property that 
 a positive integer $s$ is an element of $S$ if and only if all positive divisors of $s$ are contained in $S$.
Examples are the sets $\{1,2,\ldots, m\}$ and $P=\{1, p, p^2, \ldots\}$, for $p$ a prime number.  
For a truncation set $S$  and $n\in \N$ we define the new truncation set  
$S/n:=\{s\in S\,|\, ns\in S\}$. Notice that $S/n$ is the empty set if and only if $n\not\in S$.
We denote by $J$ the category of truncation sets, where the morphisms are inclusions.
We denote by $(\bf dga_\Z)$ the category of differential graded $\Z$-algebras in the sense of \cite[0, 3.1]{IlDRW}.

Let $R$ be a ring containing a field. 
Recall (see e.g. \cite[4]{He}) that the big de Rham-Witt complex of $R$ is a functor
\[J^{\rm op}\to  ({\bf dga_\Z}),\quad S\mapsto \W_S\Omega^\cdot_R,\]
that takes limits to colimits and
which is equipped with graded ring homomorphisms, called Frobenius morphisms,
\[F_n: \W_S\Omega^\cdot_R\to \W_{S/n}\Omega^\cdot_R , \quad S\in J, n\in \N,\]
and homomorphisms of graded groups, called Verschiebung morphisms,
\[V_n: \W_{S/n}\Omega^\cdot_R\to \W_S\Omega^\cdot_R,\quad S\in J, n\in \N.\]
These maps are in fact natural transformations between functors on $J$ (in the obvious sense)
and satisfy various relations, see \cite[Def 4.1]{He}. 
Notice that since $R$ is defined over a field we have $\dlog[-1]=0\in \W_S\Omega^1_R$ for all $S$ 
(see \cite[Rmk 4.2, (c)]{He}) and $\W_S\Omega^\cdot_R$  is a quotient of $\Omega^\cdot_{\W_S(R)/\Z}$.
This implies that $\W_S\Omega^\cdot_R$ is really a differential graded algebra in this case; in particular
the relation $x\cdot x=0$ for a homogeneous element $x\in \W_S\Omega^\cdot_R$ of odd degree holds.

Some facts: $\W_\emptyset\Omega^\cdot=0$, $\W_S\Omega^0_R=\W_S(R)=$ the ring of big Witt vectors,
$\W_{\{1\}}\Omega^\cdot_R=\Omega^\cdot_{R/\Z}=:\Omega^\cdot_R$ and for a finite truncation set $S$  
the dga $\W_S\Omega^\cdot_R$ is a quotient of $\Omega^\cdot_{\W_S(R)/\Z}$.
It follows that the restriction maps $\W_S\Omega^\cdot_R\to \W_T\Omega^\cdot_R$ ($T\subset S$) are surjective.
Finally if $R$ is defined over a field of positive characteristic $p$, we have
  $\W_{\{1,p,\ldots, p^{n-1}\}}\Omega^\cdot_R=W_n\Omega^\cdot_R$
 the $p$-typical de Rham-Witt complex of length $n$ of Bloch-Deligne-Illusie. When working with
the  $p$-typical de Rham-Witt complex we write $F^s=F_{p^s}$ and $V^s= V_{p^s}$. We set 
$\W_m\Omega^\cdot_R:=\W_{\{1,2,\ldots, m\}}\Omega^\cdot_R$.

\begin{lem}\label{lem:deRham-direct-limit}
Let $k$ be a field  and $(R_i)_{i\in I}$ a direct system  system of $k$-algebras.
Set $R=\varinjlim_{i\in I} R_i$.
Then for all finite truncation sets $S$ we have
\[\W_S\Omega^q_{R}=\varinjlim_{i\in I} \W_S\Omega^q_{R_i}.\]
\end{lem}
\begin{proof}
For a finite truncation set $S$,  we put $E^\cdot_S:=\bigoplus_{q\ge 0}\varinjlim_{i\in I}\W_S\Omega^q_{R_i}$.
We have a natural map of graded rings $E_S^\cdot\to \W_S\Omega^\cdot_R$.
Furthermore for a general truncation set $S$ we define $E^\cdot_S:=\varprojlim_{T\subset S} E^\cdot_{T}$,
where the limit is over all finite truncation sets $T$ contained in $S$ and the transition maps are 
induced by the obvious restriction maps. It then straightforward to check that  $S\mapsto E_S^\cdot$
is a Witt complex over $R$ (in the sense of \cite[Def 4.1]{He}). Since  $\W_{-}\Omega^\cdot_R$ is the initial
object in the category of Witt complexes, we obtain a morphism of graded rings
$\W_S\Omega^\cdot_R\to E_S^\cdot$ for all truncation sets $S$. For a finite truncation set
this map is clearly inverse to the natural map above.
\end{proof}

\begin{no}\label{big-vs-p-typical} 
{\em Relation big - and {p}-typical de Rham-Witt.}
Let $R$ be a ring containing a field of characteristic exponent $p\ge 1$ and $S\in J$ be a finite truncation set. 
Set
\[\epsilon_S:=
    \prod_{\genfrac{}{}{0pt}{}{\text{primes }\ell\in S}{\ell\neq p}}(1-\tfrac{1}{\ell}V_{\ell}(1))\in \W_S(R),\]
where the product is over all primes $\ell\in S$ different from $p$.
Then for all $q\ge 0$  there is an isomorphism of abelian groups
\[\W_S\Omega^q_R\xr{\simeq} \prod_{\genfrac{}{}{0pt}{}{j\in S}{(j,p)=1}} \W_{S/j\cap P}\Omega^q_R,\quad
   \alpha\mapsto (F_j(\alpha)_{|S/j\cap P})_j,\]
where $P=\{1,p,p^2,\ldots\}$, with inverse map given by
\eq{big-vs-p-typical1}{\prod_{\genfrac{}{}{0pt}{}{j\in S}{(j,p)=1}}\W_{S/j\cap P}\Omega^q_R
                  \xr{\simeq} \W_S\Omega^q_R,
                    \quad  (\alpha_j)\mapsto \sum_{j}\frac{1}{j} V_j(\epsilon_{S/j}\tilde{\alpha}_j),}
where $\tilde{\alpha}_j\in \W_{S/j}\Omega^q_R$ is some lift of $\alpha_j\in \W_{S/j\cap P}\Omega^q_R$.
These isomorphisms are functorial in $S$ in the obvious sense.
(See \cite[1.2]{HeMa01} or \cite[Thm 1.11]{Kay}.)
\end{no}

\begin{no}{}\label{big-de-Rham-Witt-sheaf}
Let $X$ be a scheme over a field and $S$ a truncation set. 
Then there is a unique sheaf of groups $\W_S\Omega^q_{X}$ on $X$
such that for any open affine $U=\Spec R\subset X$ we have $\Gamma(X, \W_S\Omega^q)=\W_S\Omega^q_R$.
Indeed, this is true for the $p$-typical de Rham-Witt  and therefore if $S$ is a finite truncation set 
we have to set 
\[\W_S\Omega^q_X:=\prod_{\genfrac{}{}{0pt}{}{j\in S}{(j,p)=1}} \W_{S/j\cap P}\Omega^q_X\]
and if $S$ is infinite then $\W_S\Omega^q_X:=\varprojlim_{T\subset S }\W_T\Omega^q_X$, where the
limit is over all finite subsets. Clearly all the structure maps sheafify. 
Notice that $\W_S\Omega^0_X=\W_S\sO_X$ is the sheaf of big Witt vectors over $X$.
\end{no}

\begin{rmk}\label{rmk:DRW-graded-structure-char0}
In case $p=1$ the isomorphism from \ref{big-vs-p-typical} above
 has the shape $\W_m\Omega^q_R\cong \prod_{j=1}^m \Omega^q_R$.
It is direct to check that under this isomorphism the restriction $\W_m\Omega^q_R\to \W_{m-1}\Omega^q_R$
is given by projecting to the first $m-1$-components. In particular we have an exact sequence
\[0\to \Omega^q_R\xr{\frac{1}{m}V_m} \W_m\Omega^q_R\to\W_{m-1}\Omega^q_R\to 0. \]
\end{rmk}

\begin{no}\label{BK-graded-group}
Let $k$ be a perfect field of characteristic $p>0$ and $R$ an essentially smooth $k$-algebra.
Let $C^{-1}: \Omega^q_R\to \Omega^q_R/B_1^q$ be the inverse Cartier operator,
where $B_1^q=d\Omega^{q-1}_R$. Recall that it is injective with image $Z_1^q/B_1^q$, where
$Z_1^q=\Ker( d:\Omega^q_R\to \Omega^{q+1}_R)$. 
We obtain a chain of subgroups (see e.g. \cite[(1.3)]{BK})
\[0=B_0^q\subset B_1^q\subset\ldots \subset B^q_n\subset B^q_{n+1}\subset\ldots \subset Z_{r+1}^q\subset Z_r^q
\subset \ldots\subset Z_1^q\subset Z_0^q:=\Omega^q_R,\]
where  by definition $C^{-1}(B_i^q)= B^q_{i+1}/B^q_1$ and $C^{-1}(Z_i^q)=Z_{i+1}^q/B_1^q$, for $i\ge 0$.
Notice that we can iterate the inverse Cartier operator $n$-times to obtain a morphism
\[C^{-n}: \Omega^q_R\to \Omega^q_R/B^q_n,\]
which is injective and has image equal to $Z_n^q/B_n^q$. By convention $C^{-0}=\id$.

Let $m\ge 1$ be an integer and write $m=m_1p^s$ with $(m_1,p)=1$ and $s\ge 0$. 
Following \cite[(4.7)]{BK} we define 
\[\theta:\Omega^{q-1}_R\to(\Omega^q_R/B^q_s)\oplus (\Omega^{q-1}_R/B^{q-1}_s),
\quad \alpha\mapsto (C^{-s}(d\alpha), (-1)^{q-1}m_1 C^{-s}(\alpha))\]
and 
\[{\rm gr}^q_m(R):=\Coker(\theta:\Omega^{q-1}_R \to(\Omega^q_R/B^q_s)\oplus (\Omega^{q-1}_R/B^{q-1}_s)).\]
(This is the group denoted by $^mG^{q+1}_n$ in \cite[(4.7)]{BK}, for $n>s$.) 
\end{no}

\begin{prop}[cf. {\cite[I, Cor 3.9]{IlDRW}, \cite[Thm (4.4)]{HK}}]\label{prop-DRW-graded-structure}
In the above situation let $m$ be a positive integer and write $m=m_1p^s$ with $(m_1,p)=1$ and $s\ge 0$.
Then there is an exact sequence of groups
\[0\to {\rm gr}^q_m(R)\to \W_m\Omega^q_R\to\W_{m-1}\Omega^q_R\to 0,\]
where the map on the right is given by restriction and the map on the left is induced by 
\[\Omega^q_R\oplus\Omega^{q-1}_R\to \W_m\Omega^q_R, \quad 
          (\alpha, \beta)\mapsto V_m(\alpha)+(-1)^q dV_m(\beta).\]
\end{prop}
\begin{proof}
For $j\in \{1,2,\ldots, m\}$ with $(j,p)=1$ denote by $n(j,m)$ the unique integer $\ge 1$ satisfying
\[j p^{n(j,m)-1}\le m< j p^{n(j,m)}.\]
We get 
\[n(j,m)=\begin{cases} s+1, &\text{if }j=m_1\\
                                  n(j,m-1),& \text{else.} \end{cases} \]
Hence under the isomorphism from \ref{big-vs-p-typical} the restriction $\W_m\Omega^q_R\to \W_{m-1}\Omega^q_R$
becomes
\[\left(\prod_{\genfrac{}{}{0pt}{}{1\le j\le m}{j\neq m_1, (j,p)=1}} W_{n(j,m)}\Omega^q_R\right)
                 \times W_{s+1}\Omega^q_R  \to 
   \left(\prod_{\genfrac{}{}{0pt}{}{1\le j\le m}{j\neq m_1, (j,p)=1}} W_{n(j,m)}\Omega^q_R\right)
                    \times W_s\Omega^q_R, \]
which is the identity on the first component and the restriction $W_{s+1}\to W_s$ on the second.
(Here  $W_s\Omega^q_R=0$ for $s=0$, by convention.) Thus the kernel of 
$\W_m\Omega^q_R\to \W_{m-1}\Omega^q_R$ is given by the image of 
\[{\rm gr}^s W\Omega^q_R:=\Ker(W_{s+1}\Omega^q_R\to W_s\Omega^q_R)\]
under the isomorphism of \ref{big-vs-p-typical}. By \cite[I, Cor. 3.9]{IlDRW} there is a surjection
\[\psi:\frac{\Omega^q_R}{B_s^q}\oplus \frac{\Omega^{q-1}_R}{Z^q_{s+1}}\to {\rm gr}^s W\Omega^q_R,
\quad (\alpha,\beta)\mapsto m_1V^s(\alpha)+ (-1)^q dV^s(\beta)\]
with 
\[\Ker\, \psi= 
\left\{(\alpha,\beta)\in \frac{B^q_{s+1}}{B^q_s}\oplus \frac{Z_s^{q-1}}{Z_{s+1}^{q-1}}\,|\,
                                                    m_1V^s(\alpha)=(-1)^{q-1}dV^s(\beta)\right\}.\]
It follows that for any $(\alpha,\beta)\in \Ker \,\psi$ there exist elements $\alpha', \beta'\in \Omega^{q-1}_R$ 
with
\[(\alpha, \beta)= (C^{-s}(d\alpha'), C^{-s}(\beta')).\]
Now take any $\alpha'', \beta''\in W_{s+1}\Omega^{q-1}_R$ lifting $\alpha', \beta'$. Then by \cite[I, Prop. 3.3]{IlDRW}
\[\alpha=C^{-s}(d\alpha')\equiv F^s(d\alpha'') \text{ mod } B^q_s, \quad 
     \beta\equiv C^{-s}(\beta')\equiv F^s(\beta'') \text{ mod }Z^{q-1}_{s+1}.\]
Now $m_1V^s(\alpha)=(-1)^{q-1}dV^s(\beta)$ yields
\[m_1p^s d\alpha''= (-1)^{q-1}p^sd\beta''\quad \text{in } W_{s+1}\Omega^q_R.\]
Since the map $\Omega^{q-1}_R\to W_{s+1}\Omega^{q-1}_R$ given by 
lifting and multiplying with $p^s$ is injective by \cite[I, Prop. 3.4]{IlDRW}, we obtain
\[\beta'\equiv m_1(-1)^{q_1}\alpha' \text{ mod } Z^{q-1}_1.\]
Define 
\[\theta':\Omega^{q-1}_R\to (\Omega^q_R/B^q_s)\oplus (\Omega^{q-1}_R/Z^{q-1}_{s+1}),
\quad \alpha\mapsto (C^{-s}(d\alpha), m_1(-1)^{q-1}C^{-s}(\alpha)).\]
We obtain
\[\Ker\,\psi= \im\, \theta'.\]
There is a natural surjection 
\[{\rm gr}^q_m(R)=\Coker\,\theta\surj \Coker\,\theta'.\]
This map is in fact an isomorphism, as
follows directly from the observation $\Ker\, \theta'= Z^{q-1}_1$ and the Snake Lemma.
Altogether we see that $\psi$ induces an isomorphism ${\rm gr}^q_m(R)\xr{\simeq}  {\rm gr}^s W\Omega^q_R$.
Finally the composition of $\psi$ with the isomorphism \eqref{big-vs-p-typical1} 
 sends $(\alpha,\beta)\in \Omega^q_R\oplus\Omega^{q-1}_R$ to 
\begin{align*}
      &\tfrac{1}{m_1}V_{m_1}(\epsilon_{S/m_1}m_1V_{p^s}(\alpha)) +
                              \tfrac{1}{m_1}V_{m_1}(\epsilon_{S/m_1}(-1)^q d V_{p^s}(\beta) )\\
       =& V_{m_1p^s}(F_{p^s}(\epsilon_{S/m_1})\alpha) 
           +(-1)^q d V_{m_1p^s}(F_{p^s}(\epsilon_{S/m_1})\beta)    \\
      = &V_m(\alpha)+(-1)^q dV_m(\beta),         
\end{align*}
where we set $S:=\{1,\ldots, m\}$ and view $V_{p^s}$ as map $\W_{\{1\}}=\W_{S/m_1p^s}\to \W_{S/m_1}$.
This finishes the proof.
\end{proof}

\begin{prop}[cf. {\cite[Prop. 4.6]{HK}}]\label{prop-big-DRW-presentation}
Let $k$ be a field, $X$ a regular scheme over $k$ and $S$ a finite truncation set. 
Then there is a surjective morphism
\eq{prop-big-DRW-presentation1}{
(\W_S\sO_X\otimes_\Z \wedge^q_\Z \sO_X^\times)\oplus 
 (\W_S\sO_X\otimes_\Z \wedge^{q-1}_\Z \sO_X^\times)\to       
                    \W_S\Omega^q_X,}
which on local sections is defined by
\[(w\otimes a_1\wedge\ldots\wedge a_q, 0)\mapsto w\dlog[a_1]\cdots \dlog[a_q]\]
and 
\[(0, w\otimes a_1\wedge\ldots\wedge a_{q-1})\mapsto d w\dlog[a_1]\cdots \dlog[a_{q-1}],\]
where $[-] : \sO_X^\times \to \W_S\sO_X^\times$ denotes the Teichm\"uller lift.
Furthermore, if $F\subset k$ is the prime field of $k$, the kernel of this map is the 
sheaf of $\W_S(F)$-modules generated by the local sections
\eq{prop-big-DRW-presentation2}{
(V_n([a_1])\otimes a_1\wedge\ldots\wedge a_q,0)- n(0, V_n([a_1])\otimes a_2\wedge\ldots \wedge a_q),\quad
                   a_i\in \sO_X^\times, n\in S.}
\end{prop}
\begin{proof}
Denote by $E_S$ the sheaf on the left-hand side of \eqref{prop-big-DRW-presentation1} and by
$K_S$ the sheaf of $\W_S(F)$-modules generated by the elements $\eqref{prop-big-DRW-presentation2}$.
Clearly there is a well-defined and unique morphism $E_S\to \W_S\Omega^q_X$ as in the statement.
Further the relations $d\W_S(F)=0$, $n d V_n=V_n d$ and $V_n(\alpha \dlog[a])=V_n(\alpha)\dlog[a]$ 
imply that $K_S$ lies in the kernel of this map. The rest of the statement is local. Hence we may assume
that $X$ is the spectrum of a regular local $k$-algebra $R$.
By \cite[(2.7) Cor]{Popescu86} $R$ is a filtered direct limit of local rings which are essentially smooth over $F$.
Hence by Lemma \ref{lem:deRham-direct-limit} we can assume that $R$ is essentially
 smooth over $F$. Consider the following group homomorphism
\eq{prop-big-DRW-presentation3}{\prod_{\genfrac{}{}{0pt}{}{j\in S}{(j,p)=1}} E_{S/j\cap P}\to E_S}
given by 
\[(x_j\otimes a_j, y_j\otimes b_j)_{j}\to 
\sum_j\bigg(\tfrac{1}{j} V_j(\epsilon_{S/j}\tilde{x}_j)\otimes a_j, \, 
                    \tfrac{1}{j} V_j(\epsilon_{S/j}\tilde{y}_j)\otimes b_j\bigg),\]
where $x_j, y_j\in \W_{S/j\cap P}(R)$, $\tilde{x}_j, \tilde{y}_j\in \W_{S/j}(R)$ are lifts of $x_j,y_j$ and
$a_j\in \wedge^q R^\times$, $b_j\in \wedge^{q-1}R^\times$, the $\epsilon_{S/j}$'s
are the ones from \ref{big-vs-p-typical} and $p$ is the characteristic exponent of $F$. 
The isomorphism \eqref{big-vs-p-typical1} for $q=0$ immediately gives that 
\eqref{prop-big-DRW-presentation3} is an isomorphism. We obtain a commutative square
\[\xymatrix{ E_S\ar[r] & \W_S\Omega^q_R\\
   \prod_{\genfrac{}{}{0pt}{}{j\in S}{(j,p)=1}} E_{S/j\cap P}
         \ar[u]^{\eqref{prop-big-DRW-presentation3}}_\simeq\ar[r] &
   \prod_{\genfrac{}{}{0pt}{}{j\in S}{(j,p)=1}} 
                            \W_{S/j\cap P}\Omega^q_R. \ar[u]_{\eqref{big-vs-p-typical1}}^\simeq
}\] 
In case $p=1$ it is straightforward to check that the bottom horizontal map is surjective with
kernel equal to $\prod_j K_{S/j \cap P}$.
(It suffices to show $E^1_{\{1\}}/K^1_{\{1\}}\cong \Omega^1_R$, which is easily done using the universal
 property of $\Omega^1_R$.) In case $p>1$, this follows from \cite[Prop. (4.6)]{HK}.
(Notice that $\W_{S/j\cap P}(\F_p)$ is a quotient of $\Z$ and hence $K_{S/j\cap P}$ is equal to the
{\em group} generated by the elements \eqref{prop-big-DRW-presentation2}.)
Hence the top map is surjective. It is a direct computation that the vertical arrow on the left-hand side maps
$\prod_j K_{S/j \cap P}$ into $K_S$ (use $\frac{1}{j}V_j(xy)=\frac{1}{j^2}V_j(x)V_j(y)$).
This finishes the proof.
\end{proof}

\begin{rmk}\label{rmk-kernel-generators}
Let $S$ be a finite truncation set and $p\ge 1$ the characteristic exponent of the perfect field $F$.
The $\W_S(F)$-submodule of 
$(\W_S\sO_X\otimes_\Z \wedge^q_\Z \sO_X^\times)\oplus 
(\W_S\sO_X\otimes_\Z \wedge^{q-1}_\Z \sO_X^\times)$
generated by the elements \eqref{prop-big-DRW-presentation2} is actually equal to the 
{\em group} generated by the elements
\eq{rmk-kernel-generators1}{(V_n([\lambda a_1])\otimes a_1\wedge\ldots\wedge a_q,0)- 
                                            n(0, V_n([\lambda a_1])\otimes a_2\wedge\ldots \wedge a_q),}
for $a_i\in \sO_X^\times, \lambda\in F,  n\in S$.

Indeed, take $n,r\in S$, $\lambda\in F$, $a\in \sO_X^\times$ and write $n=n'p^t$ with $(n',p)=1$ and 
$r=r'p^s$ with $(r',p)=1$ and $c:={\rm gcd}(r',n)={\rm gcd}(r',n')$.
Notice $[\lambda]=F_{p^s}[\lambda]\in \W_S(F)$. Then on the one hand we get
\begin{align*}
V_r([\lambda])\cdot\bigg(V_n([a])\otimes a, \,-n V_n(a)\bigg)  = 
                   p^s V_{r'}([\lambda])\cdot\bigg(V_n([a])\otimes a,\,-n V_n([a])\bigg)\\
                  = p^s\tfrac{c^2}{r'}\bigg(V_{\tfrac{nr'}{c}}([\lambda^{\tfrac{n}{c}}a^{\tfrac{r'}{c}}])\otimes 
                          a^{\tfrac{r'}{c}},\,
                      - \tfrac{nr'}{c}V_{\tfrac{nr'}{c}}([\lambda^{\tfrac{n}{c}}a^{\tfrac{r'}{c}}])\bigg).
\end{align*}
On the other hand we have
\begin{align*}
\bigg(V_n([\lambda a])\otimes a, \,- n V_n([\lambda a])\bigg)  = 
                 \tfrac{1}{n'} V_{n'}([\lambda])\cdot  \bigg(V_n([a])\otimes a,\, - n V_n([a])\bigg).	
\end{align*}
\end{rmk}

\subsection{De Rham-Witt and relative Milnor {$K$}-sheaves}\label{dRWrelML}
\begin{no}{}
Let $R$ be a noetherian local domain containing a field. We denote $R((T)):=R[[T]][\frac{1}{T}]$.
 By definition the $r$-th Milnor $K$ group of $R((T))$ 
 is the quotient of $R((T))^{\otimes_\Z r}$  by the subgroup generated by the elements
\[b_1\otimes\ldots \otimes b_{i-1}\otimes a\otimes (1-a)\otimes b_{i+2}\otimes\ldots \otimes b_r,\]
$b_i\in R((T))^\times$, $a, 1-a\in R((T))^\times$. Notice that $R((T))$ is a local ring containing an
infinite field. Hence the relations $\{a,-a\}=0$ and $\{a,b\}=-\{b,a\}$ hold, see e.g. \cite[Lem 2.2]{Kerz09}.
In particular our definition of $K^M_r(R((T)))$  coincides with the one from  \cite[4.]{BK}
and also with $\hat{K}^M_r(R((T)))$ defined in \cite{Kerz09}. Let $K$ be the fraction field of $R$
then the natural map $K^M_{r}(R((T)))\inj K^M_r(K((T)))$ is injective, see \cite[Prop 10]{Kerz09}.

We denote by 
\[U^m K^M_r(R((T)))\]
the subgroup of $K^M_r(R((T)))$ generated by symbols of the form 
\[\{1+xT^m, y_1,\ldots, y_{r-1}\}, \quad x\in R[[T]], y_i\in R((T))^\times.\]
\end{no}

\begin{notation}\label{not:MilnorK-regularscheme}
Let $X$ be a regular connected scheme over a field and $\A^1=\Spec \Z[T]$. For $m\ge 0$ we set 
\[A_X|m:=(X\times_\Z \A^1 , m\cdot( X\times\{0\})).\]
We define $\sK^M_{r,  X\times \A^1}$ as in \ref{MilnorK} (there for a smooth scheme)
and $\sK^M_{r, A_X|m}$ as in Definition \ref{defn:MilnorModulus}.
\end{notation}

\begin{lem}\label{lem-additiveMilnorK-global-local}
We keep the notations from above. 
Let $j:X\times (\A^1\setminus\{0\})\inj X\times \A^1$ be the open immersion and 
$x\in X$  a point. Set $R:=\sO_{X,x}$.
Then for all $m\ge 1$ there is a natural isomorphism
\[(j_*\sK^M_{r, X\times (\A^1\setminus\{0\})}/\sK^M_{r, A_X|m})_x\xr{\simeq}
 K^M_r(R((T)))/U^m K^M_r(R((T))),\]
where we view $x$ via $X\cong X\times\{0\}\inj X\times \A^1$ as a point on $X\times\A^1$.
\end{lem}
\begin{proof}
Set $A:=\sO_{X\times \A^1, x\times\{0\}}$ and $K=\Frac(A)$.
As in Lemma \ref{lem:MilnorKsuppSNCD} and Remark \ref{rmk:MilnorModulus} we have
the following equalities of subgroups of $K^M_r(K(T))$
\[(j_*\sK^M_{r, X\times (\A^1\setminus\{0\})})_{x\times\{0\}}= 
\{(A[\tfrac{1}{T}])^\times,\ldots, (A[\tfrac{1}{T}])^\times\},\] 
\[\sK^M_{r,A_X|m, \,x\times\{0\}}= \{1+T^m A, (A[\tfrac{1}{T}])^\times,\ldots, (A[\tfrac{1}{T}])^\times\}.\]
Since under the natural map $K(T)\inj K((T))$ the ring $A[\frac{1}{T}]$ is mapped into
$R((T))$ and $(1+T^m A)$ is mapped into $(1+T^m R[[T]])$, we obtain a natural map as in the statement.
The inverse map is constructed  in the same way as in Lemma \ref{lem:DVR-cDVR}.
\end{proof}

\begin{no}\label{big-Witt-power series}
We recall (see e.g. \cite[Ex. 1.16]{He}) that for all $m\ge 1$  and all rings $R$, 
there is an isomorphism of groups
\[\gamma: \W_m(R)\xr{\simeq} \frac{1+TR[[T]]}{1+T^{m+1}R[[T]]}, \quad 
            \sum_{n=1}^m V_n([a_n])\mapsto \prod_{n=1}^m (1+a_nT^n).\]
There are different conventions for this isomorphism (see \cite[before Add 1.15]{He}), we pick the one
which is compatible with \cite{BK}.
\end{no}
The following theorem generalizes the above isomorphism to higher degree and
is reminiscent of Bloch's original construction of the $p$-typical de Rham-Witt complex in \cite{BlochDRW}.

\begin{thm}\label{thm-DRW-MilnorK}
Let $X$ be a regular scheme over a field. 
For $r\ge 0$ and $m\ge 1$ there is an isomorphism of sheaves of abelian groups on $X$ 
\eq{thm-DRW-MilnorK1}{\W_m\Omega^r_X\xr{\simeq} \frac{\sK^M_{r+1,A_X|1}}{\sK^M_{r+1,A_X|(m+1)}},}
which sends 
\[w \dlog[a_1]\cdots\dlog[a_r]\mapsto \{\gamma(w), a_1,\ldots, a_r\}\]
and 
\[dw \dlog[a_1]\cdots\dlog[a_{r-1}]\mapsto (-1)^r\{\gamma(w), a_1, \ldots, a_{r-1}, T\}, \]
where $ w\in \W_m\sO_X, a_i\in \sO_X^\times$.
\end{thm}
\begin{proof}
Denote by $F\subset \sO_X$ the prime field and by $p\ge 1$ its characteristic exponent.  
We will need the following lemma:
\begin{lem}\label{lem:vanF} 
With the above notations we have in $\sK^M_{r+1,A_X|1}/\sK^M_{r+1,A_X|(m+1)}$
\[\{1+a_1T^n, \lambda , a_2,\ldots, a_r\}=0,\]
for all $a_i\in \sO_X^\times$, $\lambda\in F$, $1\le n\le m$ and  $r\ge 1$.
\end{lem}
\begin{proof}[Proof of Lemma \ref{lem:vanF}] 
If $p>1$, then $F=\F_p$ and the statement follows directly from $\lambda=\lambda^{p^s}$, 
for $\lambda\in \F_p$ and $s\ge 0$.
Thus we assume $p=1$, i.e. $F=\Q$. 
For all $\nu\ge 1$ we have the map 
\[\dlog: \sK^M_{r+1, A_X|\nu}\to \Omega^{r+1}_{\A^1_X}(\log \{0\}_X)(-\nu\{0\}_X),\]
where $\{0\}_X=X\times \{0\}$,
at our disposal, see the proof of Proposition \ref{prop:higher-gr-map-char0}.
By Proposition \ref{prop:higher-gr-map-char0}, \eqref{prop:higher-gr-map-char01}
and the fact that the composition 
\eqref{prop:higher-gr-map-char03} is equal to the differential, the $\dlog$ map
induces injective maps on the graded pieces
\[\frac{\sK^M_{r+1, A_X|\nu}}{\sK^M_{r+1, A_X|(\nu+1)}}\inj 
\frac{\Omega^{r+1}_{\A^1_X}(\log \{0\}_X)(-\nu\{0\}_X)}
    { \Omega^{r+1}_{\A^1_X}(\log \{0\}_X)(-(\nu+1)\{0\}_X)},\quad \nu\ge 1.\]
Hence also the induced map
\[\dlog_{m+1}:\frac{\sK^M_{r+1, A_X|1}}{\sK^M_{r+1, A_X|(m+1)}}
     \inj 
    \frac{\Omega^{r+1}_{\A^1_X}(\log \{0\}_X)(- \{0\}_X)}
                          { \Omega^{r+1}_{\A^1_X}(\log \{0\}_X)(-(m+1)\{0\}_X)} \]
is injective. Since $\dlog_{m+1}(\{1+a_1T^n, \lambda , a_2,\ldots, a_r\})=0$ the Lemma is proven.
\end{proof}
We resume with the proof of Theorem \ref{thm-DRW-MilnorK}.
By the above Lemma 
the following equality holds in $\sK^M_{r+1,A_X|1}/\sK^M_{r+1,A_X|(m+1)}$,
 for all $a_i\in \sO_X^\times$, $\lambda\in F$ and $1\le n\le m$,
\[\{1+\lambda a_1T^n, a_1, a_2,\ldots, a_r\} \]
\[= \{1+\lambda a_1T^n, \lambda a_1, a_2,\ldots, a_r\}- 
            \{1+\lambda a_1T^n, -\lambda a_1T^n, a_2,\ldots, a_r\}\]
\[=(-1)^{r} n\cdot\{1+\lambda a_1T^n, a_2,\ldots, a_r, T\}.\]
This together with Proposition \ref{prop-big-DRW-presentation} and Remark \ref{rmk-kernel-generators}
directly implies that there is a well-defined map as in the statement.
To show that it is an isomorphism, we may assume that $X$ is the spectrum of a regular local ring
and by \cite[(2.7) Cor]{Popescu86} and Lemma \ref{lem:deRham-direct-limit} we may further assume
that $X=\Spec R$, with $R$ a local ring which is essentially smooth over $F$.

We first assume $p>1$. In view of Lemma \ref{lem-additiveMilnorK-global-local} 
the map defined above has the shape
\[\W_m\Omega^r_R\to U^1K^M_{r+1}(R((T)))/U^{m+1}K^M_{r+1}(R((T))):= U^1/U^{m+1}.\]
This map clearly induces a morphism from the exact sequence from Proposition \ref{prop-DRW-graded-structure},
to the exact sequence 
\[0\to U^m/U^{m+1}\to U^1/U^{m+1}\to U^1/U^m\to 0.\]
The map on the kernels ${\rm gr}^q_m(R)\to U^m/U^{m+1}$ precomposed with the natural surjection
$\Omega^r_R\oplus \Omega^{r-1}_R\to {\rm gr}^q_m(R)$ is given by 
\[(a \dlog b_1\wedge\ldots\wedge \dlog b_r, 0)\mapsto \{1+aT^m, b_1,\ldots, b_r\}\]
and
\[(0, a\dlog b_1\wedge\ldots\wedge\dlog b_{r-1})\mapsto \{1+aT^m, b_1,\ldots, b_{r-1},T\},\]
where $a\in R$, $b_i\in R^\times$.
This is the map $\rho_m$ from \cite[(4.3)]{BK}, which by \cite[Rmk (4.8)]{BK} induces an isomorphism
${\rm gr}^q_m(R)\xr{\simeq} U^m/U^{m+1}$, for all $m\ge 1$. 
Hence \eqref{thm-DRW-MilnorK1} is an isomorphism by induction on $m$.

Now assume $p=1$, i.e. $F=\Q$. In this case the map \eqref{thm-DRW-MilnorK1} induces 
a morphism from the exact sequence  of Remark \ref{rmk:DRW-graded-structure-char0}
to the exact sequence
\[0\to \frac{\sK^M_{r+1,A_R|m}}{\sK^M_{r+1,A_R|(m+1)}}\to 
          \frac{\sK^M_{r+1,A_R|1}}{\sK^M_{r+1,A_R|(m+1)}}\to
         \frac{ \sK^M_{r+1,A_R|1}}{\sK^M_{r+1,A_R|m}} \to 0,\]
where we abuse notations and write $R$ instead of $\Spec R$.
The map  on the kernels is given by
\eq{thm-DRW-MilonorK2}{\Omega^r_R\to \sK^M_{r+1,A_R|m}/\sK^M_{r+1,A_R|(m+1)}}
\[a\dlog b_1\wedge\ldots\wedge \dlog b_r\mapsto \{1+\tfrac{1}{m}aT^m, b_1,\ldots, b_r\},\]
$a\in R$, $b_i\in R^\times$, and it suffices to show that this map is an isomorphism.
With the notation from  
 \eqref{restricted-log-differentials1} the global sections over $\Spec R$ of 
the sheaf $\omega^{r}_{A_R|m, m,1}$ are
\[T^m\cdot R\otimes_R (\Omega^r_R\oplus (\Omega^{r-1}_R\wedge \dlog T))\]
and the differential $d: \omega^{r-1}_{A_R|m,m,1}\to \omega^r_{A_R|m, m ,1}$ is given by
\[d(T^m\otimes (\alpha, \beta\wedge\dlog T))= 
             T^m\otimes(d\alpha, (-1)^{r-1}m  \alpha+ d\beta)\wedge\dlog T).\]
It is direct to check that $\Omega^r_R\to \omega^r_{A_R|m,m,1}/B^r_{A_R|m, m,1}$, 
 $\alpha\mapsto T^m\otimes (\frac{1}{m}\alpha, 0)$ is an isomorphism.  
Hence, by Proposition \ref{prop:higher-gr-map-char0}, the map \eqref{thm-DRW-MilonorK2} 
is an isomorphism as well.  This finishes the proof.
\end{proof}

\begin{cor}\label{cor-mult-p-gradedK}
Let $p$ be a prime number and $R$ be a regular local  $\F_p$-algebra.
Then the multiplication with $p$ on $K^M(R((T)))$ induces an injective homomorphism
\[U^m K^M_r(R((T)))/U^{m+1}K^M_r(R((T)))\xr{p\cdot} U^{pm} K^M_r(R((T)))/U^{pm+1}K^M_r(R((T))),\]
for all $r, m\ge 1$, .
\end{cor}
\begin{proof}
As above, using Lemma \ref{lem-additiveMilnorK-global-local} and \cite[(2.7) Cor]{Popescu86}
we reduce to the case where $R$ is local and essentially smooth over $\F_p$.
In this case, lifting and multiplying with $p$ induces an injective map
$\ul{p}: \W_m\Omega^{r-1}_R\to \W_{pm}\Omega^{r-1}_R$, by \cite[I, Prop 3.4]{IlDRW}
and \ref{big-vs-p-typical}. Hence the statement follows directly from Theorem \ref{thm-DRW-MilnorK}.
\end{proof}

\subsection{Motivic cohomology of {$(\A^1,(m+1)\cdot\{0\})$} and additive Chow groups}
Let $k$ be a field of characteristic $\neq 2$. We write $\A^1_k=\Spec k[T]$.

\begin{no}\label{DRW-addChow}
Recall from \cite[Thm 3.20]{Kay} that with the notation from \ref{not:MilnorK-regularscheme}
there is an isomorphism for all $m, r\ge 1$
\eq{DRW-addChow1}{\theta: \CH^r(A_k|(m+1), r-1)\xr{\simeq} \W_m\Omega^{r-1}_k}
which sends the class of a closed point 
$P\in (\A^1_k\setminus\{0\})\times (\P^1\setminus\{0,1,\infty\})^{r-1}$ to 
\[\theta([P])= \Tr_{k(P)/k}\left(   \frac{1}{[T(P)]} \dlog [y_1(P)]\cdots \dlog [y_{r-1}(P)]\right),\]
where $\Tr_{k(P)/k}: \W_m\Omega^{r-1}_{k(P)}\to \W_m\Omega^{r-1}_k$ is the trace map 
 from \cite[Thm 2.6]{Kay}. 
Let $f\in 1+Tk[T]$ be an irreducible polynomial of degree $\le m$ and denote by $w(f)\in \W_m(k)$ 
the corresponding Witt vector, see \ref{big-Witt-power series}. 
Let $P,Q\in (\A^1\setminus\{0\})\times (\P^1\setminus\{0,1,\infty\})^{r-1}$ be two closed points defined
by the following vanishing sets
\eq{DRW-addChow2}{P=V(f, y_1- b_1,\ldots, y_{r-1}- b_{r-1}), \quad b_i\in k^\times,}
\eq{DRW-addChow3}{Q= V(f, 1-Ty_1 ,y_2 -b_1,\ldots, y_{r-1}-b_{r-2}), \quad b_i\in k^\times.}
Then 
\[\theta(P)=w(f) \dlog [b_1]\cdots \dlog [b_{r-1}]\in \W_m\Omega^{r-1}_k \]
and 
\[\theta(Q)=dw(f) \dlog [b_1]\cdots \dlog [b_{r-2}]\in \W_m\Omega^{r-1}_k.\]
(Indeed, set $L:=k[T]/(f)$ and denote by $t\in L$ the residue class of $T$. Then the above formulas
follow immediately from the fact that $\Tr_{L/k}: \W_{m}\Omega^\bullet_L\to \W_m\Omega^\bullet_k$
is a map of differential graded $\W_m\Omega^\bullet_k$-modules (see \cite[Thm 2.6]{Kay})
and the fact that $\Tr_{L/K}(1/[t])=w(f)$, see\cite[(3.7.3)]{Kay}.)
\end{no}

\begin{lem}\label{lem:add-mot-coh-MilnorK}
The cycle map $\phi^r_{A_k|(m+1)}: \tau_{\ge r}\Z(r)_{A_k|(m+1)}\to \sK^M_{r,A_k|(m+1)}$
(see Corollary \ref{cor:cycle-maps} in the case where $D_{\rm red}$ is smooth) 
induces an isomorphism 
\eq{lem:add-mot-coh-MilnorK1}{H^{r+1}_\sM(A_k|(m+1), \Z(r))\xr{\simeq} U^1K^M_r(k((T)))/U^{m+1}K^M_r(k((T))),}
for all $r,m\ge 1$.
\end{lem}
\begin{proof}
By Theorem \ref{thm:rel-cycle-iso} the cycle map induces an isomorphism
\[H^{r+1}_{\sM}(A_k|(m+1),\Z(r))\xr{\simeq} H^1(\A^1_k, \sK^M_{r, A_k|(m+1)}).\]
Set $\sQ:=\sK^M_{r,\A^1_k}/\sK^M_{r, A_k|(m+1)}$. We obtain an exact sequence
\[H^0(\A^1_k, \sK^M_{r,\A^1_k})\to H^0(\A^1_k,\sQ)\to H^1(\A^1_k,\sK^M_{r, A_k|(m+1)})
      \to H^1(\A^1_k, \sK^M_{r, \A^1_k}).\]
Now the term on the very right vanishes by homotopy invariance and for the same reason the term on the very
left equals $K^M_r(k)$. Furthermore  $\sQ$ is supported at the closed point $x:=\{0\}\in \A^1_k$
and therefore $H^0(\A^1_k,\sQ)=\sK^M_{r,\A^1_k,x}/\sK^M_{r, A_k|(m+1),x}$. 
We obtain an isomorphism
\eq{lem:add-mot-coh-MilnorK2}{
         H^1(\A^1,\sK^M_{A_k|(m+1)})\cong \sK^M_{r,\A^1_k,x}/(\sK^M_{r, A_k|(m+1),x}+K^M_r(k))}
The statement follows from Lemma \ref{lem-additiveMilnorK-global-local} and the observation
that the right hand side is canonically isomorphic to $\sK^M_{r, A_k|1,x}/\sK^M_{r, A_k|(m+1),x}$.
For the latter it suffices to show that $T\mapsto 0$ induces an isomorphism
$\sK^M_{r,\A^1_k,x}/\sK^M_{r,\A_k|1,x}\xr{\simeq} \sK^M_{r, k}$, which is a special case
of Proposition \ref{prop:gr0-map}.
\end{proof}

\begin{thm}\label{thm:add-Zar-descent}
Let $k$ be a field of characteristic $\neq 2$. 
The following diagram is commutative for all $r,m\ge 1$
\[\xymatrix{ 
 \CH^r(\A^1_k|(m+1)\{0\}, r-1)\ar[d]^\simeq_{(-1)^{r(r-1)/2}\cdot\eqref{DRW-addChow1}}
                                           \ar[r]^-{\rm nat.} & 
                 H^{r+1}_{\sM}(\A^1_k|(m+1)\{0\},\Z(r))\ar[d]_\simeq^{\eqref{lem:add-mot-coh-MilnorK1}}\\
              \W_m\Omega^{r-1}_k\ar[r]_-\simeq^-{\cdot\eqref{thm-DRW-MilnorK1}} & 
                     U^1K^M_r(k((T)))/U^{m+1}K^M_r(k((T))).
            }\]
In particular the natural maps 
\[\CH^r(\A^1_k|(m+1)\{0\}, r-n)\xr{\simeq} H^{r+n}_\sM(\A^1_k|(m+1)\{0\}, \Z(r)),\quad  n \ge 1,\]
are isomorphisms. (Notice that for $n\ge 2$ the left hand side is clearly zero and the right hand side
is zero by Theorem \ref{thm:rel-cycle-iso}.)
\end{thm}
\begin{proof}
We show that the two compositions
\begin{align}
 \alpha: \CH^r(A_k|(m+1), r-1) & \xr{\rm nat.} H^{r+1}(\A^1_k,\Z(r)_{A_k|(m+1)})\notag \\
                             &  \xr{\phi^{1,r}_{A_k|(m+1)}} H^1(\A^1_k,\sK^M_{r,A_k|(m+1)})\notag
\end{align}
and 
\begin{align}
\beta: \CH^r(A_k|(m+1), r-1) & \xr{\eqref{DRW-addChow1}} \W_m\Omega^{r-1}_k \notag\\
       &  \xr{\eqref{thm-DRW-MilnorK1}} U^1K^M_r(k((T)))/U^{m+1}K^M_r(k((T)))=:U^1/U^{m+1}   \notag\\
                  & \xr{{\rm via} \eqref{lem:add-mot-coh-MilnorK2}}  H^1(\A^1_k,\sK^M_{r,A_k|(m+1)})\notag
\end{align}
coincide. By Proposition \ref{prop-big-DRW-presentation}  and \ref{DRW-addChow} 
it suffices to  show
\[\alpha(P)=\beta(P), \quad \alpha(Q)=\beta(Q),\]
 where $P$ and $Q$ are the points defined in \eqref{DRW-addChow2} and \eqref{DRW-addChow3}, respectively. 
In the following we fix the elements $f\in 1+Tk[T]$ and $b_i\in k^\times$ defining $P$ and $Q$.
Using the  Cousin resolution of $\sK^M_{r, A_k|(m+1)}$, see \ref{CousinComplex-relK}, we get 
a surjection
\[H^1_{\{0\}}(\A^1_k, \sK^M_{r,A_k|(m+1)})\oplus \bigoplus_{x\in \A^1\setminus\{0\}} K^M_{r-1}(k(x))
                             \surj H^1(\A^1_k,\sK^M_{r,A_k|(m+1)}).\]
Set $L=k[T]/(f)$ and denote by $t\in L$ the residue class of $T$.
We denote by $\imath_L: K^M_{r-1}(L)\to H^1(\A^1_k,\sK^M_{r,A_k|(m+1)})$ the map induced by the
above surjection. Then by definition of $\phi^r_{A_k|(m+1)}$, 
see \ref{CycleMap-relMilnorK} and \ref{CycleMap-MilnorK}, we have 
\[\alpha(P)=\imath_L(\{b_{r-1},\ldots, b_{1}\}) \quad \text{and}\quad 
\alpha(Q)=\imath_L(\{b_{r-2},\ldots, b_1, \tfrac{1}{t}\}).\]
On the other hand the images of $P$ and $Q$ in $U^1/U^{m+1}$
under the composition $\eqref{thm-DRW-MilnorK1}\circ ((-1)^{r(r-1)/2}\cdot\eqref{DRW-addChow1})$
equal
\[P\mapsto \{b_{r-1},\ldots, b_1,f\}  \text{ mod } U^{m+1}\]
and 
\[Q\mapsto -\{b_{r-2},\ldots, b_1,T,f\} \text{ mod } U^{m+1}.\]
We have to compute the images  of these elements under the connecting homomorphism
\eq{thm:add-Zar-descent1}{U^1/U^{m+1}\to H^1(\A^1_k,\sK^M_{r,A_k|(m+1)}). }
To this end, let $C^\bullet:=\Gamma(\A^1,C^\bullet_{r, \A^1_k})$ and 
$C^\bullet_{m+1}:=\Gamma(\A^1,C^\bullet_{r, A_k|(m+1)})$ 
be the global sections of the Cousin complexes of 
$\sK^M_{r, \A^1_k}$ and $\sK^M_{r, A_k|(m+1)}$, respectively, and $\nu:C^\bullet_{m+1}\to C^\bullet$
the natural map between them. Notice that $C^0=K^M_r(k(t))=C^0_{m+1}$.
Set $D^\bullet={\rm cone}(C^\bullet_{m+1}\to C^\bullet)$, i.e. $D^\bullet$ 
is the complex sitting in degrees $[-1,1]$ 
\[ C^0_{m+1}\xr{d^{-1}}C^0\oplus C^1_{m+1} \xr{d^0} C^1 \]
with $d^{-1}(a)= (a, - d^0_{C_{m+1}}(a))$ and $d^0(b,c)=d^0_{C}(b)+\nu(c)$.
Then $D^\bullet$ is quasi-isomorphic to $U^1/U^{m+1}$ (see after \eqref{lem:add-mot-coh-MilnorK2}).
The boundary map \eqref{thm:add-Zar-descent1}
is given by:  
\begin{enumerate}
\item[1.] Lift an element from $U^1/U^{m+1}$ to $\Ker( d^0)\subset C^0\oplus C^1_{m+1}$.
\item[2.] Apply $-\pi$, with $\pi: C^0\oplus C^1_{m+1}\to C^1_{m+1}$  the projection. 
\item[3.] Consider the class of the resulting element 
modulo the image of $d^0_{C_{m+1}}:C^0_{m+1}\to C^1_{m+1}$.
\end{enumerate}
The boundary $d^0_C$ is given by the tame symbols $\partial_x$ along the various
points  $x\in\A^1_k$. We have 
\[\partial_x(\{b_{r-1},\ldots, b_1,f\})=\begin{cases} \{b_{r-1},\ldots, b_1\}\in K^M_{r-1}(L), & \text{if }x=V(f),\\
                                                                                     0, & \text{else,}
                                                       \end{cases}\]
and 
\[\partial_x(-\{b_{r-1},\ldots, b_1,T,f\})=\begin{cases}
                     \{b_{r-1},\ldots, b_1, \tfrac{1}{t}\}\in K^M_{r-1}(L), & \text{if }x=V(f),\\
                                                                                     0, & \text{else.}
                                                       \end{cases}\]
All together we obtain $\beta(P)=\imath_L(\{b_{r-1},\ldots, b_{1}\})$ and 
$\beta(Q)=\imath_L(\{b_{r-2},\ldots, b_1, \tfrac{1}{t}\})$. This finishes the proof.
\end{proof}

\section{A vanishing result}

\begin{thm}\label{thm:Anvan}
Let $k$ be a field and $X$ a smooth equidimensional $k$-scheme of dimension $d$, $D$ an effective Cartier divisor on $X$ such that $D_{\rm red}$ is a simple normal crossing divisor. For $n\ge 1$ and $\fm=(m_1,\ldots, m_n)\in \N^n$ define the divisor
$E_{\fm}:=\sum_{i=1}^n m_i\cdot q_i^*\{0\}$ on $\A^n_k$ ,
where $q_i:\A^n_k\to \A^1_k$ denotes the projection to the $i$-th factor. 
Denote by $p:X\times \A^n\to X$ the projection map and set $E_{\fm, X}:= X\times E_\fm$.  Then:
\begin{enumerate}
\item\label{thm:Anvan1}
                \[H^{d+r+1}_{\sM}(X\times \A^1| E_{(m+1), X},\Z(r))=
                \begin{cases}  0, & \text{if } m=0,\\ H^d(X,\W_m\Omega^{r-1}_X), &\text{if } m\ge 1.\end{cases}\]
\item\label{thm:Anvan2} For all $n\ge 2$ and all $\fm\in \N^d$, 
                    \[H^{d+r+n}_{\sM,\Nis}(X\times \A^n| (p^*D+E_{\fm, X}), \Z(r))=0.\]
 
\end{enumerate}
\end{thm}
\begin{proof}
By Theorem \ref{thm:rel-cycle-iso}, \eqref{thm:rel-cycle-iso2}, it suffices to prove the corresponding Nisnevich statement of \eqref{thm:Anvan1}. Therefore, 
we will work in the Nisnevich topology and with the Nisnevich sheafification of the relative Milnor $K$-theory for the rest of the proof
and drop the index $\Nis$ everywhere. 
Set
\[Q_{p^*D+E_\fm}:= \sK^M_{r, \A^n_X}/\sK^M_{r, \A^n_X|(p^*D+E_{\fm,X})}.\]
We have
\eq{thm:Anvan3}{H^j(\A^1_X, \sK^M_{r, \A^1_X| E_{1,X}})=0,\quad \text{for all }j.}
Indeed, by Proposition \ref{prop:gr0-map}
$Q_{E_1}\cong i_*\sK^M_{r, X}$, where $i: X\times\{0\}\inj X\times \A^1$ is the closed immersion.
Therefore, the natural map
\[H^j( \A^1_X,\sK^M_{r, \A^1_X})\to H^j(\A^1_X, Q_{E_1})\]
is an isomorphism for all $j$ by homotopy invariance. Hence \eqref{thm:Anvan3}
follows from the long exact cohomology sequence induced by
\[0\to \sK^M_{r, \A^1_X|E_{1,X}}\to \sK^M_{r,\A^1_X}\to Q_{E_1}\to 0.\]
This gives the vanishing for $m=0$ in \eqref{thm:Anvan1}, by Theorem \ref{thm:rel-cycle-iso}. 
By Theorem \ref{thm-DRW-MilnorK} we have an exact sequence
\[0\to \sK^M_{r,\A^1_X|E_{(m+1),X}}\to \sK^M_{r, \A^1_X|E_{1,X}}\to i_* \W_m\Omega^{r-1}_X\to 0.\]
Hence the statement for $m\ge 1$ in \eqref{thm:Anvan1}, follows from Theorem \ref{thm:rel-cycle-iso} and 
\eqref{thm:Anvan3}.

Next we prove \eqref{thm:Anvan2}. Notice that the general case is implied by the case $n=2$.
For $\fm\in\N^2 $ we have an exact sequence
\[H^{d+1}(\sK^M_{r, \A^2_X})\to H^{d+1}(Q_{p^*D+E_\fm})\to H^{d+2}(\sK^M_{r,\A^2_X|(p^*D+E_{\fm,X})})\to
 H^{d+2}(\sK^M_{r,\A^2_X}),\]
where we abbreviate $H^j(\A^2_X, -)$ by $H^j(-)$. Here, the two outer terms vanish by homotopy invariance and the
Nisnevich version of Grothendieck's general vanishing theorem. By Theorem \ref{thm:rel-cycle-iso}, we therefore
have to show 
\[H^{d+1}(\A^2_X, Q_{p^*D+E_\fm})=0.\]
We have an exact sequence 
\[0\longrightarrow \frac{\sK^M_{r,\A^2_X|(p^*D_{\rm red}+E_{\fm, X,{\rm red}})}}{\sK^M_{r,\A^2_X|(p^*D+E_{\fm, X})}}
 \longrightarrow  Q_{p^*D+E_{\fm}}
\longrightarrow \frac{\sK^M_{r, \A^2_X}}{\sK^M_{r,\A^2_X|(p^*D_{\rm red}+E_{\fm, X,{\rm red}})}}
\longrightarrow 0.\]
Thus the statement follows from the two claims:
\eq{thm:Anvan4}{H^{d+1}\left(\A^2_X, 
                  \frac{\sK^M_{r, \A^2_X}}{\sK^M_{r,\A^2_X|(p^*D_{\rm red}+E_{\fm, X,{\rm red}})}}\right)=0.}
\eq{thm:Anvan5}{
H^{d+1}\left(\A^2_X, 
\frac{\sK^M_{r,\A^2_X|(p^*D_{\rm red}+E_{\fm, X,{\rm red}})}}{\sK^M_{r,\A^2_X|(p^*D+E_{\fm, X})}}\right)=0.}

We prove the vanishing \eqref{thm:Anvan4}. We do induction on the number of irreducible
components of $D$. First assume $D=0$. If $\fm=(0,0)$, there is nothing to prove. If $\fm=(1,0)$ or $(0,1)$, then the term
in \eqref{thm:Anvan4} is equal to $H^{d+1}(\A^1_X, \sK^M_{r,\A^1_X})$, by Proposition \ref{prop:gr0-map};
hence it vanishes by homotopy invariance and the Nisnevich version of Grothendieck's general vanishing theorem.
If $\fm=(1,1)$, we have by Proposition \ref{prop:gr0-map} an exact sequence
\[0\to \sK^M_{r,\A^1_X|E_{1,X}}\to \frac{\sK^M_{r,\A^2_X}}{\sK^M_{r,\A^2_X|E_{(1,1),X}}} \to
    \sK^M_{r, \A^1_X}\to 0.\]
Hence the vanishing of $H^{d+1}(\A^2_X,-)$ of the middle part follows from \eqref{thm:Anvan3} 
and homotopy invariance as before.
If $D\neq 0$, let $D_1$ be one of its irreducible components  and write
$D_{\rm red}=D_1+ D'$, where $D'$ is reduced and effective. 
By Proposition \ref{prop:gr0-map}
\eq{thm:Anvan5.5}
{\frac{\sK^M_{r, \A^2_X|p^*D'+E_{\fm, X_{\rm red}}}}{\sK^M_{r,\A^2_X|(p^*D_{\rm red}+E_{\fm, X,{\rm red}})}}
\cong i_{1*}\left( 
\frac{\sK^M_{r, \A^2_{D_1}}}{\sK^M_{r,\A^2_{D_1}|(p^*(D'\cap D_1)_{\rm red}+E_{\fm, D_1,{\rm red}})}}\right),}
where $i_1: \A^2_{D_1}\inj A^2_{X}$ is the closed immersion. 
We have $H^{d+1}(\A^2_{D_1},\sK^M_{r, \A^2_{D_1}})=0$ by homotopy invariance and
$H^{d+2}(\A^2_{D_1}, \sK^M_{r,\A^2_{D_1}|(p^*(D'\cap D_1)_{\rm red}+E_{\fm, D_1,{\rm red}})})=0$ for
dimension reasons. This implies the vanishing $H^{d+1}(\A^2_X, \eqref{thm:Anvan5.5})=0$.
Hence we are reduced to prove the vanishing \eqref{thm:Anvan4} with $D$ replaced by $D'$. We conclude
by induction.
 
We prove the vanishing \eqref{thm:Anvan5}. Consider the sheaf
\[\omega^{r-1}_{\fn,\nu}:= \omega^{r-1}_{\A^2_X|(p^*D+E_{\fm,X}),\fn,\nu},\]
with the notations from \ref{restricted-log-differentials} and define $B^{r-1}_{s+1,\fn,\nu}$
as in \ref{higher-Bs-Zs}, with $s=0$, in case $k$ has characteristic $0$.
If $(p^*D+E_{\fm,X})_\nu$ is one of the irreducible components of $p^*D$, set $X_\nu:=D_\nu\times \A^1$, 
if it is an irreducible component of $E_{\fm,X}$ set $X_\nu:= X$. Then $\omega^{r-1}_{\fn,\nu}$ is a 
locally free sheaf on $X_\nu\times \A^1$ and
$B^{r-1}_{s+1,\fn,\nu}$ is a subsheaf.
By Proposition \ref{prop:higher-gr-map-char0} and Theorem \ref{thm:higher-gr-map-poschar}
the sheaf 
$\sK^M_{r,\A^2_X|(p^*D_{\rm red}+E_{\fm, X,{\rm red}})}/\sK^M_{r,\A^2_X|(p^*D+E_{\fm, X})}$
 is a successive extension of the sheaves
$\omega^{r-1}_{\fn,\nu}/B^{r-1}_{s+1,\fn,\nu}$, for certain $s,\fn,\nu$. 
Since $H^{d+2}(X_\nu\times\A^1, B^{r-1}_{s+1,\fn,\nu})=0$, for dimension reasons, it suffices to show
\eq{thm:Anvan9}{H^{d+1}(X_\nu\times\A^1, \omega^{r-1}_{\fn,\nu})=0.}
Denote by $a: X_\nu\times\A^1\to X_\nu$ and by $b: X_\nu\times\A^1\to \A^1$ the projection maps.
Since $\Omega^1_{\A^1}(\log \{0\})(-m\cdot\{0\})\cong\sO_{\A^1}$, 
it follows directly from the definition of $\omega^{r-1}_{\fn,\nu}$ in \ref{restricted-log-differentials} that there exist
locally free sheaves $\omega^{r-1}$ and $\omega^{r-2}$ on $X_\nu$, possibly of rank $0$, such that
  \[\omega^{r-1}_{\fn,\nu}\cong a^*\omega^{r-1}\oplus a^*\omega^{r-2}.\]
We have for $i=r-1, r-2$
\[H^{d+1}(X_\nu\times \A^1,a^*\omega^i)= H^0(\A^1, R^{d+1}b_* (a^*\omega^i))
= k[t]\otimes_k H^{d+1}(X_\nu,\omega^i)=0,\]
where the first equality follows from the Leray spectral sequence, the second from flat base change and the vanishing holds
for dimension reasons. This yields the vanishing \eqref{thm:Anvan9} and finishes the proof.
\end{proof}

\begin{rmk}
Let $X$ be an equidimensional $k$-scheme of dimension $d$ and $D$ an effective Cartier divisor on $X$.
By \cite[Thm 5.11]{KP15} we have the vanishing 
$\CH^r(X\times \A^n|(p^*D+E_{\fm}), r-(d+n))=0$,  for all $r$, all $n\ge 2$, and all $\fm\in (\N_{\ge 1})^n$. In particular,
if the assumptions of Theorem \ref{thm:Anvan} are satisfied,  the natural map
\[\CH^r(X\times \A^n|(p^*D+E_{\fm}), r-(d+n))\to H^{r+d+n}_{\sM, \Nis}(X\times \A^n| (p^*D+E_{\fm}),\Z(r))\]
is bijective. 
\end{rmk}

\bibliographystyle{amsalphacustomlabels}
\bibliography{MilnorChowModulus}

\providecommand{\bysame}{\leavevmode\hbox to3em{\hrulefill}\thinspace}
\providecommand{\MR}{\relax\ifhmode\unskip\space\fi MR }
\providecommand{\MRhref}[2]{%
  \href{http://www.ams.org/mathscinet-getitem?mr=#1}{#2}
}
\providecommand{\href}[2]{#2}
\begin{thebibliography}{EGA IV2}

\bibitem[BE03a]{BE}
Spencer Bloch and H\'el\`ene Esnault, \emph{The additive dilogarithm},
  Documenta Mathematica (2003), no.~Extra Vol., 131--155 (electronic), Kazuya
  Kato's fiftieth birthday.

\bibitem[BE03b]{BE2}
Spencer Bloch and H{\'e}l{\`e}ne Esnault, \emph{An additive version of higher
  {C}how groups}, Ann. Sci. \'Ecole Norm. Sup. (4) \textbf{36} (2003), no.~3,
  463--477.

\bibitem[BK86]{BK}
Spencer Bloch and Kazuya Kato, \emph{$p$-adic \'etale cohomology}, Institut des
  Hautes \'Etudes Scientifiques. Publications Math\'ematiques (1986), no.~63,
  107--152.

\bibitem[Blo77]{BlochDRW}
Spencer Bloch, \emph{Algebraic {$K$}-theory and crystalline cohomology}, Inst.
  Hautes \'Etudes Sci. Publ. Math. (1977), no.~47, 187--268 (1978).

\bibitem[Blo86]{Bloch-HigherChow}
\bysame, \emph{Algebraic cycles and higher {$K$}-theory}, Adv. in Math.
  \textbf{61} (1986), no.~3, 267--304.

\bibitem[BO74]{BO}
Spencer Bloch and Arthur Ogus, \emph{Gersten's conjecture and the homology of
  schemes}, Ann. Sci. \'Ecole Norm. Sup. (4) \textbf{7} (1974), 181--201
  (1975).

\bibitem[BS14]{BS14}
Federico Binda and Shuji Saito, \emph{Relative cycles with moduli and regulator
  maps}, Preprint, \url{http://arxiv.org/abs/1412.0385}, 2014.

\bibitem[BT73]{BT}
H.~Bass and J.~Tate, \emph{The {M}ilnor ring of a global field}, Algebraic
  {$K$-theory}, {II:} {“Classical”} algebraic {$K$-theory} and connections
  with arithmetic {(Proc.} Conf., Seattle, Wash., Battelle Memorial Inst.,
  1972), Springer, Berlin, 1973, pp.~349--446. Lecture Notes in Math., Vol.
  342.

\bibitem[EGA III2]{EGAIII2}
A.~Grothendieck, \emph{\'{E}l\'ements de g\'eom\'etrie alg\'ebrique. {III}.
  \'{E}tude cohomologique des faisceaux coh\'erents. {II}}, Inst. Hautes
  \'Etudes Sci. Publ. Math. (1963), no.~17, 91.

\bibitem[EGA IV1]{EGAIV1}
\bysame, \emph{\'{E}l\'ements de g\'eom\'etrie alg\'ebrique. {IV}. \'{E}tude
  locale des sch\'emas et des morphismes de sch\'emas. {I}}, Inst. Hautes
  \'Etudes Sci. Publ. Math. (1964), no.~20, 259.

\bibitem[EGA IV2]{EGAIV2}
\bysame, \emph{{\'E}l\'ements de g\'eom\'etrie alg\'ebrique. {IV.} \'{E}tude
  locale des sch\'emas et des morphismes de sch\'emas. {II}}, Institut des
  Hautes \'Etudes Scientifiques. Publications Math\'ematiques (1965), no.~24,
  231.

\bibitem[Ful98]{F}
William Fulton, \emph{Intersection theory}, second ed., Ergebnisse der
  Mathematik und ihrer Grenzgebiete. 3. Folge. A Series of Modern Surveys in
  Mathematics, vol.~2, Springer-Verlag, Berlin, 1998.

\bibitem[Har66]{Ha}
Robin Hartshorne, \emph{Residues and duality}, Lecture notes of a seminar on
  the work of A. Grothendieck, given at Harvard 1963/64. With an appendix by P.
  Deligne. Lecture Notes in Mathematics, No. 20, Springer-Verlag, Berlin-New
  York, 1966.

\bibitem[Hes15]{He}
Lars Hesselholt, \emph{The big de {R}ham--{W}itt complex}, Acta Math.
  \textbf{214} (2015), no.~1, 135--207.

\bibitem[HK94]{HK}
Osamu Hyodo and Kazuya Kato, \emph{Semi-stable reduction and crystalline
  cohomology with logarithmic poles}, Ast\'erisque (1994), no.~223, 221--268,
  P{\'e}riodes $p$-adiques (Bures-sur-Yvette, 1988).

\bibitem[HM01]{HeMa01}
Lars Hesselholt and Ib~Madsen, \emph{On the {$K$}-theory of nilpotent
  endomorphisms}, Homotopy methods in algebraic topology ({B}oulder, {CO},
  1999), Contemp. Math., vol. 271, Amer. Math. Soc., Providence, RI, 2001,
  pp.~127--140.

\bibitem[Ill79]{IlDRW}
Luc Illusie, \emph{Complexe de de {R}ham-{W}itt et cohomologie cristalline},
  Annales Scientifiques de {l'\'Ecole} Normale Sup\'erieure. Quatri\`eme
  S\'erie \textbf{12} (1979), no.~4, 501--661.

\bibitem[Ill90]{IlOrd}
Luc Illusie, \emph{Ordinarit\'e des intersections compl\`etes g\'en\'erales},
  The {G}rothendieck {F}estschrift, {V}ol.\ {II}, Progr. Math., vol.~87,
  Birkh\"auser Boston, Boston, MA, 1990, pp.~376--405.

\bibitem[Kat70]{KatzNil}
Nicholas~M. Katz, \emph{Nilpotent connections and the monodromy theorem:
  {A}pplications of a result of {T}urrittin}, Inst. Hautes \'Etudes Sci. Publ.
  Math. (1970), no.~39, 175--232.

\bibitem[Kat83]{Kato1983}
Kazuya Kato, \emph{Residue homomorphisms in {M}ilnor {$K$}-theory}, Galois
  groups and their representations ({N}agoya, 1981), Adv. Stud. Pure Math.,
  vol.~2, North-Holland, Amsterdam, 1983, pp.~153--172.

\bibitem[Kat89]{Kato-log}
\bysame, \emph{Logarithmic structures of {F}ontaine-{I}llusie}, Algebraic
  analysis, geometry, and number theory ({B}altimore, {MD}, 1988), Johns
  Hopkins Univ. Press, Baltimore, MD, 1989, pp.~191--224.

\bibitem[Ker09]{Kerz09}
Moritz Kerz, \emph{The {G}ersten conjecture for {M}ilnor {$K$-theory}},
  Inventiones Mathematicae \textbf{175} (2009), no.~1, 1--33.

\bibitem[Ker10]{Kerz10}
\bysame, \emph{Milnor {$K$}-theory of local rings with finite residue fields},
  J. Algebraic Geom. \textbf{19} (2010), no.~1, 173--191.

\bibitem[KeS14]{KeS}
Moritz Kerz and Shuji Saito, \emph{{C}how group of 0-cycles with modulus and
  higher dimensional class field theory}, Preprint,
  \url{arxiv.org/pdf/1304.4400v2}, 2014.

\bibitem[KP12]{KP}
Amalendu Krishna and Jinhyun Park, \emph{Moving lemma for additive higher
  {C}how groups}, Algebra Number Theory \textbf{6} (2012), no.~2, 293--326.

\bibitem[KP15]{KP15}
\bysame, \emph{A module structure and a vanishing theorem for cycles with
  modulus}, Preprint, \url{https://arxiv.org/pdf/1412.7396}, 2015.

\bibitem[KS86]{KS}
Kazuya Kato and Shuji Saito, \emph{Global class field theory of arithmetic
  schemes}, Applications of algebraic {$K$}-theory to algebraic geometry and
  number theory, {P}art {I}, {II} ({B}oulder, {C}olo., 1983), Contemp. Math.,
  vol.~55, Amer. Math. Soc., Providence, RI, 1986, pp.~255--331.

\bibitem[KSS]{Kato-Saito-Sato}
Kazuya Kato, Shuji Saito, and Kanetomo Sato, \emph{{$p$}-adic vanishing cycles
  and {$p$}-adic {\'e}tale {T}ate twists on generalized semistable families},
  Preprint 2014.

\bibitem[Lev09]{Le09}
Marc Levine, \emph{Smooth motives}, Motives and algebraic cycles, Fields Inst.
  Commun., vol.~56, Amer. Math. Soc., Providence, RI, 2009, pp.~175--231.

\bibitem[Mil80]{MilneEt}
James~S. Milne, \emph{\'{E}tale cohomology}, Princeton Mathematical Series,
  vol.~33, Princeton University Press, Princeton, N.J., 1980.

\bibitem[Mor12]{Morrow12}
Matthew Morrow, \emph{Continuity of the norm map on {M}ilnor {$K$}-theory}, J.
  K-Theory \textbf{9} (2012), no.~3, 565--577.

\bibitem[Nis89]{Nis}
Ye.~A. Nisnevich, \emph{The completely decomposed topology on schemes and
  associated descent spectral sequences in algebraic {$K$}-theory}, Algebraic
  {$K$}-theory: connections with geometry and topology ({L}ake {L}ouise, {AB},
  1987), NATO Adv. Sci. Inst. Ser. C Math. Phys. Sci., vol. 279, Kluwer Acad.
  Publ., Dordrecht, 1989, pp.~241--342.

\bibitem[Par09]{Pa}
Jinhyun Park, \emph{Regulators on additive higher {C}how groups}, Amer. J.
  Math. \textbf{131} (2009), no.~1, 257--276.

\bibitem[Pop86]{Popescu86}
Dorin Popescu, \emph{General {N}\'eron desingularization and approximation},
  Nagoya Math. J. \textbf{104} (1986), 85--115.

\bibitem[Ros96]{Rost}
Markus Rost, \emph{Chow groups with coefficients}, Documenta Mathematica
  \textbf{1} (1996), No. 16, 319--393 (electronic).

\bibitem[R{\"u}l07]{Kay}
Kay R{\"u}lling, \emph{The generalized de {R}ham-{W}itt complex over a field is
  a complex of zero-cycles}, Journal of Algebraic Geometry \textbf{16} (2007),
  no.~1, 109--169.

\bibitem[SGA2]{SGA2}
Alexander Grothendieck, \emph{Cohomologie locale des faisceaux coh\'erents et
  th\'eor\`emes de {L}efschetz locaux et globaux ({SGA 2})}, {North-Holland}
  Publishing Co., Amsterdam, 1968, Advanced Studies in Pure Mathematics, Vol.
  2.

\bibitem[SGA4II]{SGA4II}
\emph{Th\'eorie des topos et cohomologie \'etale des sch\'emas. {T}ome 2},
  Lecture Notes in Mathematics, Vol. 270, Springer-Verlag, Berlin-New York,
  1972, S{\'e}minaire de G{\'e}om{\'e}trie Alg{\'e}brique du Bois-Marie
  1963--1964 (SGA 4), Dirig{\'e} par M. Artin, A. Grothendieck et J. L.
  Verdier. Avec la collaboration de N. Bourbaki, P. Deligne et B. Saint-Donat.

\bibitem[SGA7II]{SGA7II}
\emph{Groupes de monodromie en g\'eom\'etrie alg\'ebrique. {II}}, Lecture Notes
  in Mathematics, Vol. 340, Springer-Verlag, Berlin-New York, 1973,
  S{\'e}minaire de G{\'e}om{\'e}trie Alg{\'e}brique du Bois-Marie 1967--1969
  (SGA 7 II), Dirig{\'e} par P. Deligne et N. Katz.

\bibitem[Tohoku]{Tohoku}
A.~{Grothendieck}, \emph{{Sur quelques points d'alg\`ebre homologique.}},
  {Tohoku Math. J. (2)} \textbf{9} (1957), 119--221.

\bibitem[Tot92]{Totaro}
Burt Totaro, \emph{Milnor {$K$}-theory is the simplest part of algebraic
  {$K$}-theory}, $K$-Theory \textbf{6} (1992), no.~2, 177--189.

\bibitem[Voe00a]{VoPST}
Vladimir Voevodsky, \emph{Cohomological theory of presheaves with transfers},
  Cycles, transfers, and motivic homology theories, Ann. of Math. Stud., vol.
  143, Princeton Univ. Press, Princeton, {NJ}, 2000, pp.~87--137.

\bibitem[Voe00b]{VoDM}
\bysame, \emph{Triangulated categories of motives over a field}, Cycles,
  transfers, and motivic homology theories, Ann. of Math. Stud., vol. 143,
  Princeton Univ. Press, Princeton, {NJ}, 2000, pp.~188--238.

\end{thebibliography}

\end{document}